\documentclass[11pt,twoside]{atmp}

\usepackage{amsmath,amssymb}

\usepackage[all]{xy}
\usepackage{color}

\newtheorem{theorem}{Theorem}[section]

\begin{document}

\title[Exact math-models of micro universes]
{Exact mathematical models of a unified quantum theory; 
Static and expanding micro universes}

%%\arxurl{} %arXiv reference #, url appears as foot note on first page
\thanks{Partially supported by NSF grant DMS-0604861}

\author[Z. I. Szab\'o]{Zolt\'an Imre Szab\'o}

\address{City University of New York and R\'enyi Institute of Mathematics 
Budapest}  
%lines should be separated with double backslashes: \\
%\addressemail{zoltan.szabo@lehman.cuny.edu}

\begin{abstract}
In this paper such Riemann metrics are established 
whose Laplace-Beltrami operators are identical to familiar 
Hamilton operators of elementary particle systems.
Such metrics are the natural positive definite invariant 
metrics defined on two-step nilpotent Lie groups.
The corresponding wave and Schr\"odinger operators 
emerge in the Laplacians both of the static and 
solvable extensions of these nilpotent groups.
The latter manifolds are endowed with natural 
invariant indefinite metric of 
Lorentz signature. Thus, these new exact 
mathematical models 
provide a relativistic theory for elementary particles. 

This theory establishes infinitely many
non-equivalent models for which even classification is possible. The
particle systems attached to them behave exactly like
their relatives introduced by the familiar standard model of elementary
particle physics.
Although there are strong connections between the two theories, the 
new mathematical models were independently discovered. 
They appeared, originally, in those works of this author where 
isospectral Riemann manifolds with different local geometries have been
constructed. Their strong connection to quantum theory,
which seems to be completely unknown in the literature, 
was realized by some years later. 

One of the most important new features of this new theory is that  
the electromagnetic, the
weak-nuclear, and the strong-nuclear forces emerge in a unified way. 
The main unifying idea is that
these forces can be described by the eigenfunctions of the very same
Laplacian such that the distinct forces emerge on 
distinct invariant 
subspaces of this common quantum operator. There are also other 
bonds realized which make the connection among these
three fundamental forces much more strong. The missing fourth 
fundamental force, gravitation, is not discussed in this paper. 
 
On the solvable extensions the new models look like
Friedmann's expanding universe being adopted to the microscopic level. 
Like the macroscopic one, also these microscopic
expanding models obey Hubble's law. The microscopic models, however,
offer much more complex structures with much more subtle
explanations for some of those phenomena which have originally been 
clarified by the Friedmann model. For instance, 
it is rigorously established in this paper that, 
although the Riemann spaces 
in the mathematical models are not spatially isotropic in general, yet
they are always spectrally isotropic. It follows that the
radiation, which is experimentally known in cosmology, must also be isotropic. 
This statement actually contrasts the widely accepted view
that the isotropic cosmic radiation is the    
chief verification of the spatial-isotropy assumed by 
Friedmann in building up his theory.  

\end{abstract}

\maketitle

\section{Introduction} 
%%\cutpage 
%move this line so that the first page breaks at the appropriate place.

%%\setcounter{page}{<insert page # for second page>}

By assuming the metric in the form
    $ds^2 = {a(t)}^2 ds_3^2 - dt^2$, 
Friedmann (1922) built up his relativistic 
universe on two 
assumptions: the universe looks identical in whichever direction we look,
and this would also be true if we were observing it anywhere else. By
the two differential equations imposed by the
Einstein equation upon the function 
$a(t)$, he
showed then that, instead of being static, the universe is expanding. 
Without knowing Friedmann's theoretical prediction, 
this phenomenon was actually 
discovered by Edwin Hubble (1929) several years later.

This phenomenon is pointed out in 
this paper in quite a different situation.
Independently from Friedmann's theory, this one has been evolved
from those works of the author where isospectral manifolds with
different local geometries were constructed on two different types of
manifolds, called Z-torus bundles (alias Z-crystals) 
and Z-ball bundles, respectively \cite{sz1}-\cite{sz4}. 
These bundles are constructed by means of  
nilpotent Lie groups as well as their solvable 
extensions such that one considers tori resp. balls in the center 
(called also
Z-space) of the nilpotent group. Both the nilpotent and extended groups 
are endowed with appropriate natural invariant metrics. 

Surprisingly enough, the Laplacians on the nilpotent groups
(endowed always with invariant positive definite Riemann metrics) 
are nothing but
the familiar Hamilton operators corresponding to 
particle-antiparticle systems. On the two types of manifolds, the
represented particles can be distinguished as follows. On Z-crystals, the
Laplacian represents particles having no interior, where it
actually appears as a Ginsburg-Landau-Zeeman operator of a system 
of electrons, positrons, and electron-positron-neutrinos. 
The particles represented by Z-ball bundles do have 
interior and the Laplacian decomposes into an exterior  
Ginsburg-Landau-Zeeman operator and an
interior spin operator by which the weak-force and the strong-force 
interactions can be described, respectively. These nuclear forces are
very different from the electromagnetic force emerging in the Laplacian
of Z-crystals. The weak nuclear force
explains the beta decay, while the strong force keeps the
parts of atomic nuclei together. Yet, the Z-crystal-Laplacian and the 
weak-force-Laplacian can be led back to the very same radial 
Ginsburg-Landau-Zeeman operator. This phenomenon is consistent with
the Weinberg-Salam theory of beta decays, which unified the weak force
with the electromagnetic force.   

There are two ways to introduce relativistic time on these models.
The static model is constructed by the Cartesian product of the nilpotent
group with $\mathbb R$. The latter component becomes the time axis 
regarding the natural Lorenz-indefinite metric. 
According to the type of model being extended, 
the Laplacian is the sum of Schr\"odinger and electron-positron-neutrino,
resp., weak-nuclear and strong-nuclear 
wave operators. The last two wave operators appear for particles having 
interior. They are further decomposed into W- and Z-operators
which are analogous to the electron-positron and 
electron-positron-neutrino wave operators.
  
Relativistic time can be introduced also by solvable extension, 
which also increases the
dimension of the nilpotent group by $1$. The new axis, 
which is just a half-line $\mathbb R_+$, can also be used 
as time-axis for introducing a natural
relativistic metric on these extensions. 
Contrary to the static case, in this way 
one defines expanding models
obeying Hubble's law, furthermore, the Laplacian decomposes into expanding
Schr\"odinger and electron-positron-neutrino, resp., weak and strong wave 
operators and the corresponding W- and Z-operators.

It is a well known experimental fact that, even though the universe is
expanding, there is no expansion measured on small scale level. 
Thus the question arises if the expanding solvable
models describe real existing microscopic
world. Fortunately this question can positively be answered. In fact, 
despite of the expansion, the particles must not be expanding. 
The reason explaining this paradoxical phenomenon is that,
defined by the angular momentum and spin operators, 
also these mathematical models correspond constant magnetic fields to the
particles, and, due to the expansion, the change of these fields induces  
electromagnetic fields, which are completely radiated out from the system,
keeping both the magnetic fields and the spectra of the 
particle-systems constant.  
There is also mathematically established in this paper that this 
radiation must be isotrpic, meaning that it is the
same whichever direction is measured from. Thus the size and 
several other constants of particles must not be changing even according
to the expanding solvable model. Actually this model 
gives a new explanation 
for the presence of an isotropic radiation in the space.

Expanding on large scale but being stationary on small
scale is a well known phenomenon, which, without the 
above explanation, could have been a major argument against 
the physical reality of the solvable extensions. 
According to physical experiments, although far
distant clusters of galaxies move very rapidly away from us,
the solar system is not expanding, nor is our galaxy or the cluster
of galaxies to which it belongs. This stagnancy is even more apparent
on microscopic level, where, for instance, the spectroscopic 
investigations
of the light arriving from far distance galaxies confirm that 
the spectrum of hydrogen atom is the same today than it was macro 
billions of years ago.
   
The existence of isotropic background radiation is also well known which 
was measured, first, by Arno Penzias and Robert Wilson, in 1965. 
It is believed, today, that these radiations are travelling to us 
across most of the observable universe, thus the radiation isotropy
proves that the universe must be the same in every direction, if only
on a large scale. This phenomenon is considered as a remarkably
accurate confirmation of Friedmann's first assumption. Our expanding
model provides a much more subtle conclusion, however: This background
radiation must be isotropic even if it arrives to us from 
very near distances. Moreover, it holds true also on non-isotropic spaces.

This exact mathematical model is not derived from the standard
model of elementary particles,
which theory is based on a non-Abelian gauge theory where 
the basic objects are 
Yang-Mills connections
defined on principal fibre bundles having structure group $SU(3)$. 
Contrary to these gauge theories, in our case
all physical quantities are defined by invariant Riemann metrics 
living on 
nilpotent resp. solvable groups. It will also be pointed out that 
no regular gauge-group exist on these models regarding of which these
objects are gauge invariant. 
Yet, there is a bridge built up between
the two theories, which explains why the particles introduced by the 
two distinct models exhibit the very same physical features. This bridge
can be regarded as a correspondence principle associating certain Riemann
metrics to the Yang-Mills models of elementary particles. 

The key point about the new exact mathematical model is that 
the center of the nilpotent group
makes room to describe also the rich ``inner life" of particles, which is 
known both experimentally and by theories 
explaining these experimental facts. 
This ``inner life" is displayed by the de Broglie waves 
which appear in a new form in this new situation such that they 
are written up in terms of the Fourier transforms
performed only on the center  of the nilpotent group. By this reason,
they are called Z-Fourier transforms, which are defined on the two types
of models accordingly. On Z-crystals, where there is no ``inner life",
it is nothing but the discrete 
Z-Fourier transform defined by the Z-lattice
by which the Z-crystals are introduced. On Z-ball bundles, however, 
in order to
obey the boundary conditions, more complicated so called twisted 
Z-Fourier transforms are introduced. 

The action of the very same Laplacian
appears quite differently on these different
function spaces. On Z-crystals, where there are no Z-boundary conditions
involved, the strong nuclear forces do not appear either. In this case, 
the eigenfunctions arise as eigenfunctions of 
Ginsburg-Landau-Zeeman operators. By this reason, they are called 
electromagnetic eigenfunctions. On Z-crystals, the theory corresponds to 
quantum electrodynamics (QED), while on Z-ball bundles 
it relates to quantum chromodynamics (QCD). 

In fact, on Z-ball bundles, due to the the Z-boundary conditions,
the Laplacian appears in a much more complex form exhibiting both the 
weak and strong forces. More precisely, the weak
force eigenfunctions satisfying a 
given boundary condition are defined by the
eigenfunctions of the exterior Ginsburg-Landau-Zeeman operator 
introduced above
for particles having interior. Although there are numerous 
differences between this exterior Ginsburg-Landau-Zeeman operator 
and the original
GLZ-operator defined on Z-crystals, they are both reduced to 
the very same radial operator acting on radial functions. As a result, from 
the point of view of the elements of the spectrum, 
they are the same operators. 
This is the mathematical certification of
the Weinberg-Salam theory which unified the weak interaction with the
electromagnetic force. The strong force eigenfunctions are defined 
by the eigenfunctions of the inner spin operator. All these forces
reveal the very same strange properties which are described in QCD.
This is how a unified theory for the three:
1.) electromagnetic, 2.) weak, and 3.) strong  nuclear forces is 
established in this paper. The only
elementary force missing from this list is the gravitation, 
which is not discussed in this paper.

This very complex physical-mathematical theory can clearly be evolved 
just gradually. In order
to understand the physical contents of the basic objects 
appearing in new forms in this new approach, first, those parts
of the classical quantum theory are reviewed which are necessary 
to grasp these renewed versions of these basic concepts. Then, after
introducing the basic mathematical objects on 2-step nilpotent Lie groups,
several versions of the Z-Fourier transform will be studied. They are the
basic tools both for introducing the de Broglie waves in a new explicit 
form and developing the theory unifying the three fundamental forces.
Besides explicit eigenfunction computations,
there is pointed out in this part that the Laplacian on Z-crystals 
is nothing but the Ginsburg-Landau-Zeeman operator of a system of 
electrons, positrons, and electron-positron-neutrinos. Furthermore,
on the Z-ball models, it is the sum of the exterior 
Ginsburg-Landau-Zeeman operator
and the interior spin operator by which the strong force interaction
can be established. 

Then, relativistic time is introduced and 
both static and expanding models are established. The Laplacian on these
space-time manifolds appears as the sum of wave operators belonging to 
the particles the system consists of. 
The paper is concluded by pointing out the spectral
isotropy in the most general situations. Since the Riemann metrics
attached to the particle systems are  
not isotropic in general, this statement points out a major 
difference between our
model and Friedmann's cosmological model where the isotropy of the
space is one of his two assumptions. Our statement says that radiation
isotropy holds true also on non-isotropic spaces and the two 
isotropy-concepts are by no means equivalent.

\section{Basics of classical quantum theory.}

In this sections three topics of classical 
quantum theory are reviewed. The first one
describes the elements of de Broglie's theory associating waves to 
particles. The second resp. third ones are surveys on meson theory resp. 
Ginsburg-Landau-Zeeman and Schr\"odinger operators of charged particles. 
\subsection{Wave-particle association.}

In quantum theory, a particle with energy $E$ and momentum $\mathbf p$
is associated with a wave, 
$Ae^{\mathbf i(K\cdot Z-\omega t)}$, where
$K=(2\pi/\lambda )\mathbf n$ is the wave vector and
$\mathbf n$ is the wave normal. These quantities yield the following
relativistically invariant relations. 

For light quanta the most familiar relations are
\begin{equation}
\label{e=hnu}
E=\hbar\omega \quad ,\quad \mathbf p=\hbar K,
\end{equation}
where the length of the wave vector yields also the following equations:
\begin{equation}
\label{p=e/c}
\mathbf k=|K|=\frac{\omega}{c},\quad \mathbf k^2=\frac{\omega^2}{c^2},
\quad {\rm and}\quad 
|\mathbf p|=p=\frac{E}{c},\quad \mathbf p^2=\frac{E^2}{c^2}.
\end{equation}

For a material particle of rest mass $m$, the fundamental relation is 
 \begin{equation}
\label{e(m)}
\frac{E}{c}=\sqrt{ \mathbf p^2+m^2c^2},
\end{equation}
which can be established by the well known equations
 \begin{equation}
E=\frac{mc^2}{\sqrt{1-v^2/c^2}},\quad\quad \mathbf p=
\frac{m\mathbf v}{\sqrt{1-v^2/c^2}}
\end{equation}
of relativistic particle mechanics. 

The idea of de Broglie was that
(\ref{e=hnu}) should also be valid for a material particle such that 
(\ref{p=e/c}) must be replaced by
\begin{equation}
\label{p(m)=e/c}
\sqrt{\mathbf k^2+\frac{m^2c^2}{\hbar^2}}=\frac{\omega}{c},\quad\quad 
{\mathbf k^2+\frac{m^2c^2}{\hbar^2}}=\frac{\omega^2}{c^2}.
\end{equation}
In this general setting, $m=0$ corresponds to the light.

In wave mechanics, de Broglie's most general wave packets are 
represented by the Fourier integral formula:
\begin{equation}
\label{wavepack}
\psi (Z,t)=\int\int\int A(K_1,K_2,K_3)
e^{\mathbf i(\langle K,Z\rangle-\omega t)}dK_1dK_2dK_3,
\end{equation}
where $\omega$ is given by (\ref{p(m)=e/c}). 
In other words, a general wave appears as superposition of
the above plane waves. Instead of the familiar $X$, the vectors from the
3-space, $\mathbb R^3$, are denoted here by $Z$, indicating that the 
reformulated de Broglie waves will be introduced in the new theory
in terms of the so called
twisted Z-Fourier transform, which is performed just on the
center (alias Z-space) of the nilpotent group. The X-space of a
nilpotent group is complement to the Z-space and the integration
in the formula of twisted Z-Fourier transform does not apply to the 
X-variable. It applies just to the Z-variable. The above denotation is
intended to help to understand the renewed de Broglie waves more easily.
 
The above wave function, $\psi$, 
satisfies the relativistic scalar wave equation:
\begin{equation}
\label{waveeq}
\big(\nabla^2-\frac{1}{c^2}\frac{\partial^2}{\partial t^2}\big)
\psi (Z,t)=\frac{m^2c^2}{\hbar^2}\psi (Z,t),
\end{equation}
which statement can be seen by substituting (\ref{wavepack}) 
into this equation.
Then (\ref{p(m)=e/c}) implies (\ref{waveeq}), indeed. According to this 
equation, the wave function is an eigenfunction of the wave operator with
eigenvalue ${m^2c^2}/{\hbar^2}$. By this observation we get that 
the spectrum of the
wave operator is continuous and the multiplicity of each eigenvalue
is infinity. 

The Fourier integral formula 
(\ref{wavepack}) converts differential operators to multiplication
operators. Namely we have:
\begin{equation}
\frac{\partial}{\partial Z_j}\sim \mathbf iK_j\quad ,\quad 
\frac{\partial}{\partial t}\sim \mathbf i\omega .
\end{equation}
These correspondences together with (\ref{e=hnu}) yield the 
translational key:
\begin{equation}
\label{correl}
-\mathbf i\hbar\frac{\partial}{\partial Z_j}\sim p_j\quad ,\quad 
\mathbf i\hbar \frac{\partial}{\partial t}\sim E 
\end{equation}
between the classical quantities $\mathbf p$ and $E$ of classical mechanics
and the operators of wave mechanics.

In his lectures on physics \cite{p} (Vol. 5, Wave mechanics, pages 3-4), 
Pauli describes the
transition from the above relativistic theory to the non-relativistic
approximation as follows. In mechanics, for $v\,<<\,c$ and  $p\,<<\,mc$,
we have
 \begin{equation}
\label{pnonrel_1}
\frac{E}{c}=\sqrt{ \mathbf p^2+m^2c^2}\sim mc(1+\frac{1}{2}
\frac{p^2}{m^2c^2}+\dots )=\frac{1}{c}(mc^2+
\frac{1}{2}\frac{p^2}{m}+\dots ).
\end{equation}
From (\ref{p(m)=e/c}) we also obtain 
 \begin{equation}
\label{pnonrel_2}
\omega =\frac{E}{\hbar}=\frac{mc^2}{\hbar}+\frac{\hbar}{2m}k^2+\dots ,
\end{equation}
where $E=mc^2+E_{kin}$ and $E_{kin}=p^2/2m$. The non-relativistic wave
\begin{equation}
\label{nonrel_wavepack}
\tilde \psi (Z,t)=\int\int\int A(K_1,K_2,K_3)
e^{\mathbf i(\langle K,Z\rangle -\tilde \omega t)}dK_1dK_2dK_3,
\end{equation}
is defined in terms of
\begin{equation}
\label{pnonrel_3}
\tilde\omega =\frac{\hbar}{2m}\mathbf k^2=\omega -\frac{mc^2}{\hbar},
\end{equation}
which relates to the relativistic wave function by the formula:
\begin{equation}
\label{pnonrel_4}
\psi (Z,t)=e^{-\frac{\mathbf imc^2}
{\hbar}t} \tilde \psi (Z,t).
\end{equation}
Substitution into (\ref{waveeq}) yields then:
\begin{equation}
\label{pnonrel_5}
\nabla^2 \tilde \psi +\frac{m^2c^2}{\hbar^2}
\tilde \psi + 2\frac{\mathbf im}{\hbar}
\frac{\partial\tilde \psi}{\partial t}
-\frac{1}{c^2}\frac{\partial^2\tilde\psi}{\partial t^2}
=\frac{m^2c^2}{\hbar^2}\tilde\psi ,
\end{equation}
which is nothing but the non-relativistic wave equation:
\begin{equation}
\label{nonrel_waveeq}
\nabla^2 \tilde\psi + \mathbf i\frac{2m}{\hbar}
\frac{\partial\tilde\psi}{\partial t}
-\frac{1}{c^2}\frac{\partial^2\tilde\psi }{\partial t^2}=0. 
\end{equation}
 
The imaginary coefficient
ensures that there is no special direction in time. This equation is 
invariant under the transformations 
$t\to -t$ and $\tilde \psi\to\tilde \psi^{*}$,
whereby $\tilde \psi\tilde \psi^*$, 
where $*$ means conjugation of complex numbers, 
remains unchanged (this conjugation will be denoted by 
$\overline\psi$ later on). According to quantum theory,
the physically measurable quantity is not the wave
functions $\psi$ or $\tilde \psi$, but the probability densities 
$\psi\psi^*$ resp. $\tilde \psi\tilde \psi^*$.
 
\subsection{Meson theory.}

The above solutions of the wave equation strongly relate to the
theory of nuclear forces and mesons. 
To review this field, we literally quote Hideki Yukawa's Nobel Lecture
"Meson theory in its developments", 
delivered on December 12, 1949. Despite
the fact that elementary particle physics went through enormous 
developments since 1949, this review has been chosen
because it is highly suggestive regarding 
the physical interpretations of the 
mathematical models developed in this paper:

"The meson theory started from the extension of the concept of the 
field of force so as to include the nuclear forces in addition 
to the gravitational and electromagnetic forces. The necessity of 
introduction of specific nuclear forces, which could not be reduced 
to electromagnetic interactions between charged particles, was 
realized soon after the discovery of the neutron, which was to be 
bound strongly to the protons and other neutrons in the atomic nucleus. 
As pointed out by Wigner$^1$, specific nuclear forces between two
nucleons, each of which can be either in the neutron state or 
the proton state, must have a very short range of the order of 
10-13 cm, in order to account for the rapid increase of the 
binding energy from the deuteron to the alphaparticle. The binding 
energies of nuclei heavier than the alpha-particle do not increase 
as rapidly as if they were proportional to the square of the mass 
number A, i.e. the number of nucleons in each nucleus, 
but they are in fact approximately proportional to A. This 
indicates that nuclear forces are 
saturated for some reason. Heisenberg$^2$
suggested that this could be accounted for, if we assumed a force 
between a neutron and a proton, for instance, due
to the exchange of the electron or, more generally, due to the exchange
of the electric charge, as in the case of the chemical 
bond between a hydrogen atom and a proton. Soon afterwards, 
Fermi$^3$ developed a theory of beta-decay based on the hypothesis 
by Pauli, according to which a neutron,
for instance, could decay into a proton, 
an electron, and a neutrino, which
was supposed to be a very penetrating neutral particle with a 
very small mass. 

This gave rise, in turn, to the expectation 
that nuclear forces could be reduced to the exchange of a pair of 
an electron and a neutrino between two nucleons, just as 
electromagnetic forces were regarded as due to the exchange
of photons between charged particles. It turned out, however, that
the nuclear forces thus obtained was much too small$^4$, 
because the betadecay was a very slow process compared with 
the supposed rapid exchange of the electric charge responsible 
for the actual nuclear forces. The idea of the meson field was 
introduced in 1935 in order to make up this gaps. Original assumptions 
of the meson theory were as follows: 

I. The nuclear forces are described by a scalar field U, 
which satisfies the wave equation 
\begin{equation}
\label{y1}
\big(\frac{\partial^2}{\partial Z_1^2}+\frac{\partial^2}{\partial Z_2^2}
+\frac{\partial^2}{\partial Z_3^2}-\frac{1}{ c^2}
\frac{\partial^2}{\partial t^2}
-\kappa^2\big )U=0
\end{equation}
in vacuum, where x is a constant 
with the dimension of reciprocal length. Thus, the static potential 
between two nucleons at a distance r is proportional
to $\exp (-xr ) / rh$, the range of forces being given by $1/x$.

II. According to the general principle of quantum theory, the field U is
inevitably accompanied by new particles or quanta, which have the mass
\begin{equation}
\mu =\frac{\kappa\hbar}{ c}
\end{equation}
and the spin $0$, obeying Bose-Einstein statistics. The mass of 
these particles can be inferred from the range of nuclear forces. If 
we assume, for instance, $x = 5 \times 10^{12} cm$$^1$, we obtain 
$\mu\sim 200 m_e$, where $m_e$ is the mass of the electron.

III. In order to obtain exchange forces, we must assume that these mesons
have the electric charge $+ e$ or $- e$, and that a positive (negative) 
meson is emitted (absorbed) when the nucleon jumps from the proton state 
to the neutron state, whereas a negative (positive) meson is emitted 
(absorbed) when the nucleon jumps from the neutron to the proton. 
Thus a neutron and a proton can interact with each other by 
exchanging mesons just as two charged particles interact by 
exchanging photons. In fact, we obtain an exchange force of Heisenberg 
type between the neutron and the proton of the correct magnitude, 
if we assume that the coupling constant g between the nucleon and 
the meson field, which has the same dimension as the elementary
charge $e$, is a few times larger than $e$.

However, the above simple theory was incomplete in various respects. For
one thing, the exchange force thus obtained was repulsive for triplet 
S-state of the deuteron in contradiction to the experiment, and 
moreover we could
not deduce the exchange force of Majorana type, which was necessary in
order to account for the saturation of nuclear forces just at 
the alpha-particle. In order to remove these defects, 
more general types of meson fields including vector, pseudoscalar 
and pseudovector fields in addition to the scalar
fields, were considered by various authors$^6$. In particular, the 
vector field was investigated in detail, because it could give a 
combination of exchange
forces of Heisenberg and Majorana types with correct signs 
and could further account for the anomalous magnetic moments 
of the neutron and the proton qualitatively. Furthermore, the 
vector theory predicted the existence of noncentral forces between 
a neutron and a proton, so that the deuteron might have the electric 
quadripole moment. However, the actual electric quadripole
moment turned out to be positive in sign, whereas the vector theory
anticipated the sign to be negative. The only meson field, which gives the
correct signs both for nuclear forces and for the electric quadripole 
moment of the deuteron, was the pseudoscalar field$^7$. 
There was, however, another feature of nuclear forces, 
which was to be accounted for as a consequence of
the meson theory. Namely, the results of experiments on the scattering of
protons by protons indicated that the type and magnitude 
of interaction between two protons were, at least approximately, 
the same as those between a neutron and a proton, apart from 
the Coulomb force. Now the interaction between two protons or 
two neutrons was obtained only if we took into account the 
terms proportional to $g^4$, whereas that between a neutron and a
proton was proportional to $g^2$, as long as we were considering charged
mesons alone. Thus it seemed necessary to assume further:

IV. In addition to charged mesons, there are neutral mesons with the mass
either exactly or approximately equal to that of charged mesons. They must
also have the integer spin, obey 
Bose-Einstein statistics and interact with
nucleons as strongly as charged mesons. This assumption obviously 
increased the number of arbitrary constants in meson theory, 
which could be so adjusted as to agree with a variety of experimental
facts. These experimental facts could-not be restricted to those of
nuclear physics in the narrow sense, but was to include 
those related to cosmic rays, because we expected that mesons 
could be created and annihilated due to the interaction of cosmic 
ray particles with energies much larger than cosmic rays in 1937$^8$ 
was a great encouragement to further developments of
meson theory. At that time, we came naturally to the conclusion that the
mesons which constituted the main part of the hard component of cosmic
rays at sea level was to be identified with the mesons 
which were responsible for nuclear force$^9$. Indeed, cosmic 
ray mesons had the mass around $200 m_e$
as predicted and moreover, there was the definite evidence 
for the spontaneous decay, which was the consequence of the 
following assumption of the original mesonn theory :

V. Mesons interact also with light particles, i.e. electrons 
and neutrinos, just as they interact with nucleons, the only 
difference being the smallness of the coupling constant g’ in 
this case compared with g. Thus a positive (negative)
meson can change spontaneously into a positive (negative) electron and a
neutrino, as pointed out first by Bhabha$^10$. 
The proper lifetime, i.e. the mean
lifetime at rest, of the charged scalar meson, for example, is given by

\begin{equation}
\tau_0=2\big (\frac{\hbar c}{ (g^\prime)^2}\big ) 
\big (\frac{\hbar}{ \mu c^2}\big )
\end{equation}

For the meson moving with velocity $\nu$, the lifetime increases 
by a factor $1 /\sqrt{1- (\nu/c)2}$ 
due to the well-known relativistic delay of the moving clock.
Although the spontaneous decay and the velocity dependence of the lifetime
of cosmic ray mesons were remarkably confirmed by various 
experiments$^{11}$,
there was an undeniable discrepancy between theoretical and experimental
values for the lifetime. The original intention of meson 
theory was to account
for the beta-decay by combining the assumptions III and V together.
However, the coupling constant $g’$, which was so adjusted as to give the
correct result for the beta-decay, turned out to be too large 
in that it gave the lifetime $\tau_0$ of mesons of the order of 
$10^{-8} sec$, which was much smaller
than the observed lifetime $2 x 10^{-6} sec$. 
Moreover, there were indications,
which were by no means in favour of the expectation that cosmic-ray mesons
interacted strongly with nucleons. For example, the observed crosssection
of scattering of cosmic-ray mesons by nuclei was much smaller than
that obtained theoretically. Thus, already in 1941, 
the identification of the
cosmic-ray meson with the meson, which was supposed to be responsible
for nuclear forces, became doubtful. In fact, Tanikawa and Sakata$^{12}$ 
proposed in 1942 a new hypothesis as follows: 
The mesons which constitute the hard
component of cosmic rays at sea level are not 
directly connected with nuclear
forces, but are produced by the decay of heavier mesons which interacted
strongly with nucleons.
However, we had to wait for a few years before this two-meson hypothesis
was confirmed, until 1947, when two very important facts were discovered.
First, it was discovered by Italian physicists’s that the negative mesons
in cosmic rays, which were captured by lighter atoms, 
did not disappear instantly,
but very often decayed into electrons in a mean time interval of the
order of $10^{-6} sec$. This could be understood 
only if we supposed that ordinary
mesons in cosmic rays interacted very weakly with nucleons. 
Soon afterwards, Powell and others$^{14}$ discovered 
two types of mesons in cosmic rays,
the heavier mesons decaying in a very short time into lighter mesons. Just
before the latter discovery, the two-meson hypothesis was proposed by
Marshak and Bethe$^{15}$ independent of the Japanese 
physicists above mentioned.
In 1948 , mesons were created artificially in Berkeley$^{16}$ and 
subsequent experiments confirmed the general picture of two-meson theory. 
The fundamental assumptions are now$^{17}$

(i) The heavier mesons, i.e. n-mesons with the mass $m_\pi$, 
about $280 m_e$ interact
strongly with nucleons and can decay into lighter mesons, i.e. 
$\pi$-mesons and neutrinos with a lifetime of the order of 
$10^{-8} sec$; $\pi$-mesons have integer spin
(very probably spin $0$) and obey Bose-Einstein statistics. 
They are responsible for, at least, a part of nuclear forces. 
In fact, the shape of nuclear potential
at a distance of the order of $\hbar/m_\pi c$ or larger could 
be accounted for as due to
the exchange of $\pi$-mesons between nucleons.

(ii) The lighter mesons, i.e. $\mu$-mesons with the mass about $210 m_e$ 
are the main constituent of the hard component of cosmic rays 
at sea level and can decay into electrons and neutrinos with the 
lifetime $2 x 10^{-6} sec$. They have
very probably spin ${1\over 2}$ and obey Fermi-Dirac statistics. 
As they interact only
weakly with nucleons, they have nothing to do with nuclear forces.
Now, if we accept the view that $\pi$-mesons are the mesons that have been
anticipated from the beginning, then we may expect the existence of neutral
$\pi$-mesons in addition to charged p-mesons. Such neutral mesons, 
which have integer spin and interact as strongly as charged 
mesons with nucleons, must be very unstable, because each 
of them can decay into two or three photons$^18$.
In particular, a neutral meson with spin $0$ can decay into two photons and
the lifetime is of the order of $10^{-14} sec$ or even less than that. 
Very recently, it became clear that some of the experimental 
results obtained in Berkeley could be accounted for consistently 
by considering that, in addition to
charged n-mesons, neutral $\nu$-mesons 
with the mass approximately equal to
that of charged p-mesons were created 
by collisions of high-energy protons
with atomic nuclei and that each of these neutral 
mesons decayed into two
mesons with the lifetime of the order of 
$10^{-13} sec$ or less$^{19}$.
Thus, the neutral mesons must have spin $0$.
In this way, meson theory has changed a great deal during these fifteen
years. Nevertheless, there remain still many questions unanswered. Among
other things, we know very little about mesons heavier than $\pi$-mesons. 
We do not know yet whether some of the heavier mesons are responsible for
nuclear forces at very short distances. The present form of meson theory is
not free from the divergence difficulties, although recent development of
relativistic field theory has succeeded in removing some of them. We do not
yet know whether the remaining divergence 
difficulties are due to our ignorance
of the structure of elementary particles themselves$^{20}$.
We shall probably
have to go through another change of the theory, before we shall be
able to arrive at the complete understanding 
of the nuclear structure and of
various phenomena, which will occur in high energy regions.

1. E. Wigner, Phys. Rev., 43 (1933 ) 252.

2. W. Heisenberg, Z. Physik, 77 (1932) I; 78 (1932) 156; 80 (1933) 587

3. E. Fermi, Z. Physik, 88 (1934) 161.

4. I. Tamm, Nature, 133 (1934) 981; D. Ivanenko, Nature, 133 (1934) 981.

5. H. Yukawa, Proc. Phys.-Math. SOc.Japan, 17 (1935) 48 ; 
H. Yukawa and S. Sakata, ibid., I9 (1937) 1084.

6. N. Kemmer, Proc. Roy. SOc. London, A 166 (1938) 127; 
H. Fröhlich, W. Heitler, and N. Kemmer, 
ibid., 166 (1938) 154; H. J.. Bhabha, ibid., 
166 (1938) 501; E. C.
G. Stueckelberg, Helv.Phys. Acta, 
II (1938) 299; H. Yukawa, S. Sakata, and M.
Taketani, Proc. Phys.-Math. SOc. Japan, 20 (1938) 319; H. 
Yukawa, S. Sakata, M.
Kobayasi, and M. Taketani, ibid., 20 (1938) 720.

7. W. Rarita and J. Schwinger, Phys. Rev., 59 (1941) 436, 556.

8. C. D. Anderson and S. H. Neddermeyer, Phys. Rev., 
51 (1937) 884; J. C. Street
andE.C. Stevenson, ibid., 5I(1937) 1005; Y. Nishina, 
M.Takeuchi, and T. Ichimiya, ibid., 52 (1937) 1193.

9. H. Yukawa, Proc. Phys.-Math. Soc. Japan, 19 (1937) 712; 
J. R. Oppenheimer and R. Serber, Phys. Rev., 51 (1937) 1113; 
E. C. G. Stueckelberg, ibid., 53 (1937) 41.

IO. H. J. Bhabha, Nature,141 (1938) 117.

11. H. Euler and W. Heisenberg, Ergeb. Exakt. Naturw., I: 
(1938) I ; P. M. S. Blackett,
Nature, 142 (1938) 992; B. Rossi, Nature, 142 (1938) 993; 
P. Ehrenfest, Jr. and A. Freon, Coopt. Rend., 207 (1938) 853 ; 
E. J. Williams and G. E. Roberts, Nature, 145 (1940) 102.

12. Y. Tanikawa, Progr. Theoret. Phys. Kyoto, 2 (1947) 220; 
S. Sakata and K. Inouye, ibid., I (1946) 143.

13. M. Conversi, E. Pancini, and O. Piccioni, Phys. Rev., 71 ( 1947) 209.

14.. C. M. G. Lattes, H. Muirhead, G. P. S. Occhialini, 
and C. F. Powell, Nature, 159 (1947) 694; C. M. G. Lattes, 
G. P. S. Occhialini, and C. F. Powell, Nature, 160 (1947) 453,486.

15, R. E. Marshak and H. A. Bethe, Phys. Rev., 72 (1947) 506.

16. E. Gardner and C. M. G. Lattes, Science, 107 (1948) 270; 
W. H. Barkas, E. Gardner, and C. M. G. Lattes, Phys. Rev., 74 (1948) 1558.

17. As for further details, see H. Yukawa, Rev. Mod. Phys., 21 (1949) 474.
134 1949 H . Yukawa

18. S. Sakata and Y. Tanikawa, Phys. Rev., 57 (1940) 548; R. J. 
Finkelstein, ibid., 72 (1947) 415.

19. H. F. York, B. J. Moyer, and R. Bjorklund, Phys. Rev., 76 (1949) 187.

20. H. Yukawa, Phys. Rev., 77 (1950) 219."

\subsection{Zeeman and Schr\"odinger operators of electrons.}

In physics, the classical Zeeman operator of a charged particle
is: 
\begin{equation}
\label{land}
-{\hbar^2\over 2m}\Delta_{(x,y)} -
{\hbar eB\over 2m c\mathbf i}
D_z\bullet
+{e^2B^2\over 8m c^2}(x^2+y^2) +eV .
\end{equation}
Originally, this operator is considered 
on the 3-space expressed in terms of 
3D Euclidean Laplacian and 3D magnetic dipole momentum operators.
The latter operators are the first
ones where a preliminary version of spin concept appears in the history
of physics. This is the so called exterior or orbiting spin 
associated with the 3D angular momentum
\begin{eqnarray}
\label{3Dang_mom}
\mathbf P=(P_1,P_2,P_3)=\frac{1}{\hbar}Z\times\mathbf p\quad\quad
{\rm where}\\
P_1=\frac{1}{h}(Z_2p_3-Z_3p_2)=\frac{1}{\mathbf i}
\big(Z_2\frac{\partial}{\partial Z_3}-
Z_3\frac{\partial}{\partial Z_2}\big).
\dots
 \label{3Dangmom}
\end{eqnarray}
Note that the above 2D operator keeps only 
component $P_3$ of this angular momentum.
The components of the complete 3D angular momentum 
obeys the commutation relations:
\begin{equation}
[P_1,P_2]=\mathbf iP_3,\quad [P_1,P_3]=-\mathbf iP_2,\quad 
[P_2,P_3]=\mathbf iP_1.
\end{equation}

In the mathematical models, the particles are orbiting in complex 
planes determined by the complex structures associated with the component
angular momenta $P_j$, thus the 2D version plays more important role
in this paper than the 3D version. 
The 2D operator is obtained from the 3D operator 
by omitting $P_1$ and $P_2$ and by restricting the rest
part onto the $(x,y)$ plane.
If Coulomb potential $V$ is omitted, it is nothing but 
the {\it Ginsburg-Landau-Zeeman operator} of a charged
particle orbiting on the $(x,y)$-plane in a constant magnetic field
directed toward the z-axis.
The {\it magnetic dipole momentum operator}, which is the term involving 
$D_z\bullet :=x\partial_y-y\partial_x$ and which is associated with the
{\it angular momentum operator} $hD_z$, 
commutes with the 
rest part, $\mathbf O$, of the operator, therefore, splitting the
spectral lines of $\mathbf O$. The {\it Zeeman effect} is explained 
by this fine structure of the Zeeman operator.

The Hamilton operator represents the total energy of a given physical 
system. More precisely, the eigenvalues of this operator are the discrete
(quantized) energy values which can be assumed by the system. Thus,
correspondence (\ref{correl}) implies Schr\"odinger's wave equation 
\begin{equation}
\label{land}
-\big({\hbar^2\over 2m}\Delta_{(x,y)} -
{\hbar eB\over 2m c\mathbf i}
D_z\bullet
+{e^2B^2\over 8m c^2}(x^2+y^2) -eV\big)\psi =\mathbf i\hbar
\frac{\partial \psi}{\partial t}
\end{equation}
of an electron orbiting in the $(x,y)$-plane. 

As it well known, Schr\"odinger discovered first the
relativistic equation which is a second order differential operator 
regarding the $t$-variable. By the time that Schr\"odinger came to 
publish this equation, it had already been independently 
rediscovered by O. Klein and W. Gordon. This is why it is usually 
called Klein-Gordon equation. Numerous problems had arise regarding
this equation. Schr\"odinger became discouraged because it gave the
wrong fine structure for hydrogen. Some month later he realized, however,
that the non-relativistic approximation to his relativistic equation
was of value even if the relativistic equation was incorrect. This
non-relativistic approximation is the familiar Schr\"odinger equation.

Dirac also had great concerns about the Klein-Gordon equation. 
His main objection
was that the probabilistic quantum theory based on this equation produced
negative probabilities. Actually, the elimination of this problem led
Dirac to the discovery of his relativistic electron equation. By this
theory, however, proper probabilistic theory can be developed only
on the relativistic space-time. This feature was strongly 
criticized by
Pauli, according to whom such theory makes sense only on the space.

The journey toward an understanding of the nature of spin and its 
relationship to statistics
has been taking place on one of the most difficult and exciting routes
\cite{b,tom}. 
Although the Schr\"odinger wave equation
gives excellent agreement with experiment in predicting the frequencies
of spectral lines, small discrepancies are found, which can be explained
only by adding an intrinsic angular momentum to its usual orbital
angular momentum of the electron that acts as if it came from a spinning
solid body. The pioneers of developing this concept 
were Sommerfeld, Land\'e, and Pauli. They found that agreement 
with the Stern-Gerlach experiment
proving the existence of Zeeman effect 
can be obtained by assuming that the
magnitude of this additional angular momentum was $\hbar /2$. The
magnetic moment needed to obtain agreement was, however, $e\hbar /2mc$,
which is exactly the same as that arising from an orbital angular moment
of $\hbar$. The gyromagnetic ratio, that is, the ratio of magnetic moment
to angular momentum is therefore twice as great for electron spin as it is
for orbital motion.

Many efforts were made to connect this intrinsic angular momentum to an 
actual spin of the electron, considered it as a rigid body. In fact,
the gyromagnetic ratio needed is exactly that which would be obtained
if the electron consisted of a uniform spherical shell spinning about
a definite axis. The systematic 
development of such a theory met, however,
with such great difficulties that no one was able to carry it through
to a definite conclusion. Somewhat later, Dirac derived his above
mentioned relativistic
wave equation for the electron, in which the spin and charge were shown
to be bound up in a way that can be understood  only in connection
with the requirements of relativistic invariance. In the non-relativistic
limit, however, the electron still acts as if it had an intrinsic angular
momentum of $\hbar /2$. Prior to the Dirac equation, this non-relativistic
theory of spin was originally developed by Pauli.

Finally, we explain yet why the Coulomb operator does not appear in 
this paper and how can it be involved into the further investigations.  
The present ignorance is mainly due to the fact 
that also
the Hamilton operator (Laplacian) on nilpotent
groups involves no Coulomb potential. There appear, instead, 
nuclear potentials like those Yukawa described in meson theory. 
An other major distinguishing feature is that this Laplacian 
(Hamilton operator) includes also terms corresponding 
to the electron-positron-neutrino, which, by the standard model, is always
something of a silent partner in an electron-positron-system, because,
being electrically neutral, it ignores not only the nuclear force but
also the electromagnetic force. Although the operator corresponding 
to this silent partner will be established
by computations evolved by Pauli to determine the non-relativistic 
approximation, all these operators appear in the new theory 
as relativistic operators complying with Einstein's equation of general
relativity.   

The Coulomb force can be considered just later, after developing certain
explicit spectral computations. This spectral theory
includes also a spectral decomposition of the corresponding $L^2$-function
spaces such that the subspaces appearing in this decomposition are 
invariant with respect to the actions both of the Hamilton operator and
the complex Heisenberg group representation. Also the latter 
representation is naturally inbuilt into these mathematical models. 

The complications about the Coulomb operator
are due to the fact that these invariant subspaces 
(called also zones) are not
invariant regarding the Coulomb's multiplicative operator 
\cite{sz5}- \cite{sz7}.
In order to extend the theory also to electric fields, the Coulomb 
operator must be modified such that also this operator leaves the zones
invariant. Such natural zonal Coulomb operator can be defined for a 
particular zone such that,
for a given function $\psi$ from the zone, function $V\psi$ is projected
back to the zone. This modified Coulomb operator is the correct one
which must be added to the Laplace operator on the nilpotent Lie group
in order to have a relevant unified electro-magnetic particle theory.
 
However, this modified Coulomb force is externally 
added and not naturally inbuilt into the Laplacian (Hamiltonian) of 
the nilpotent Lie group. In order to construct such Riemann manifolds
whose Laplacian unifies the Ginsburg-Landau-Zeeman +neutrino+
nuclear operators also 
with an appropriate Coulomb operator, the mathematical models must be 
further developed such that,
instead of nilpotent Lie groups, one considers general nilpotent-type
Riemann manifolds and their solvable-type extensions. This generalization
of the theory,
which is similar to passing from special relativity to the general 
one, will be the third step in developing this theory.

\section{Launching the mathematical particle theory.}

The physical features of elementary particles 
are most conspicuously exhibited also by certain Riemann manifolds. 
A demonstration of this apparent physical content present
in these abstract mathematical structures is, for instance, 
that the classical Hamilton and Schr\"odinger operators
of elementary particle systems appear as Laplace-Beltrami
operators defined on these manifolds. Thus, these
abstract structures are really deeply inbuilt into the very fabric 
of the physical world which can serve also as fundamental tools for 
building up a comprehensive unified quantum theory.  

Yet, this physical content of these particular mathematical 
structures has never been recognized in the 
literature so far. The ignorance is probably due to the fact  
that this new theory is not in a direct genetic relationship
with fundamental theories such as the Gell-Mann$\sim$Ne'eman 
theory of quarks by which the standard model of elementary particles
has been established. Neither is a direct genetic connection to 
the intensively studied super string
theory, which, by many experts, is thought to be 
the first viable candidate ever for a 
unified quantum field theory of all of the elementary particles and their
interactions which had been provisionally described by the standard model. 
Actually, the relationship between the two 
approaches to elementary particle physics is more precisely described
by saying that there are both strong connections and
substantial differences between the standard and our new models. 
In order to clearly explain the new features of
the new model, we start with a brief review of the standard model. 
A review of string theory is, however, beyond the scope of this article.

\subsection{A rudimentary review of the standard model.} 
The Gell-Mann$\sim$Ne'eman
theory \cite{gn} of quantum chromodynamics (QCD) grew out, in the 60's, 
from Yang-Mills' \cite{ym} 
non-Abelian gauge theory where the
gauge group was taken to be the $SU(2)$ group of isotopic spin 
rotations, and the vector fields analogous to the photon field were
interpreted as the fields of strongly-interacting vector mesons of 
isotopic spin unity. QCD is also a non-Abelian gauge theory
where, instead of $SU(2)$, the symmetry group is
$SU(3)$. This model adequately described the rapidly growing
number of elementary particles by grouping the known baryons and
mesons in various irreducible representations of this gauge group.

The most important difference between this and the new mathematical 
model is that the Yang-Mills 
theories do not introduce a (definite or indefinite) Riemann metric
with the help of which all other physical objects are defined. They
rather describe the interactions by Lagrange functions which contain more 
than a dozen arbitrary constants, including those yielding the various
masses of the different kinds of particles. What makes this situation even 
worser is that all these important numbers are incalculable in principle.
Under such circumstances the natural Hamilton and wave operators
of particles can not emerge as Laplacians on Riemann manifolds.
This statement is the main attraction on the new models.

To speak more mathematically, the fundamental structures for 
non-Abelian gauge theories are principal fibre bundles over 
Minkwski space with compact non-Abelian structure group $SU(n)$ on which
a potential is defined by a connection $A$, with components $A_\mu$, 
in the Lie algebra $su(n)$. The field is the curvature whose
components are
$F_{\nu\mu}=\partial_\mu A_\nu -\partial_\nu A_\mu +[A_\mu ,A_\nu ]$.
The most straightforward generalization of Maxwell's equations are
the Yang-Mills equations $dF=0$ and  $d^*F=0$, where $d$ and $d^*$ are
covariant derivatives. Gauge theories possess an infinite-dimensional
symmetry group given by functions $g: M\to SU(n)$ and all physical
or geometric properties are gauge invariant.

To specify a physical theory the usual procedure is to define a 
Lagrangian. In quantum chromodynamics (QCD) such Lagrangian is to be 
chosen which is capable to portray the elementary particles in 
the following very rich complexity: 
The neutron and
proton are composite made of quarks. There are quark of six
type, or "flavors", the $u, c$, and $t$ quarks having charge $2/3$,
and the $d, s$, and $b$ quarks having charge $-1/3$ (these denotations
are the first letters of words: up, down, charm, strange, top, 
and bottom). Quarks of each
flavor come in three "colors" which furnish the defining representation
$\mathbf 3$  of the $SU(3)$ gauge group.
 
Quarks have the remarkable property 
of being permanently trapped inside "white" particles such 
as neutron, proton, baryon and meson. Baryon resp.
meson are color-neutral
bound states of three quarks resp. quarks and antiquarks. Neutron resp.
proton are barions consisting an up- and two down- resp. one down- and 
two up-quarks. Thus the neutron has no charge while, in the same units
in which the electron has an electric charge of $-1$, the 
proton has a charge of $+1$. The total charge of a white particle
is always an integer number. Only the quarks confined inside of them
can have non-integer charges. 

QCD can be regarded as the 
modern theory of strong nuclear forces holding the quarks together.  
With no scalar fields,
the most general renormalizable Lagrangian describing also these strong
interactions can be put in the form \cite{w2}
\begin{equation}
\label{QCDLagrange}
\mathcal L=-\frac 1{4}F^{\alpha \mu\nu}{F^\alpha}_{\mu\nu} 
-\sum_n\overline{\psi}_n[x^\mu(\partial_\mu-
\mathbf igA^\alpha_\mu t_\alpha )+m_n]\psi_n,
\end{equation}
where $\psi_n(x)$ is a matter field, $g$ is the strong coupling constant,
$t_\alpha$ are a complete set of generators of color $SU(3)$ in the 
$\mathbf 3$-representation (that is , 
Hermitian traceless $3\times 3$ matrices 
with rows and columns labelled by the three quark colors), normalized so 
that $Tr(t_\alpha t_\beta )=\delta_{\alpha\beta}/2$, and the subscript
$n$ labels quark flavors, with quark color indices suppressed. 
The first term is called matter Lagrangian density, the second one is
the gauge field. 

Just as the electromagnetic force between electrons is generated by
the virtual exchange of photons, so the quarks are bound to one another
by a force that comes from the exchange of other quanta, called gluons
because they glue the quarks together to make observable white objects.
The gluons are flavor blind, paying no attention to flavor, however,
they are very sensitive of color. They interact with color much as the
photon interacts with electron charge.

\subsection{More specifics about the new abstract model.}
The above sketchily described objects are the most fundamental 
concepts in QCD. 
Their properties are established by the Lagrangians introduced there. 
In order to compare them, we review some more details 
about the new theory. The concepts introduced here will rigorously
be establishment in the following sections.

The mathematical structures on which the new theory is built on 
are 2-step nilpotent Lie groups and their 
solvable extensions. Both type of manifolds are endowed
with natural left invariant metrics. In this scheme, the 
nilpotent group plays the role of space, on which always
positive definite metric is considered. This choice is dictated also by
the fact that the Hamilton operators of elementary particle 
systems emerge
as the Laplacians of these invariant Riemann metrics. In order to
ensure that these systems have positive energies,
just positive definite metric can be chosen on these manifolds; 
for indefinite metrics the Laplacian never 
appears as the Hamilton operator of a particle system.

Time can be introduced by adding  
new dimension to these nilpotent manifolds.
This can be implemented either by a simple Cartesian product with the 
real line $\mathbb R$, or
by the solvable extensions of nilpotent groups. Both
processes increase the dimension of the nilpotent groups
by $1$ and in both cases invariant indefinite metrics are 
defined such that the time-lines intersect the nilpotent subgroup
perpendicularly, furthermore, also 
$\langle\partial_t,\partial_t\rangle <0$ holds.
On these extended manifolds the Laplacian appears as the natural
wave operator (Schr\"odinger operator) attached to the particle
systems. The difference between the two constructions 
is that the first one 
provides a static model, while the second one is an
expanding model which yields the Hubble law of cosmology.

Although both are relativistic, these space-time concepts
are not quite the same than those developed in general relativity. 
Actually, Einstein's 4D space-time concept 
has no room for exhibiting the rich ``inner life'' of particles which
is attributed to them by meson theory, or, by the general 
standard model of elementary particle physics. This ``inner life" can
not be explained only by the properties of space-time.
For instance, the symmetries underlying the electroweak theory are
called internal symmetries, because one thinks of them as
having to do with the intrinsic nature of the particles, rather 
than their position or motion.
 
The abstract mathematical models, however, 
do make room for both the rich ``inner life" 
and ``exterior life" of particles. 
The main tool for exhibiting  
the inner physics is the center of the  
nilpotent Lie algebra, while the stage for the ``exterior life" 
(that is, for the motion
of particles) is the so called X-space denoted by $\mathcal X$. 
This space is a complement
of the center $\mathcal Z$, which is called also Z-space. 

The space-like Z-space exhibits, actually, 
dualistic features. The primary meaning of vectors lying in the center
is that they are the axes of angular momenta defined 
for the charged
particles which are orbiting in complex planes in 
constant magnetic 
fields standing perpendicular to these complex planes. Actually, this 
axis-interpretation of vectors, $Z$, is developed in
the following more subtle way: For any unit vector $Z$ there is a complex
structure, $J_Z$ acting on the X-space corresponded such that
the particles are orbiting in the complex planes defined by $J_Z$
along the integral curves of vector fields defined by $X\to J_Z(X)$. 
As it is pointed out in the next section,
this 2-step nilpotent Lie group is uniquely determined by the linear space,
 $J_{\mathcal Z}$, of skew endomorphisms $J_Z$. Bijection $Z\to J_Z$
provides a natural identification between $\mathcal Z$ and
$J_{\mathcal Z}$.
More precisely, the group can be
considered such that it is defined by a linear space of 
skew angular momentum
endomorphisms acting on a Euclidean space, $\mathcal X$, such that it is 
considered, primarily,
as an abstract space $\mathcal Z$ which is identified with the 
endomorphism space, $J_{\mathcal Z}$, by the natural bijection 
$Z\to J_Z$.    
Note that in this interpretation, the 
axis, $Z$, of the angular momentum, $J_Z$, is separated from 
the complex plane where the actual orbiting is taking place. Anyhow, 
from this point of view,
the Z-vectors exhibit space-like features. 

Whereas, 
the constant magnetic field defined by the structure pins 
down a unique inertia system on which relations
$B=constant$ and $E=0$ holds. Thus 
a naturally defined individualistic inner time is given 
for each of these particles. This time can be synchronized, allowing
to define also a common time, $T$, which defines the time 
both on the static models and the solvable extensions. 
From this point of view, the center exhibits time-like features.
This argument clarifies the contradiction between the 
angular-momentum-axis- and the customary time-axis-interpretation 
of the center of the Heisenberg groups.

Although the concepts of relativity and quantum theory 
appear in new forms, they 
should be considered as refined versions 
of the original classical objects. Beside the above one, an other example
for this claim is the new form by which 
de Broglie's waves are introduced
on these groups. The most important new feature 
is that the Fourier transform  
is defined only on the center, $\mathcal Z=\mathbb R^l$, 
by the following formula:
\begin{equation}
\label{newDiracwave}
\int_{\mathbb R^l}A(|X|,K)\prod z_i^{p_i}(K_u,X)\overline z_i^{q_i}(K_u,X)
e^{\mathbf i(\langle K,Z\rangle-\omega t)}dK,
\end{equation}
where, for a fixed complex basis $\mathbf B$, complex coordinate system 
$\{ z_i(K_u,X)\}$ on the X-space is defined regarding the complex
structure $J_{K_u}$, for all unit Z-vector $K_u$. This so called
twisted Z-Fourier transform binds the Z-space and the X-space together
by the polynomials $\prod z_i^{p_i}\overline z_i^{q_i}$ which depend
both on the X- and K-variables. It appears also in several other 
alternative forms. Due to this complexity of the wave functions,
the three main forces: the electromagnetic; the weak; and
the strong forces of particle theory can be introduced in a 
unified way such that each of them
can be expressed in terms of the weak forces.

The main objects on these abstract structures are the Laplace operators 
considered both on nilpotent and solvable groups. 
They turn out to be the Hamilton resp. Schr\"odinger operators of the 
particle-systems represented by these metric groups. 
Due to the fact that these
operators are Laplace operators on Riemann manifolds, 
the conservation of energy is
automatically satisfied. In classical quantum theory 
the Hamilton function,
which counts with the total energy of a system, is replaced
by the Hamilton operator whose discrete eigenvalues are the quantized
energy levels on which the system can exist.  
On the relativistic mathematical models the total energy is encoded 
into the Einstein tensor (stress-energy tensor) of the indefinite 
Riemannian metrics. In the quantum theory developed on these manifold,
this stress-energy tensor is replaced by the Laplacians of 
these manifolds, which are, actually, the Hamilton resp. 
Schr\"odinger operators of the
particle-systems represented by these models. In other words, this is a
correspondence principle associating the Laplacian resp. the 
eigenfunction-equations to the Einstein tensor resp. 
Einstein equation defined on these indefinite Riemann manifolds.

Let it be emphasized again that this theory will be evolved
gradually without adding any new objects to those defined mathematically 
on these abstract structures. 
The main focus is going to be to rediscover the most important 
physical features which are known by the standard model. 
Before starting this exploration, 
we describe, yet, a more definite bond between the two models.

\subsection{Correspondence principle bridging the two models.}
Correspondence principle associating 2-step nilpotent groups
to $SU(n)$-mo\-dels can also be introduced.
It can be considered such that, to the Lagrange functions defined
on the $SU(n)$-models, 2-step metric nilpotent Lie groups are 
corresponded. The combination of this correspondence principle with the
above one associates the Laplacian of the metric group 
to the Lagrange function defined on the Yang-Mills model. 
This association explains why are
the conclusions about the nature of 
electromagnetic, strong, and weak
forces so similar on the two models. This bridge
can be built up as follows.

As it is explained above,
the invariant Riemann metrics defined on 2-step nilpotent groups, 
modelling the particle systems in the new
theory, can be defined for any linear space, $J_{\mathcal Z}$, 
of skew endomorphisms acting on the X-space. Thus, for a
faithful representation, $\rho$, 
of $su(n)\subset so(2n)$ in the Lie algebra of real orthogonal 
transformations acting on a Euclidean space $\mathcal X$, one can define 
a natural 2-step nilpotent metric group by the endomorphism space 
$J_{\mathcal Z}=\rho (su(n))\subset so(\mathcal X)$ whose X-space
is $\mathcal X$ and the Z-space is the abstract linear space 
$\mathcal Z=\rho (su(n))$. 
Actually, this is the maximal 2-step nilpotent metric 
group which can be corresponded to a Yang-Mills principal fibre bundle
having structure group $SU(n)$, and whose fibres consists of orthonormal 
frames of $\mathcal X$ on which the action of $\rho (SU(n))$ is
one-fold transitive. Note that this group is still independent from
the Yang-Mills connection $A_\mu$ by which the Yang-Mills field 
(curvature), $F_{\mu\nu}$, is defined. 

By the holonomy group 
$Hol_p(A)\subset\rho (SU(n))_p$, 
defined at a fixed
point $p$ of the Minkowski space, 
groups depending on
Yang-Mills fields can also be introduced. 
Even gauge-depending groups can be
constructed, by sections $\sigma :M\to\tilde M$, where
$\tilde M$ denotes the total space of the fibre bundle 
and $\sigma (p)$ is lying in the fibre
over the point $p$. In this case the endomorphism space is
spanned by the skew 
endomorphisms $A_\mu (\sigma (p))$ considered for all $p\in M$ and indices
$\mu$. This construction depends on sections $\sigma (p)$. Since the
gauge group is transitive on the set of these sections, this 
correspondence is not gauge invariant. 

These correspondences 
associate 2-step nilpotent metric groups also to the    
representations of the particular Lie algebras 
$su(2)\subset so(4)$ resp. $su(3)\subset so(6)$ by which
the Yang-Mills- resp. Gell-Mann$\sim$Ne'eman-models are introduced.
This association does not mean, however, the equivalence of the two
theories. The X-space (exterior world) of the associated group is
the Euclidean space $\mathcal X$ where the skew endomorphisms 
from $\rho (su(n))$ are acting. The Z-space is the abstract
space $\mathcal Z=\rho (\mathit g)$ where $\mathit g$ can be any
of the subspaces of $su(n)$ which were introduced above. 
In other words, for a Yang-Mills model, X-space $\mathcal X$ is already 
given and angular momentum
endomorphism space is picked up in $\rho (su(n))$ in order to define
the Z-spaces of the corresponded nilpotent groups. The space-time, 
including both the exterior and interior worlds, are defined 
by these spaces. Note that this construction completely ignores the
Minkowski space which is the base-space for the Yang-Mills principle
fibre bundle. These arguments also show, 
that the nilpotent groups corresponded to a fixed 
Yang-Mills model are not uniquely determined, they depend on the Z-space
chosen on the YM-model.

Whereas, on a YM-model, the exterior world is the Minkowski space over
which the principal fibre bundle is defined and the interior world is
defined there by $\rho (su(n))$. An other fundamental difference is
that YM-models are based on submersion-theory, however, not this is the
case with the new models. For instance, the Hamilton operators of
particles never appear as sub-Laplacians defined on the X-space.
Quite to the contrary, these Hamilton operators are acting
on the total $(X,Z)$-space, binding the exterior and interior worlds
together into an unbroken unity not characteristic for submersions. 
These are the most important roots explaining
both the differences and similarities between the two models. 
These arguments also show
that the new theory deals with much more general situations 
than those considered in Yang-Mills' resp. Gell-Mann$\sim$Ne'e\-man's
$SU(2)$- resp. $SU(3)$-theories. It goes far beyond the 
$SU(n)\times\dots\times SU(n)$-theories. 
Since the Laplacians are natural Hamilton operators of 
elementary particles also in these 
most general situations, there is no reason to deny their involvement
into particle theory. 

Differences arise also regarding the Maxwell theory of electromagnetism.
The $SU(n)$-theories are non-Abelian gauge theories where 
the field is the curvature, $F$, of a Yang-Mills connection (potential)
and the properties are gauge invariant regarding an infinite dimensional
gauge group whose Lie algebra consists of $su(n)$-valued 1-forms, 
$\omega$, satisfying $d\omega =F$. The 
customary reference to this phenomenon
is that only $F$ is the physical object and 1-forms $\omega$ do not have 
any physical significance, they are the results of mere mathematical 
constructions. This interpretation strongly contrasts the Aharanov-Bohm
theory where these vector potentials do have physical meanings by which
the effect bearing their names can be established \cite{ab}, \cite{t}. 
Whereas, on the nilpotent groups and
their solvable extensions
the fundamental fields are the natural invariant Riemann metrics, $g$, 
by which all the other
objects are defined. The classical Hamilton resp. Schr\"odinger operators
emerge as Laplacians of metrics $g$.
Curvature $F$ corresponds to the field 
$g(J_Z(X),Y)$ in this interpretation. 
Since the basic objects are not invariant regarding their actions, 
the gauge-symmetries are not involved to these investigations. One-form 
$\omega_Z(Y)=(1/2)g(J_Z(X),Y)$, defined over the points $X\in\mathcal X$,
is the only element of the Lie algebra of the gauge-symmetry group
which is admitted to these considerations. It defines the constant
magnetic field and vanishing electric field which are 
associated with the orbiting spin and inner time $T$. 

This interpretation 
shows that this model breaks off, at some point, 
from the Maxwell theory whose greatest achievement was that it unified
the, until 1865, separately handled partial theories of magnetism
and electricity. The Yang-Mills gauge theory is a generalization of this
unified theory to vector valued fields and potentials. The nilpotent
Lie group model approaches to the electromagnetic 
phenomena from a different
angle. Since the potentials do have significance there, 
it stands closer to the Aharanov-Bohm theory
than to the Maxwell-Yang-Mills gauge theory. Actually the full recovery 
of electromagnetism under this new 
circumstances will be not provided in this
paper. Note that only the magnetic field has been appeared so far
which is not associated with a non-trivial electric field. In other
words, the magnetism and electricity emerge as being separated 
in this paper. 
The reunion of this temporarily separated couple can be established just
at a later point after further developing this new theory.

This bridge explains the great deal 
of properties which manifest similarly on these two models, however,
turning from one model to the other one is not simple and the
complete exploration of overlapping phenomena requires further extended
investigations.
In this paper the first order task is to firmly 
establish the point about the new model.
Therefore, when several properties are named by the same name 
used in QCD, their definition strictly refer to our
setting. Whereas, by these deliberately chosen names,  
the similarities between the two theories are indicated.

\section{Two-step nilpotent Lie groups.}
\subsection{Definitions and interpretations.}
A 2-step nilpotent metric Lie algebra, 
$\{\mathcal N ,\langle ,\rangle\}$, is defined on a real vector space 
endowed with a positive definite inner product.
The name indicates that the 
center, $\mathcal Z$, can be reached
by a single application of the Lie bracket, thus its second application
always results zero. 
The orthogonal complement of the center is denoted by $\mathcal X$. 
Then the Lie bracket operates among these subspaces according to
the following formulas: 
\begin{equation}
[\mathcal N,\mathcal N]=\mathcal Z\quad ,\quad 
[\mathcal N,\mathcal Z]=0\quad ,\quad 
\mathcal N=\mathcal X\oplus\mathcal Z=\mathbb R^{k}\times\mathbb R^l.
\end{equation}
Spaces $\mathcal Z$ and $\mathcal X$ are called also Z- and
X-space, respectively. 

Upto isometric isomorphisms, such a Lie algebra is uniquely
determined by the linear space, $J_{\mathcal Z}$, of skew endomorphisms 
$J_Z:\mathcal X\to\mathcal X$ defined for any
$Z\in\mathcal Z$ by the formula 
\begin{equation}
\label{brack}
\langle [X,Y],Z\rangle =\langle J_Z(X),Y\rangle , 
\forall Z\in\mathcal Z.
\end{equation}
This statement means that for an orthogonal direct sum, 
$\mathcal N=\mathcal X\oplus\mathcal Z=\mathbb R^{k}\times\mathbb R^l$,
of Euclidean spaces a non-degenerated linear map, 
$\mathbb J:\mathcal Z\to SE(\mathcal X)\, ,\, Z\to J_Z$,
from the Z-space into the space of skew endomorphisms acting on the 
X-space, defines a 2-step nilpotent metric Lie algebra on $\mathcal N$ by 
(\ref{brack}). Furthermore, an other non-degenerated linear map 
$\tilde{\mathbb J}$ having the same range
$\tilde J_{\mathcal Z}=J_{\mathcal Z}$ as 
$\mathbb J$ defines isometrically isomorphic Lie algebra. 

By means of the exponential map, also
the group can be considered such that it is defined on $\mathcal N$.
That is, a point is denoted by $(X,Z)$ on 
the group as well. Then, the group multiplication is given
by the formula
$(X,Z) (X^\prime ,Z^\prime)=(X+X^\prime ,Z+Z^\prime 
+{1\over 2}[X.X^\prime ])$.  
Metric tensor, $g$, is defined by the left invariant extension of 
$\langle ,\rangle$ onto the group $N$.

Endomorphisms $J_Z(.)$ will be associated with angular momenta. It must
be pointed out, however,
a major conceptual difference between the classical
3D angular momentum, introduced in (\ref{3Dangmom}), and this new sort
of angular momentum.
For a fixed axis $Z\in \mathbb R^3$
the endomorphism associated with the classical 3D angular momentum
is defined with the help of the cross product $\times$ by
the formula $J_Z: X\to Z\times X$. That is, 
axis $Z$ is lying in the same
space, $\mathbb R^3$, where the endomorphism itself is acting.   
Linear map $\mathbb J:\mathcal Z\to SE(\mathcal X)$ on a 2-step nilpotent
Lie group, however, separates 
axis $Z\in\mathcal Z$ from the X-space where the
endomorphism $J_Z(.)$ is acting.
In other words, the latter endomorphism defines just orbiting of
a position vector $X$ in the plane spanned 
by $X$ and $J_Z(X)$, but the axis of
orbiting is not in the X- but in the Z-space. 

In this respect, the Z-space
is the abstract space of the axes associated with the angular momentum
endomorphisms. According to this interpretation, 
for a fixed axis $Z$ in the Z-space, a particle occupies a 
complex plane in the complex space
defined by the complex structure $J_Z$ on the X-space. 
Abstract axis, $Z$, is considered
as an ``inner dial" associated with the particles, which is represented
separately in the Z-space. This Z-space contributes new dimensions to 
the X-space which is considered as the inner-world supplemented to the
exterior-world in order to have a natural stage for describing the inner
physics of elementary particles.   
 
The above definition of 2-step nilpotent Lie groups
by their endomorphism spaces $J_{\mathcal Z}$ shows the large variety
of these groups. For instance, if $\mathcal Z$ is an $l$-dimensional 
Lie algebra of a compact group and  
$\mathbb J:\mathcal Z\to so(k)$ (which corresponds $J_Z\in so(k)$
to $Z\in \mathcal Z$) is its representation
in a real orthogonal Lie algebra (that is, in the Lie algebra of
skew-symmetric matrices) defined for $\mathcal X=\mathbb R^k$,
then the system 
$\{\mathcal N=\mathcal X\oplus \mathcal Z, J_{\mathcal Z}\}$
defined by orthogonal direct sum determines a unique 2-step nilpotent
metric Lie algebra where the inner product on $\mathcal Z$ is 
defined by $\langle Z,V\rangle =-Tr(J_{Z}\circ J_V)$.
Thus,
any faithful representation $\mathbb J:\mathcal Z\to so(k)$
determines a unique
two-step nilpotent metric Lie algebra. Since Lie algebras 
$su(k/2)\subset so(k)$ used in non-Abelian gauge theories 
are of
compact type, therefore, to any of their representations in orthogonal
Lie algebras, one can associate a natural two-step nilpotent metric 
Lie algebra. This association is the natural bridge between a non-Abelian
$SU(n)$-theory and the new theory developed in this paper.

The 2-step nilpotent Lie groups form even a much 
larger class than those constructed above 
by orthogonal Lie algebra representations. In fact, in the most general
situation, endomorphism space $J_{\mathcal Z}$ is just
a linear space defined by the range of a non-degenerated linear map
$\mathbb J:\mathcal Z\to SE(\mathcal X)$ 
which may not bear any kind of Lie algebra structure. 
However, such general groups
can be embedded into those constructed by the 
compact Lie algebras, $J_{\tilde{\mathcal Z}}$. The smallest such
Lie algebra is generated by the endomorphism space  
$J_{\mathcal Z}$ by the Lie brackets.        

Very important particular 2-step nilpotent Lie groups are the
Heisenberg-type groups, introduced by Kaplan \cite{k}, which are
defined by endomorphism spaces $J_{\mathcal Z}$
satisfying the Clifford condition $J^2_Z=-\mathbf z^2id$, where 
$\mathbf z=|Z|$ denotes the length of the corresponding vector. 
These groups are attached to Clifford modules 
(representations of Clifford algebras). 
The well known classification 
of these modules provides  
classification also for the Heisenberg-type groups. 
According to this classification, 
the X-space and the endomorphisms
appear in the following form: 
\begin{equation}
\mathcal X=
(\mathbb R^{r(l)})^a\times (\mathbb{R}^{r(l)})^b\, ,\,
J_{Z} =
(j_{Z} \times\dots\times j_Z) \times (-j_Z\times\dots\times -j_Z),
\end{equation}
where $l=dim(\mathcal Z)$ and the endomorphisms $j_Z$ act on the 
corresponding component,
$\mathbb R^{r(l)}$, of this Cartesian product. The
groups and the corresponding natural metrics are denoted by 
$H^{(a,b)}_l$ and $g^{(a,b)}_l$ respectively.  
Particularly important examples are the 
H-type groups $H^{(a,b)}_3$, where the
3-dimensional Z-space, $\mathbb R^3$, 
is considered as the space of imaginary 
quatrnions, furthermore, action of $j_Z$ on the space
$\mathbb R^{r(3)}=\mathbb H=\mathbb R^4$
of quaternoinic numbers is defined by left multiplications with $Z$. 

A brief account  on the classification of Heisenberg 
type groups is as follows.
If $l=dim(J_{\mathcal Z})\not =3\mod 4$, then, upto equivalence, 
there exist exactly one
irreducible H-type endomorphism space acting 
on a Euclidean space $\mathbb R^{n_l}$,
where the dimensions $n_l$, which depend just on $l$, 
are described below. These endomorphism spaces
are denoted by $J_l^{(1)}$. If $l=3\mod 4$, then, upto equivalence, 
there exist exactly
two non-equivalent irreducible H-type endomorphism spaces acting on
$\mathbb R^{n_l}$. They are denoted by 
$J_l^{(1,0)}$ and
$J_l^{(0,1)}$ 
respectively. They relate to each other by the relation 
$J_l^{(1,0)}\simeq -J_l^{(0,1)}$. 

The values $n_l$ corresponding to
$
l=8p,8p+1,\dots ,8p+7
$
are

\begin{eqnarray}
n_l=2^{4p}\, ,\, 2^{4p+1}\, , \, 2^{4p+2}\, , \,
2^{4p+2}\, ,
\, 2^{4p+3}\, ,\, 2^{4p+3}\, , \, 2^{4p+3}\, , \,
2^{4p+3}.
\label{cliff}
\end{eqnarray}

The reducible Clifford endomorphism spaces can be built up by these
irreducible ones. They are denoted by 
$J_l^{(a)}$ resp. $J_l^{(a,b)}$.
The corresponding Lie algebras are denoted by 
$\mathcal H^{(a)}_r$ and
$\mathcal H^{(a,b)}_l$ respectively, 
which define the groups 
$H^{(a)}_r$ resp. 
$H^{(a,b)}_l$. 
In the latter case, the X-space 
is defined by the $(a+b)$-times product 
$\mathbb R^{n_l}\times\dots\times\mathbb R^{n_l}$
such that, on the last $b$ component, the action of a $J_Z$ is defined by 
$J_Z^{(0,1)}\simeq
-J_Z^{(1,0)}$,
and, on the first $a$ components, the action is defined by
$J^{(1,0)}_Z$. In the first case this process should be applied only on 
the corresponding $a$-times product.  

One of the fundamental statements in this theory is that, in case of 
$l=3\mod 4$, two groups  $H^{(a,b)}_l$ and $H^{(a^\prime ,b^\prime )}_l$
are isometrically isomorphic if and only if $(a,b)=(a^\prime ,b^\prime )$
upto an order. By a general statement, two metric
2-step nilpotent Lie groups with Lie algebras
$\mathcal N=\mathcal X\oplus\mathcal Z$ and
$\mathcal N^\prime =\mathcal X^\prime\oplus\mathcal Z^\prime$
are isometrically isomorphic if and only if there exist orthogonal
transformations $A :\mathcal X\to\mathcal X^\prime$ and
$B :\mathcal Z\to\mathcal Z^\prime$
such that $J_{B (Z)}=A\circ J_Z\circ A^{-1}$
holds, for all $Z\in\mathcal Z$. The isomorphic isometry between
$H^{(a,b)}_l$ and $H^{(b,a)}_l$ is defined by $A =id$ and 
$B =-id$. If
$(a,b)\not =(a^\prime ,b^\prime )$ (upto an order) then the
corresponding groups are not isometrically isomorphic.  

In order to unify the two cases, denotations
$J_l^{(1,0)}=J_l^{(1)}$ and
$J_l^{(0,1)}=-J_l^{(1)}$ are used also in cases
$l\not =3\mod 4$. One should keep in mind, however, that these 
endomorphism spaces are equivalent, implying that 
two groups  $H^{(a,b)}_l$ and $H^{(a^\prime ,b^\prime )}_l$ 
defined by them
are isometrically isomorphic if and only if 
$a+b=a^\prime +b^\prime$ holds.

H-type groups can be characterized as being such particular 2-step metric 
nilpotent Lie groups on which
the skew endomorphisms, $J_Z$, for any fixed $Z\in\mathcal Z$, 
have the same eigenvalues $\pm \mathbf z\mathbf i$.
By polarization we have:
\begin{equation}
\frac{1}{2}(J_{Z_1}J_{Z_2}+ J_{Z_2}J_{Z_1})=-\langle Z_1,Z_2\rangle Id, 
\end{equation}
which is called Dirac's anticommutation equation. It implies
that two endomrphisms, $J_{Z_1}$ and $J_{Z_2}$, with
perpendicular axes, $Z_1\perp Z_2$, 
anticommute with each other. Endomorphism
spaces satisfying this weaker property define more general, so called
totally anticommutative 2-step nilpotent Lie groups on which the 
endomorphisms can have also properly distinct eigenvalues. 
The classification
of these more general groups is unknown in the literature. 

Let it be mentioned, yet, that groups constructed above by 
$su(2)$-repre\-sen\-tations  
are exactly the groups $H^{(a,b)}_3$, while those constructed 
by $su(3)$-rep\-re\-sen\-ta\-tions are 
not even totally commutative spaces. 
Thus, they are
not H-type groups either. It is also noteworthy, that the Cliffordian 
endomorphism spaces $J_l^{(a,b)}$ do not form a Lie algebra in general.
In fact, only the endomorphism spaces defined for $l=1,3,7$ can
form a Lie algebra. In the first two cases, pair $(a,b)$ can be arbitrary,
while in case $l=7$ only cases $a=1,b=0$, or, $a=0,b=1$ yield Lie algebras.
It is an interesting question that which orthogonal Lie algebras can be
generated by Cliffordian endomorphism spaces $J_l{(a,b)}$.
Let it be mentioned, yet, that Lie algebra
$su(3)$ does not belong even to this category.

\subsection{Laplacian and curvature.}

Although most of the results of this paper extend to the most general 
2-step metric nilpotent Lie groups, in what follows only H-type groups
will be considered. 
On these groups, the Laplacians appear in the following form:
\begin{equation}
\label{Delta}
\Delta=\Delta_X+(1+\frac 1{4}\mathbf x^2)\Delta_Z
+\sum_{\alpha =1}^r\partial_\alpha D_\alpha \bullet,
\end{equation}
where $D_\alpha\bullet$ denotes directional derivatives along
the vector fields 
$X\to J_\alpha (X)=J_{e_\alpha}(X)$, furthermore, $\mathbf x=|X|$
denotes the length of X-vectors.

This formula can be established by the following explicit formulas. 
Consider orthonormal bases $\big\{E_1;\dots;E_k\}$ and 
$\big\{e_1;\dots;e_l\big\}$ on the X- and Z-space respectively.
The coordinate systems defined by them are 
denoted by
$\big\{x^1;\dots;x^k\big\}$ and $\big\{z^1;\dots;z^l\big\}$ respectively.
Vectors $E_i;e_{\alpha}$ extend into the left-invariant 
vector fields
\begin{eqnarray}
\label{invar-vect}
\mathbf
X_i=\partial_i + \frac 1 {2} \sum_{\alpha =1}^l
\langle [X,E_i],e_{\alpha}\rangle  \partial_\alpha = 
\partial_i + \frac 1 {2} \sum_{{\alpha} =1}^l \langle
J_\alpha\big(X\big),E_i\rangle
\partial_\alpha 
\end{eqnarray}
and $\mathbf Z_{\alpha}=\partial_\alpha$, respectively,
where $\partial_i=\partial /\partial x^i$, $\partial_\alpha
=\partial/\partial z^\alpha$
and $J_\alpha = J_{e_\alpha}$.

The covariant derivative acts on these
invariant vector fields according to the following formulas.
\begin{equation}
\label{nabla}
\nabla_XX^*=\frac 1 {2} [X,X^*]\quad ,
\quad\nabla_XZ=\nabla_ZX=-\frac 1 {2}
J_Z\big (X\big)\quad ,\quad\nabla_ZZ^*=0.
\end{equation}

The Laplacian, $\Delta$, acting on functions can 
explicitly be established by substituting (\ref{invar-vect}) and 
(\ref{nabla}) into
the following well known formula
\begin{equation}
\label{inv-delta}
\Delta=\sum_{i=1}^k\big(\mathbf X_i^2-
\nabla_{\mathbf X_i}\mathbf X_i\big)
+\sum_{\alpha =1}^l\big (\mathbf Z_{\alpha}^2-\nabla_{\mathbf Z_{\alpha}} 
\mathbf Z_{\alpha}\big ).
\end{equation}

These formulas allow to compute also the Riemannian curvature, $R$, on 
$N$ explicitly. Then we find:

\begin{eqnarray}
R(X,Y)X^*=\frac 1 {2} J_{[X,Y]}(X^*) - 
\frac 1 {4} J_{[Y,X^*]}
(X) + \frac 1 {4} J_{[X,X^*]}(Y); \\
R(X,Y)Z=-\frac 1 {4} [X,J_Z(Y)]+
\frac 1 {4} [Y,J_Z(X)] ;
\quad R(Z_1,Z_2)Z_3=0; \\
R(X,Z)Y=-\frac 1 {4} [X,J_Z(Y)] ; \quad 
R(X,Z)Z^*=-\frac 1 {4}
J_ZJ_{Z^*}(X); \\
R(Z,Z^*)X=-\frac 1 {4} J_{Z^*}J_Z(X) + \frac 1 {4}J_Z
J_{Z^*}(X),
\end{eqnarray}
where $X;X^*;Y \in \mathcal X$ and $Z;Z^*;Z_1;Z_2;Z_3 \in \mathcal Z$
are considered as the elements of the Lie algebra $\mathcal N$.
The components of this tensor field on coordinate systems
$\big\{x^1;\dots;x^k;z^1;\dots;z^l\big\}$ 
can be computed by formulas (\ref{invar-vect}).

\section{Particles without interior.}

These Riemann manifolds were used, originally \cite{sz1}-\cite{sz4}, 
for isospectrality constructions  
in two completely different situations. In the first
one, the Z-space is factorized by a Z-lattice, $\Gamma_Z$, defined on the
Z-space, which process results a torus bundle over the X-space. In the 
second case, Z-ball resp. Z-sphere bundles are considered 
by picking Z-balls resp. Z-spheres in the Z-space over the points of the
X-space. It turns out that, apart from a constant term, 
the Laplacian on the 
Z-torus bundles, called also Z-crystals, 
is the same as the classical Ginsburg-Landau-Zeeman operator
of an electron-positron system 
whose orbital angular momentum is expressed in terms of the endomorphisms
$J_{Z_\alpha}$ defined by the lattice points $Z_\alpha\in \Gamma_Z$.
More precisely, the Z-lattice defines a natural decomposition 
$\sum_\alpha W_\alpha$ of the $L^2$-function space
such that the components $W_\alpha$ are invariant under the action of 
the Laplacian, which, after restricting it onto a fixed $W_\alpha$, 
appears as the Ginsburg-Landau-Zeeman operator whose 
orbital angular momentum 
is associated with the fixed endomorphism $J_{Z_\alpha}$. 

It turns out, in the next chapters, that the constant
term corresponds to neutrinos (massless particles with no charge) 
accompanying an electron-positron system. Thus, 
altogether, the Laplacian appears as the Hamilton operator of a system
of electrons positrons and 
electron-positron-neutrinos which is formally the sum of
a Ginsburg-Landau-Zeeman operator and a constant term. Names 
Ginsburg-Landau
indicate that no Coulomb potential or any kind of electric forces are 
involved into this operator. Thus the forces which manifest themself
in the eigenfunctions of the Laplacian, 
are not the complete electromagnetic
forces, yet. However, when also these forces will be introduced, the 
eigenfunctions remain the same, the electric force contributes 
only to the magnitude of the eigenvalues. By this reason, 
the forces associated with these
models are called electromagnetic forces. Since the Z-lattices consist
of points and intrinsic physics is exhibited on the Z-space, particles
represented by Z-crystals are considered as point-like particles
having no insides. The theory developed for them is in the
strongest connection with quantum electrodynamics (QED).   

In the second case both the Laplacian 
and the angular
momentum operator appear in much more complex forms. Contrary
to the first case, the angular
momentum operator is not associated with a fixed Z-vector, but it
represents spinning about each Z-vector. 
Beside the orbiting spin, also natural 
inner operators emerge which can be 
associated both with weak and strong 
nuclear forces. The particles represented by these models do have 
inside on which the intrinsic physics described in QCD is exhibited
on a full scale. An openly admitted purpose
in the next sections is to give a unified theory for the 3 forces:
1.) electromagnetic- 2.) weak-nuclear- 3.) strong-nuclear-forces. 
However, some concepts such as Dirac's spin 
operator will be introduced in a 
subsequent paper into this theory.
    
\subsection{Z-crystals modelling Ginsburg-Landau-Zeeman operators.}

The Z-torus bundles are defined by a factorization,  
$\Gamma\backslash H$, of the nilpotent group $H$ by a Z-lattice,
$\Gamma=\{Z_\gamma\}$, which is defined  
only on $\mathcal Z$ and not on the whole $(X,Z)$-space. 
Such a factorization defines a Z-torus bundle over the X-space.
The natural Z-Fourier decomposition,
$L^2_{\mathbb C}:=\sum_\gamma W_\gamma$, of the $L^2$ function space 
belonging to this bundle is defined such that subspace
$W_\gamma$ is spanned by functions of the form
\begin{equation}
\label{diskFour}
\Psi_\gamma (X,Z)=\psi (X)e^{2\pi\mathbf i\langle Z_\gamma ,Z\rangle}.
\end{equation}
Each $W_\gamma$ is invariant under the action of $\Delta$, more precisely
we have:
\begin{eqnarray}
\Delta \Psi_{\gamma }(X,Z)=(\lhd_{\gamma}\psi )(X)
e^{2\pi\mathbf i\langle Z_\gamma ,Z\rangle},\quad
{\rm where}
\\
\lhd_\gamma
=\Delta_X + 2\pi\mathbf i D_{\gamma }\bullet 
-4\pi^2\mathbf z_\gamma ^2(1 + \frac 1 {4} \mathbf x^2).
\end{eqnarray}
In terms of parameter $\mu =\pi \mathbf z_\gamma $, 
this operator is written
in the form
$
\lhd_{\mu}=
\Delta_X +2 \mathbf i D_{\mu }\bullet -\mu^2\mathbf x^2-4\mu^2.
$
Although it is defined in terms of the X-variable,
this operator is not a sub-Laplacian resulted by a submersion. 
It rather is the restriction of the total Laplacian onto 
the invariant subspace $W_\gamma$.
Actually, the Z-space is represented by the constant $\mu$
and operator $D_\mu\bullet$. A characteristic feature of this 
restricted operator is that it involves only a single endomorphism,
$J_{Z_\gamma}$. 

In the 2D-case, such an operator can be transformed 
to the the Ginsburg-Landau-Zeeman
operator (\ref{land}) by choosing $\mu ={eB/2\hbar c}$ and  
multiplying the whole operator with 
$-{\hbar^2/2m}$.
In general dimensions,
number $\kappa =k/2$ means the {\it number of particles}, and, 
endomorphisms $j_Z$ and $-j_Z$ in the above formulas are attached to 
systems electrons resp. positrons. More precisely, by the classification 
of H-type groups, these endomorphisms are acting on the irreducible
subspaces $\mathbb R^{n_l}$ and the system is interpreted such that 
there are $n_l/2$ particles of the same charge orbiting on complex planes
determined by the complex structures $j_{Z_u}$ resp. $-j_{Z_u}$
in constant magnetic 
fields whose directions are perpendicular to the complex planes
where the particles are orbiting. 
The actuality of the
complex planes where the orbiting takes place can be determined just 
probabilistically by the probability amplitudes defined for such systems.
The total number of particles is $\kappa =(a+b)n_l/2$. The probability
amplitudes must refer to $\kappa$ number of particles, that is, they
are defined on the complex X-space $\mathbb C^\kappa$ defined by the
complex structure $J_{Z_u}$. This theory
can be established just after developing an adequate spectral theory.

Above, adjective "perpendicular" is meant to be just symbolic, 
for axis $Z$ is actually separated
from the orbiting. That is, it is not the actual axis of orbiting,
but it is a vector in the Z-space which is attributed to the orbiting
by the linear map $\mathbb J:Z\to J_Z$.  
Thus, it would be more appropriate
to say that this constant perpendicular 
magnetic field, $B$, is just "felt" 
by the particle orbiting on a 
given complex plane. Anyhow, the $B$ defines
a unique inertia system on the complex 
plane in which $E=0$, that is, the
associated electric field vanishes. In all of the other inertia system
also a non-zero $E$ must be associated with $B$. The relativistic time
$T$ defined on this unique inertia system is the inner time defined 
for the particle orbiting on a complex plane. 
This time can be synchronized, meaning that
common time $T$ can be introduced which defines time on 
each complex plane.     

Note that this operator contains
also an extra constant term, $4\mu^2$, which is explained later
as the total energy of neutrinos accompanying 
the electron-positron system. 
This energy term is neglected in the original
Ginsburg-Landau-Zeeman Hamiltonian. 
Thus the higher dimensional mathematical model 
really represents a system of particles and antiparticles
which are orbiting in constant
magnetic fields. Operator, $D_\mu\bullet$, associated with magnetic
dipole resp. angular momentum operators, 
are defined for the lattice points
separately. Therefore this model can be viewed such that it is 
associated with 
{\it magnetic-dipole-moment-crystals}, or, 
{\it angular-moment-crystals}. In short, they are called Z-crystals. 
They are particularly interesting 
on a group $H^{(a,b)}_3$ where the Z-space is $\mathbb R^3$. On this 
Euclidean space all
possible crystals are well known by classifications. 
It would be interesting to know what does this mathematical
classification mean from physical point of view and if these Z-crystals
really exist in nature?

\subsection{Explicit spectra of Ginsburg-Landau-Zeeman operators.}
 
On Z-crystals, the spectral 
investigation of the total operator (\ref{Delta}) can be
reduced to the operators $\lhd_\gamma$, induced by $\Delta$ 
on the invariant
subspaces $W_\gamma$.   
On Heisenberg-type groups this operator involves only a single parameter 
$\mu >0$, where $\mu^2$ is the single eigenvalue of 
$-J_\gamma^2$. By this reason, it is denoted by $\lhd_\mu$. 

This problem
is traced back to an ordinary differential operator 
acting on radial functions, which can be found by seeking 
the eigenfunctions in the form
$F(X)=f(\langle X,X\rangle )\mathtt H^{(n,m )}
(X)$, 
where $f$ is an even function 
defined on $\mathbb R$ and
$\mathtt H^{(n,m)}(X)$ is a complex valued
homogeneous harmonic polynomial of order
$n$, and, simultaneously, it is also an eigenfunction 
of operator
$\mathbf iD_\mu\bullet$
with eigenvalue $m\mu$. 
Such polynomials can be constructed as follows.

Consider a complex orthonormal basis, 
$\mathbf B=\{B_1,\dots ,B_{\kappa}\}$,
on the complex space defined by the complex structure 
$J=(1/\mu)J_\mu$. The corresponding
complex coordinate system is denoted by
$\{z_1,\dots ,z_{\kappa}\}$.
Functions
$
P=z^{p_1}_1\dots z^{p_{\kappa}}_{\kappa}
\overline z^{q_1}_1\dots \overline z^{q_{\kappa}}_{\kappa}
$
satisfying
$p_1+\dots +p_{\kappa}=p,\, 
q_1+\dots +q_{\kappa}=n-p$
are $n^{th}$-order homogeneous polynomials which are eigenfunctions
of $\mathbf iD\bullet$ with eigenvalue $m =2p-n$.
However, these polynomials are not harmonic. 
In order to get the harmonic
eigenfunctions, they must be exchanged for the 
polynomials
$
\Pi^{(n)}_X (P)
$,
defined by projections, $\Pi^{(n)}_X$, onto the space of 
$n^{th}$-order 
homogeneous harmonic 
polynomials of the X-variable. By their explicit description (\ref{proj}),
these projections are of the form 
$\Pi^{(n)}_X =\Delta_X^0+B^{(n)}_1\mathbf x^2\Delta_X +
B^{(n)}_2\mathbf x^4\Delta_X^2+\dots$,
where $\Delta_X^0=id$. By this formula,
also the harmonic polynomial obtained by this 
projection is an eigenfunction of $\mathbf iD\bullet$ 
with the same eigenvalue $m\mu$. 

When operator $\lhd_\mu$ is acting on 
$F(X)=f(\langle X,X\rangle )\mathtt H^{(n,m)}(X)$, 
it defines an ordinary differential operator acting on $f$. Indeed,
by $D_{\mu }\bullet f=0$, we have: 

\begin{eqnarray}
(\lhd_{\mu }F)(X)=\big(4\langle X,X\rangle f^{\prime\prime}
(\langle X,X\rangle )
+(2k+4n)f^\prime (\langle X,X\rangle )\\
-(2m\mu +4\mu^2((1 +{1\over 4}\langle X,X\rangle )
f(\langle X,X\rangle ))\big)\mathtt H^{(n,m)} (X).
\nonumber
\end{eqnarray}
The eigenvalue problem can, therefore, be reduced to an ordinary
differential operator. More precisely, we get:
\begin{theorem}
On a Z-crystal, $B_R\times T^l$, under a given
boundary condition $Af^\prime (R^2)+Bf(R^2)=0$ defined by constants 
$A,B\in\mathbb R$, 
the eigenfunctions
of $\lhd_{\mu }$ can be represented
in the form $f(\langle X,X\rangle )\mathtt H^{(n,m)} (X)$, where the radial
function $f$ is an eigenfunction of the ordinary differential operator
\begin{equation}
\label{Lf_lambda}
(\large{\Diamond}_{\mu,\tilde t}f)(\tilde t)=
4\tilde tf^{\prime\prime}(\tilde t)
+(2k+4n)f^\prime (\tilde t)
-(2m\mu +4\mu^2(1 +
{1\over 4}\tilde t))f(\tilde t).
\end{equation}
For fixed degrees $n$ and $m$, the multiplicity 
of such an eigenvalue is the
dimension of space formed by the spherical harmonics 
$\mathtt H^{(n,m)} (X)$.

In the non-compact case,
when the torus bundle is considered over the whole X-space, 
the $L^2$-spectrum can explicitly be computed. 
Then, the above functions are sought
in the form 
$f_\mu (\tilde t)=u(\tilde t\mu)e^{-\tilde t\mu /2}$ 
where $u(\tilde t)$ is a uniquely determined $r^{th}$-order polynomial 
computed for $\mu =1$.
In terms of these parameters, the elements of the spectrum are
$\nu_{(\mu,r,n,m)}=-((4r+4p+k)\mu+4\mu^2)$. This
spectrum depends just on $p=(m+n)/2$ and the same 
spectral element appears
for distinct degrees $n$. Therefore, the multiplicity of each
eigenvalue is infinity.
\end{theorem}

\begin{proof}
Only the last statement is to be established. We proceed, first, with the
assumption $\mu =1$.
Then, function 
$e^{-{1\over 2}\tilde t}$ 
is an eigenfunction of this radial
operator with eigenvalue $-(4p+k+4)$. The general
eigenfunctions are sought in the form
\begin{equation}
f(\tilde t)=u(\tilde t)e^{-{1\over 2}\tilde t}.
\end{equation}
Such a function is an eigenfunction of 
$\Diamond$
if and only if $u(\tilde t)$ is an eigenfunction of the differential 
operator
\begin{equation}
\label{Pop}
(P_{(\mu =1,n,m )}u)(\tilde t)=
4\tilde tu^{\prime\prime}(\tilde t)
+(2k+4n-4\tilde t)u^\prime (\tilde t)
-(4p+k+4)u(\tilde t).
\end{equation}
Because of differentiability conditions, 
we impose $u^\prime (0)=0$ on the
eigenfunctions. Since in this case, $u(0)\not =0$ hold for any non-zero
eigenfunction, also the condition $u(0)=1$ is imposed.

In the compact case, corresponding to the ball$\times$torus-type manifolds
defined over a ball $B_R$, the spectrum of this Laguerre-type differential
operator can not be explicitly computed. For a given boundary condition  
(which can be Dirichlet, $u (R)=0$, or Neumann, 
$u^\prime (R)=0$) the spectrum
consists of a real sequence $0\leq \mu_1>\mu_2>\dots \to -\infty$.
The multiplicity of each of these Laguerre-eigenvalues is $1$ and the 
multiplicity corresponding to the Ginsburg-Landau-Zeeman operator 
is the dimension
of the space of spherical harmonics $\mathtt H^{(n,m)}$. 
The elements of the Laguerre-spectrum are zeros
of a holomorphic function expressed by an integral formula \cite{coh}.

Contrary to the compact case, the spectrum 
can be explicitly computed for the non-compact torus bundle
$\Gamma\backslash H$, defined over the whole X-space. 
An elementary argument shows that for any $r\in \mathbb N$, operator 
(\ref{Pop}) has a
uniquely determined polynomial eigenfunction
\begin{equation}
\label{lag}
u_{(\mu =1,r,n,m )}(\tilde t)
=\tilde t^r+a_1\tilde t^{r-1}+a_2\tilde t^{r-2}+\dots +
a_{r-1}\tilde t+a_r
\end{equation}
with coefficients satisfying the recursion formulas
\begin{equation}
a_0=1\quad ,\quad a_i=-a_{i-1}(r-i)(r+n+{1\over 2}k+1-i)r^{-1}.
\end{equation}
 
Actually, this argument can be avoided and these polynomials can 
explicitly be established by observing
that they are  nothing but 
the Laguerre polynomials defined as the $r^{th}$-order polynomial
eigenfunctions
of operator
\begin{equation}
\label{Lop}
\Lambda_{\alpha} (u)(\tilde t)=\tilde tu^{\prime\prime}+
(\alpha +1-\tilde t)u^\prime ,
\end{equation}
with eigenvalues $-r$.
This statement follows from identity
\begin{equation}
P_{(\mu =1,n,m )}=4\Lambda_{({1\over 2}k+
n-1)}-(4p+k+4),
\end{equation}
implying that
the eigenfunctions of operators 
(\ref{Pop}) and (\ref{Lop}) 
are the same and 
the eigenvalue corresponding to (\ref{lag}) is
\begin{equation}
\label{leigv}
\nu_{(\mu =1,r,n,m )}=-(4r+4p+k+4)\quad ,
\quad p={1\over 2}(m +n).
\end{equation}
We also get that, for fixed values of $k,n,m$ 
(which fix the value also for $p$), functions 
$
u_{(\mu =1,r,n,m )}\, ,\, n=0,1,\dots \infty
$ form a basis in $L^2([0,\infty ))$.

In case of a single $\mu$, 
the eigenfunctions are sought in the form
\begin{equation}
\label{eigfunc}
u_{\mu r nm}(\langle X,X\rangle )e^{-{1\over 2}\mu 
\langle X,X\rangle}\mathtt H^{(n,m)}(X).
\end{equation}
It turns out that
\begin{equation}
u_{(\mu ,r,n,m)}(\tilde t)=
u_{(\mu =1 ,r,n,m)}(\mu \tilde t)
\end{equation}
and the corresponding eigenvalue is
\begin{equation}
\label{eigval}
\nu_{(\mu,r,n,m)}=-((4r+4p+k)\mu+4\mu^2).
\end{equation}

This statement can be explained as follows. For a general $\mu$, 
the action of (\ref{Lf_lambda}) on a function 
$f(\tilde t)=u(\mu \tilde t)e^{-{1\over 2}\mu \tilde t}$ can be described
in terms of $\tau =\mu \tilde t$ as follows:
\begin{equation}
\label{Lf_lambda}
(L_{(\mu ,n,m )}f=\mu (4\tau f_{\tau\tau}+(2k+4n)f_{\tau}
-(2m+\tau )f)-4\mu^2f,
\end{equation}
from which the statement follows.
\end{proof}

This technique extends to general 2-step nilpotent Lie groups, where
the endomorphisms may have distinct eigenvalues, 
$\{\mu_i\}$. In this case 
the eigenfunctions are represented
as products of functions of the form
\begin{equation}
F_{(i)}(X)=f_{(i)}(\langle X,X\rangle)
\mathtt H^{(n_i,m_i )}(X),
\end{equation}
where the functions in the formula are defined on the maximal
eigensubspace corresponding to the parameter $\mu_i$. When the 
spectrum is computed on the whole X-space, this method works out for the
most general Ginsburg-Landau-Zeeman operators. In case of a single $\mu$,
this method applies also to computing the spectra on torus 
bundles over balls and spheres.
In case of multiple $\mu$'s, it applies to torus bundles over 
the Cartesian product
of balls resp. spheres defined on the above $X_i$-spaces.

\section{Particles having interior.}

In order to sketch up a clear map for this rather complex section, 
we start with a review of the main mathematical and physical ideas 
this exposition is based on. These ideas are rigorously established 
in the subsequent subsections.
\subsection{A preliminary review of the main ideas.}
Systems of particles having insides can be attached to 
ball$\times$ball- and ball$\times$\-sphe\-re-type 
manifolds. Originally they emerged in the second type of spectral 
investigations performed in \cite{sz2}-\cite{sz4}. These manifolds 
are defined by appropriate smooth fields of Z-balls resp. Z-spheres 
of radius $R_Z(\mathbf x)$ 
over the points of a fixed X-ball $B_X$ whose radius is denoted 
by $R_X$. Note that radius $R_Z(\mathbf x)$ depends just on the
length, $\mathbf x:=|X|$, of vector $X\in B_X$ over which the 
Z-balls resp. Z-spheres are considered. The centers all
of the balls resp. spheres which show up in this definition are 
always at the origin of the corresponding spaces. The boundaries
of these manifolds are the so called
sphere$\times$ball- resp. sphere$\times$sphere-type manifolds, 
which are trivial Z-ball-
resp. Z-sphere-bundles defined
over fixed X-spheres of radius $R_X$.
In short, one considers Z-balls resp. Z-spheres instead of the 
Z-tori used in the previous constructions of Z-crystals. In the
isospectrality investigations these compact domains
corresponding to $R_X <\infty$ play the 
primary interest.
In physics, however, the non-compact bundles corresponding to 
$R_X =\infty$ (that is, which are defined over the whole X-space)
become the most important cases. In what follows, both the 
compact and non-compact cases will be investigated.

Contrary to the Z-crystal models, 
the computations in this case can not be reduced to a single 
endomorphism. Instead, they always have to be established for the complete
operator $\mathbf M=\sum\partial_\alpha D_\alpha\bullet$ which includes
the angular momentum endomorphisms $J_Z$ with respect to any Z-directions.
This operator strongly relates both to the
3D angular momentum 
$\mathbf P=(P_1,P_2,P_3)=\frac{1}{\hbar}Z\times\mathbf p$, 
defined in (\ref{3Dang_mom}), and the strong interaction 
term $\mathbf igx^\mu A^\alpha_\mu t_\alpha$ 
of the QCD Lagrangian (\ref{QCDLagrange}). Actually, it has a rather
apparent formal identity with Pauli's intrinsic spin Hamiltonian
$\sum B_iP_i$ defined by magnetic fields $\mathbf B=(B_1,B_2,B_3)$
(cf. \cite{p}, Volume 5, pages 152-159).
In $SU(3)$-theory the term corresponding to the 3D angular momentum
is exactly the above mentioned strong interaction term.
Due to the new form 
(\ref{newDiracwave}) of the de Broglie waves,
where the angular-momentum-axes are separated from the planes on which 
the particles are orbiting, also the angular momentum operator 
has to appear in a new form. 

In order to make the argument about the analogy with Pauli's
intrinsic spin resp. strong interaction term more clear, note 
that the Lie algebras determined by the 3D angular momenta 
resp. Gell-Mann's matrices $t_\alpha$ are $su(2)$ resp. $su(3)$. Thus, the
nilpotent-group-models corresponding to these classical Yang-Mills
models are $H^{(a,b)}_3$ resp. the
nilpotent group constructed by the representations of $su(3)$. 
When operator $\mathbf M$ acts on wave function
(\ref{newDiracwave}), it appears behind the integral in the form 
$\mathbf iD_K\bullet$. Consider, first, group $H^{(a,b)}_3$ and suppose
that $K=e_1$, where $\{e_1,e_2,e_3\}$ is the natural basis on 
$\mathbb R^3$. Then, on the $\{e_2,e_3\}$-plane, which is a complex plane
regarding the complex structure $J_{e_1}$,
operator $\mathbf iD_{e_1}\bullet$ is nothing but the first component,
$-P_1=\mathbf i\big(Z_2\frac{\partial}{\partial Z_3}
-Z_3\frac{\partial}{\partial Z_2}\big)$, 
of Pauli's angular momentum operator. That is, the analogy
between $\mathbf M$ and the classical angular momentum operator 
becomes apparent after letting $\mathbf M$ 
act on wave functions (\ref{newDiracwave}). This action is the very
same how $P$ is acting on the original de Broglie waves. Thus this 
new form 
of action, which can be described by axis-separation and 
placing the orbiting particles onto the complex planes, can really 
be considered as adjustment to the 
new forms of the wave functions. These arguments
work out also for groups constructed by $su(3)$-representations. Thus
it is really justified to consider $\mathbf M$ as a spin operator 
appearing in a new
situation. The greatest advantage of this new form is 
that it describes also the strong nuclear forces.

This complication gives rise to a
much more complex mathematical and physical situation 
where both the 
exterior and the interior life of particle systems exhibit themself
on a full scale. First, let the physical role of the Fourier 
transforms appearing in the formulas be clarified. If term involving time
is omitted from the formula of wave functions, the rest is called
timeless probability amplitude. Both these amplitudes and wave functions
could have been defined also by means of the inverse function
$e^{-\mathbf i\langle Z,K\rangle}$. Wave functions (or amplitudes) 
obtained from
the very same function by using $e^{\mathbf i\langle Z,K\rangle}$ resp. 
$e^{-\mathbf i\langle Z,K\rangle}$ in their 
Fourier transforms are said to be wave functions (or amplitudes)
defined for particle- resp. antiparticle-systems. In other words,
the definition of particles and antiparticles is possible because
of these two choices. Calling one object particle and its
counter part antiparticle is very similar to naming one of the poles 
of a magnet north-pole and the other one south pole. 
Since the Laplace operators on 2-step nilpotent Lie groups are defined 
by means of constant magnetic
fields, this is actually the right physical explanation for choosing  
$e^{\mathbf i\langle Z,K\rangle}$ or 
$e^{-\mathbf i\langle Z,K\rangle}$ to introduce probability amplitudes. 
By this reason, $\mathbf M$ is called unpolarized magnetic dipole moment
or angular momentum operator. The polarized operators appear behind 
the integral sign of the Fourier integral formula when $\mathbf M$ is 
acting on the formula. 

The complexity of this operator is fascinating. For instance, it  
is the sum of extrinsic, $\mathbf{L}$, and intrinsic, $\mathbf{S}$, 
operators which do not commute with each other. Furthermore, operator
$\OE =\Delta_X+(1+\frac 1{4}\mathbf x^2)\Delta_Z+\mathbf{L}$
is a Ginsburg-Landau-Zeeman operator which exhibits just orbital spin. The
intrinsic life of particles is encoded into $\mathbf{S}$. 
Also the strong nuclear forces, keeping the 
particles having interior together, can be explained by this operator. 
In order to make this 
complicated situation as clear as possible, we describe, in advance, 
how certain 
eigenfunctions of $\Delta$ can explicitly be computed. These 
computations provide a great opportunity also for a preliminary review
of the general eigenfunction computations, which will be connected
to these particular computations as follows.

Since these particular functions
do not satisfy the boundary conditions, they do not
provide the final solutions for finding the eigenfunctions
yielding also given boundary conditions.  This conditions can be imposed
just after certain projections performed on the center (that is, in
the insides of the particles). But then, these projected functions 
will not be eigenfunctions of the complete $\Delta$ any more. 
They are eigenfunctions
just of the exterior operator $\OE$. By this reason,  
they are called weak force eigenfunctions. The strong force 
eigenfunctions, defined by the eigenfunctions of the complete $\Delta$
satisfying also a given boundary condition, can be expressed just
by complicated combinations of the weak force eigenfunctions, 
meaning that the strong forces are piled up by weak forces. Since the
weak force eigenfunctions are Ginsburg-Landau-Zeeman eigenfunctions, by the
combinations of which the strong force eigenfunctions can be expressed,
this theory really unifies the electromagnetic, the weak, and the strong
nuclear forces. Let it be mentioned yet that the only force-category
missing from this list is the gravitational force. At this early
point of the development, we do not comment the question if this 
unification can be extended also to this force.

Now we turn back to establish an explicit formula describing 
certain eigenfunctions of $\Delta$ on general 
Heisenberg type Lie groups $H^{(a,b)}_l$. More details about the
general eigenfunction computations will also be provided. 
Although this construction can be
implemented also on general 2-step nilpotent Lie groups,
this more complicated case is omitted in
this paper. 
First, the eigenfunctions of a single 
angular momentum operator $\mathbf D_K\bullet$, defined for a
Z-vector $K$ are described as follows.  For a fixed X-vector $Q$
and unit Z-vector $K_u={1\over \mathbf k}K_u$, consider the X-function
$\Theta_{Q}(X,K_u)=\langle Q+\mathbf iJ_{K_u}(Q),X\rangle$ and its 
conjugate $\overline{\Theta}_{Q}(X,K_u)$. For a vector $K=\mathbf kK_u$
of length $\mathbf k$,
these functions are 
eigenfunctions of $D_{K}\bullet$ with eigenvalue 
$-\mathbf k\mathbf i$ resp.
$\mathbf k\mathbf i$. The higher order eigenfunctions are of the form
$\Theta_{Q}^p\overline\Theta^q_{Q}$ 
with eigenvalue $(q-p)\mathbf k\mathbf i$.
 
In order to find the eigenfunctions of the compound operator $\mathbf M_Z$,
consider a Z-sphere bundle $S_{R_Z}(\mathbf x)$ over the X-space 
whose radius function $R_Z(\mathbf x)$ depends just on 
$|X|=\mathbf x$. For an appropriate function
$\phi (\mathbf x,K)$ 
(depending on $\mathbf x$ and $K\in S_{R_Z}$, furthermore, which  
makes the following integral formula well defined) consider 
\begin{equation}
\label{fourR_Z}
\mathcal F_{QpqR_Z}(\phi )(X,Z)
=\oint_{S_{R_Z}}e^{\mathbf i
\langle Z,K\rangle}\phi (\mathbf x,K)
(\Theta_{Q}^p\overline\Theta^q_{Q})(X,K_u)dK_{no},
\end{equation} 
where $dK_{no}$ is the normalized measure on $S_{R_Z}(\mathbf x)$.
By $\mathbf M_Z\oint =\oint \mathbf iD_K\bullet$, 
this function restricted to the Z-space over an arbitrarily fixed
X-vector is an eigenfunction
of $\mathbf M_Z$ with the real eigenvalue $(p-q)R_Z(\mathbf x)$. 
These functions are eigenfunctions also of $\Delta_Z$ with eigenvalue
$R_Z^2(\mathbf x)$. Also note 
that these eigenvalues do not change by varying $Q$.

The function space spanned by functions (\ref{fourR_Z}) which are 
defined by all possible $\phi$'s is not invariant with respect to 
the action of $\Delta_X$, thus the eigenfunctions of the complete
operator $\Delta$ do not appear in this form. In order to find the
common eigenfunctions,
the homogeneous but non-harmonic polynomials 
$
\Theta_{Q}^p\overline\Theta^q_{Q}
$
of the X-variable should be exchanged for the 
polynomials
$
\Pi^{(n)}_X (\Theta_{Q}^p\overline\Theta^q_{Q})
$,
defined by projections, $\Pi_X$, onto the space of $n=(p+q)$-order 
homogeneous harmonic 
polynomials of the X-variable. Formula 
$\Pi_X =\Delta_X^0+B_1|X|^2\Delta_X +B_2\mathbf x^4\Delta_X^2+\dots$,
established in (\ref{proj}), implies that, over each X-vector, also  
\begin{equation}
\label{HfourR_Z}
\mathcal {HF}_{QpqR_Z}(\phi )(X,Z)
=\oint_{S_{R_Z}}e^{\mathbf i
\langle Z,K\rangle}\phi (\mathbf x,K)
\Pi^{(n)}_X(\Theta_{Q}^p\overline\Theta^q_{Q})(X,K_u))dK_{no}
\end{equation}
are eigenfunctions
of $\mathbf M$ and $\Delta_Z$ with the same eigenvalues 
what are defined for (\ref{fourR_Z}). 

The action of the complete Laplacian is a combination of
X-radial differentiation, $\partial_{\mathbf x}$, 
and multiplications with functions depending 
just on $\mathbf x$. Due to the normalized measure $dK_{no}$, these
operations can be considered such that they directly act 
inside of the integral sign on function 
$\phi (\mathbf x,K)$ in terms of the $\mathbf x$-variable, only. 
That is, the action is completely reduced to X-radial
functions and the eigenfunctions of $\Delta$ can be found in the form
\begin{eqnarray}
f(\mathbf x^2)\oint_{S_{R_Z}}e^{\mathbf i
\langle Z,K\rangle}
F_{Qpq}(X,K_u))dK_{no},
\quad
{\rm where}
\\
F_{Qpq}(X,K_u))=\varphi (K)\Pi^{(n)}_X
(\Theta_{Q}^p(X,K_u)\overline\Theta^q_{Q}(X,K_u)).
\end{eqnarray}
The same computations developed for the Z-crystals
yield that this reduced operator appears in the following form:
\begin{equation}
\label{Lf_mu(x)}
({\Diamond}_{\mu (\tilde t),\tilde t}f)(\tilde t)=
4\tilde tf^{\prime\prime}(\tilde t)
+(2k+4n)f^\prime (\tilde t)
-(2m\mu (\tilde t)  +4\mu^2(\tilde t)(1 +
{1\over 4}\tilde t))f(\tilde t),
\end{equation}
where $\tilde t=\mathbf x^2$, and 
$\mu (\tilde t)=R_Z(\sqrt{\tilde t})=R_Z(\mathbf x)$. Note that
this is exactly the same operator what was obtained for 
the radial Ginsburg-Landau-Zeeman operator on
Z-crystals. Since function 
$\mu (\tilde t)$ may depend also on  $\tilde t=\mathbf x^2$,
it actually appears in a more general form here.
However, for constant radius functions $R_Z$, 
it becomes the very same operator, indeed. That is, also this
eigenfunction-problem is reduced to finding the eigenfunctions of 
the Ginsburg-Landau-Zeeman operator reduced to X-radial functions. 
This reduced
operator remains the same by varying $Q$, thus also the spectrum
on the invariant spaces considered for fixed $Q$'s is not changing
regarding these variations. This phenomena reveals the later discussed 
spectral isotropy yielded on these models.

Note that this construction is carried out by a fixed X-vector $Q$,
but it extends to general polynomials as follows.
Consider an orthonormal system 
$\mathbf B=\{B_{1},\dots ,B_{\kappa}\}$ of vectors
on the X-space. They form a complex, but generically
non-orthonormal basis regarding the complex structures $J_{K_u}$,
where the unit vectors $K_u$ yielding this property form
an everywhere dense open set on the unit Z-sphere. 
This set is the complement of a set of $0$ measure.
The corresponding complex coordinate systems  on the X-space 
are denoted by
$\{z_{K_u1}=\Theta_{B_{K_u1}},\dots ,z_{K_u\kappa}
=\Theta_{B_{K_u\kappa}}\}$. For given values
$p_1,q_1,\dots ,p_{\kappa},q_{\kappa}$,  
consider the polynomial
$
\prod_{i=1}^{\kappa}
z_{K_ui}^{p_i} 
\overline z_{K_ui}^{q_i}. 
$
Then functions  
\begin{eqnarray}
\label{fourprod}
\oint_{S_{R_Z}}e^{\mathbf i
\langle Z,K\rangle}f(\mathbf x^2)\varphi (K)
 \Pi^{(n)}_X\prod_{i=1}^{\kappa}
z_{K_ui}^{p_i} 
\overline z_{K_ui}^{q_i} dK_{no}=
\\
=f(\mathbf x^2)\oint_{S_{R_Z}}e^{\mathbf i
\langle Z,K\rangle}F_{\mathbf Bp_iq_i}(X,K_u) dK_{no}
=  \mathcal {HF}_{\mathbf Bp_iq_iR_Z}(f )(X,Z)
\nonumber
\end{eqnarray}
are eigenfunctions of $\Delta$ if and only if function 
$f(\mathbf x^2)=f(\tilde t)$ is an eigenfunction of the
radial operator (\ref{Lf_mu(x)}), where $p=\sum p_i, q=\sum q_i, n=p+q$.
 
Consider a Z-ball bundle with radius function $R_Z(\mathbf x)$ defining
a compact or non-compact ball$\times$ball-type domain. Then the 
eigenfunctions satisfying the Dirichlet or Z-Neumann conditions on this
domain can not be sought among the above eigenfunctions because
functions $F_{Qpq}(X,K_u)$ resp. $F_{\mathbf Bp_iq_i}(X,K_u)$ are not
spherical harmonics regarding the $K_u$-variable but they are
rather combinations
of several spherical harmonics belonging to different eigenvalues of 
the Z-spherical Laplacian. It turns out, however, that one can construct
the complete function space satisfying a given boundary condition by the
above formulas if functions $F_{\dots}$ are substituted by their 
projections $\Pi^{(s)}_{K_u}(F_{\dots})$ into the space of $s^{th}$-order
spherical harmonics regarding variable $K_u$.
Actually, this projection appears in a more subtle form, 
$\Pi^{(vas)}_{K_u}=\Pi^{(\alpha)}_{K_u}$,
which projects $F_{\dots}$, first, into the space of
$(v+a)^{th}$-order homogeneous polynomials 
(where $v$ resp. $a$ refer to the degrees 
of functions to which $\varphi (K)$ resp. the $(p+q)^{th}$-order 
polynomials are projected). 
The projection to the $s^{th}$-order function
space applies, then, to these homogeneous functions. The most important
mathematical tool applied in this investigations is the Hankel transform
developed later. 

Although these new functions yield the boundary conditions,
they do not remain eigenfunctions
of the complete Laplacian anymore. However, they are still 
eigenfunctions of the
partial operator $\OE$. They are called weak-force-eigenfunctions, which
can be considered as electromagnetic-force-eigenfunctions because 
both $\OE$ and the Ginsburg-Landau-Zeeman operators
can be reduced to the same radial
operator. The old ones from which the new functions are derived 
are called linkage-eigenfunctions of $\Delta$, 
which bridge the electromagnetic interactions with the weak interactions.  
       
Since the extrinsic, $\mathbf{L}$, and intrinsic, $\mathbf{S}$,
operators do not
commute, the weak-force-eigenfunctions can not be the eigenfunctions
of the complete operator $\Delta =\OE +\mathbf{S}$. In other words, the
weak-force-eigenfunctions can not be equal to the 
strong-force-eigenfunctions defined by the eigenfunctions of $\Delta$
satisfying a given boundary condition.    
To construct these functions, the Fourier integrals must be considered
on the whole Z-space $\mathbb R^l$, by seeking them in the form 
$
\int_{\mathbb R^l}e^{\mathbf i\langle Z,K\rangle}
\phi_\alpha (\mathbf x,\mathbf z)\Pi_{K}^{(\alpha)}(F_{\dots}(X,K))dK.
$
The action of $\Delta$ on these functions can be described in the form
$
\int_{\mathbb R^l}e^{\mathbf i\langle Z,K\rangle}
\bigcirc_\alpha (\phi_1,\dots ,\phi_d) \Pi_{K}^{(\alpha)}(F_{\dots}
(X,K))dK,
$
where operator $\bigcirc_\alpha (\phi_1,\dots ,\phi_d)$, corresponding
d-tuples of $(X,Z)$-radial functions to each other, is defined in terms
of Hankel transforms combined with radial derivatives of the
functions appearing in the arguments. Finding the eigenfunctions of 
$\Delta$
means finding the eigen-d-tuples, $(\phi_1,\dots ,\phi_d)$, of the
radial, so called roulette operator $\bigcirc_\alpha$. 
As it will be pointed out, these eigenfunctions exhibit properties
characteristic to strong force eigenfunctions.
 
These arguments really unify the 3 fundamental forces of particle
theory. The real union is exhibited by the common unpolarized operator
$\Delta$. After polarization, they are separated into three categories. 
These cases correspond to the function spaces on which the polarized 
operators are acting. The details are as follows.

\subsection{Twisted Z-Fourier transforms.}
 
This is the main mathematical tool which incorporates de Broglie's
wave theory into the new models in a novel, more general form. 
The name indicates that the Fourier transform is performed, 
over each X-vector,
only on the Z-space in the same manner as if one would like to consider 
the de Broglie waves 
only in the center of the Lie group. However, an important new 
feature is that this transform applies to product of functions,
where one of them purely depends just on
the center variable, $K$, while the other is a complex polynomial
of the X-variable defined in terms of the complex structures $J_{K_u}$,
where $K_u=K/\mathbf k$ and $\mathbf k=|K|$. This Z-Fourier transform 
is said to be twisted by the latter polynomials. 
Thus this transform has impact also on 
the X-variable. This simple idea establishes the necessary connection 
between the abstract axes, $K_u$, and the particles placed onto the 
complex planes of the complex structures $J_{K_u}$.
 
This transform is defined in several alternative forms corresponding
to those introduced in the previous review. The difference is
that, over each X-vector, the following functions and integrals 
are defined on the whole Z-space. This is contrary to the previous section
where the integral is defined just on Z-spheres. 
Since the eigenfunctions constructed in the
review do not satisfy any of the boundary conditions, this 
reformulation of the Z-Fourier transform is really necessary for 
the complete solutions of the considered problems.   

In the first case, consider a 
fixed X-vector $Q$, and define the same functions
\begin{equation}
\Theta_{Q}(X,K_u)=\langle Q+\mathbf iJ_{K_u}(Q),X\rangle\, ,\,
\overline{\Theta}_{Q}(X,K_u),
\end{equation}
as above.
For fixed integers $p,q\geq 0$ and $L^2$-function 
$
\phi (\mathbf x,K),
$
consider the Z-Fourier transform
\begin{equation}
\mathcal F_{Qpq}(\phi )(X,Z)=
\int_{\mathbf z} e^{\mathbf i
\langle Z,K\rangle}
\phi (\mathbf x,K)
\Theta_{Q}^p(X,K_u)\overline\Theta^q_{Q}(X,K_u)dK,
\end{equation}
which is said to be twisted by the polynomial
$
\Theta_{Q}^p(X,K_u)\overline\Theta^q_{Q}(X,K_u).
$ 
Function $\phi$ is considered also in the form 
$\phi (\mathbf x,\mathbf k)\varphi (K_u)$, where $\varphi$ is 
a homogeneous polynomial of the K-variable.
The $L^2$-space spanned by these functions is denoted by
$\mathbf \Phi^{(n)}_{Qpq}$, where $n=p+q$ indicates that these functions
are $n^{th}$-order polynomials regarding the $X$-variable. 
The space spanned by 
the twisted functions
$\phi (\mathbf x,K)\Theta_{Q}^p(X,K_u)\overline\Theta^q_{Q}(X,K_u)$
is denoted by
$\mathbf {P\Phi}_{Qpq}^n$. This is the pre-space to which the Fourier
transform is applied.

Instead of a single vector $Q$, the second alternative form is defined 
regarding $\kappa =k/2$ independent vectors,
$\mathbf B=\{E_1,\dots ,E_{\kappa}\}$, of the X-space. Such a system 
forms a complex basis for almost all complex structure $J_{K_u}$, 
where $K_u=K/\mathbf k$. Now the twisting functions are polynomials 
of the complex coordinate functions
\begin{equation}
\{z_{K_u1}(X)=\Theta_{Q_1}(X,K_u),\dots 
,z_{K_u\kappa}(X)=\Theta_{Q_{\kappa}}(X,K_u)\},
\end{equation}  
where these formulas indicate that how the coordinate functions
can be expressed in terms of the above $\Theta_Q$-functions. 
For appropriate functions $\phi (\mathbf x,K)$ and polynomial exponents
$(p_i,q_i)$ (where $i=1,\dots ,k/2$) transform
$\mathcal F_{\mathbf B(p_iq_i)}(\phi\varphi )(X,Z)$ is defined by: 
\begin{eqnarray}
\int_{\mathbb R^l} e^{\mathbf i\langle Z,K\rangle}
\phi (\mathbf x,\mathbf k)\varphi (K_u)
\prod_{i=1}^{\kappa}z^{p_i}_{K_ui}(X)\overline z^{q_i}_{K_ui}(X)dK,
\end{eqnarray}
where $\varphi (K_u)$ is the restriction of an $m^{th}$-order homogeneous
polynomial, $\varphi (K)$, onto the unit sphere of the K-space.
If $\sum_i(p_i+q_i)=n$, then the twisting functions are $n^{th}$-order
complex valued polynomials regarding the X-variable and for any fixed $X$,
the whole function is of class $L^2_K$ regarding the K-variable. These
properties are inherited also for the transformed functions.

When $\phi (\mathbf x,K)$ 
runs through all functions which are of class $L^2_K$, 
for any fixed $|X|$, the transformed functions span the function space
$\mathbf \Phi_{\mathbf B p_iq_i}^{(n)}$. 
All these function spaces, defined for index sets 
satisfying $n=\sum (p_i+q_i)$, span the function space denoted by
$\mathbf \Phi_{\mathbf B}^{(n)}=
\sum_{\{(p_iq_i)\}}\mathbf \Phi_{\mathbf Bp_iq_i}^{(n)}$.
The corresponding pre-spaces are denoted by 
$\mathbf{P\Phi}_{\mathbf B p_iq_i}^{(n)}$
and 
$\mathbf{P\Phi}_{\mathbf B}^{(n)}$
respectively.

Twisted Z-Fourier transforms (\ref{fourR_Z})-(\ref{fourprod}) 
defined in the previous section by
considering Dirac-type functions concentrated on spheres $S_{R}$ of 
radius $R(\mathbf x)$ can be generated by the familiar $L^\infty$ 
approximation of these Dirac type functions by $L^2$ functions.
That is, this radius depends just on $\mathbf x$. Function spaces
$\mathbf{\Phi}_{\mathbf B p_iq_i R}^{(n)}$,
$\mathbf{\Phi}_{\mathbf B R}^{(n)}$, and their pre-spaces
$\mathbf{P\Phi}_{\mathbf B p_iq_iR}^{(n)}$
resp.
$\mathbf{P\Phi}_{\mathbf BR}^{(n)}$ are defined similarly as for $L^2$
functions. They are defined also for one-pole functions. These cases are
denoted such that $\mathbf B$ is replaced by $Q$. 

In many cases the
theorems hold true for each version of these function spaces. By this
reason we introduce the unified denotation  
$\mathbf{\Phi}_{...R}^{(n)}$ and $\mathbf{P\Phi}_{...R}^{(n)}$, 
where the dots represent the symbols
introduced above on the indicated places. A unified denotation for
the total spaces are 
$\mathbf{\Phi}_{.R}^{(n)}$ and $\mathbf{P\Phi}_{.R}^{(n)}$,
indicating that symbols $p_iq_i$ do not show up in these formulas.
If letter $R$ is omitted, the formulas concern the previous cases
when the functions are of class $L^2$ regarding the $K$-variable.

The elements of the latter total function spaces are
complex valued functions defined on the $(X,Z)$-space such that,
upto multiplicative X-radial functions, they are
$n^{th}$-order polynomials regarding the X-variable and for any fixed $X$
they are $L^2$ functions regarding the Z-variable. It is very
important to clarify the relations between the above twisted spaces 
and the latter complex valued function spaces. By considering an 
arbitrary real basis $\mathbf Q=\{Q_1,\dots ,Q_k\}$ on the X-space,
the complex valued functions appear in the form 
$\sum_{\{a_1,\dots ,a_k\}}\phi_{a_1,\dots ,a_k}(|X|,Z)
\prod_{i=1}^k\langle Q_i,X\rangle^{a_i}$, 
where the sum is considered for all sets $\{a_1,\dots ,a_k\}$ of 
non-negative integers satisfying $\sum a_i=n$ and functions 
$\phi$ are $L^2_Z$-functions for any fixed $|X|$. The function space
spanned by these functions is called 
{\it straight $L^2_Z$ space of complex valued 
$(X,Z)$-functions}. The Z-Fourier transform
performed on such functions is called {\it straight Z-Fourier transform}. 
The explanations below show that twisted functions   
$\phi (|X|,K)
\prod_{i=1}^{k/2}z^{p_i}_{K_ui}(X)\overline z^{q_i}_{K_ui}(X)$
can be converted to straightly represented functions in the terms of which
the twisted Z-Fourier transform becomes a straight Z-Fourier transform. 
This conversion works out also in the opposite (from the straight to the
twisted) direction. The precise details below 
show that the straightly defined 
function spaces are complete regarding the $L^2_Z$ norm in which both
$\mathbf \Phi_{\mathbf B}^n$ and the pre-space 
$\mathbf {P\Phi}_{\mathbf B}^n$ 
are everywhere dense subspaces. By this reason, the corresponding straight
spaces are denoted by
$\overline{\mathbf \Phi}_{\mathbf B}^n$ resp. 
$\overline{\mathbf {P\Phi}}_{\mathbf B}^n$. 
However, these spaces are equal. 
Ultimately, both versions of the Z-Fourier 
transforms define authomorphisms 
(one to one and onto maps) of this very same ambient function space.

This situation can be illuminated by the real $k\times k$ matrix field
$A_{ij}(K_u)$ defined on the unit Z-vectors 
which transforms the real basis
$\mathbf Q$ to the vector system 
$\mathbf B_{\mathbb R}=\{B_1,\dots ,B_{k/2},B_{(k/2)+1}=J_{K_u}(B_1),
\dots , B_k=J_{K_u}(B_{k/2})\}$. 
That is, this field is uniquely determined by the
formula $B_i=\sum_{j=1}^kA_{ij}Q_j$, 
where $i=1,\dots ,k$. The entries are
polynomials of $K_u$. By plugging these formulas 
into the twisted Z-Fourier
transform formula, one gets the straight 
representations both of the twisted
functions and their Z-Fourier transforms.     

Conversion from the straight to the twisted functions is more complicated.
In this case vectors $Q_i$ should be exchanged for vectors $B_j$
according to the formula $Q_i=\sum_{j=1}^kA^{-1}_{ij}B_j$. 
Then vectors
$B_j$ resp. $B_{(k/2)+j}$, where $j\leq k/2$, should be expressed
in the form
\begin{eqnarray}
B_j={1\over 2}(B_j+J_{K_u}(B_j))+{1\over 2}(B_j-J_{K_u}(B_j))\quad 
{\rm resp.}
\\
B_{(k/2)+j}=J_{K_u}(B_j)=
{1\over 2}(B_j+J_{K_u}(B_j))-{1\over 2}(B_j-J_{K_u}(B_j)),
\end{eqnarray}
which, after performing powering and appropriate rearranging, 
provide the desired twisted formulas.Due to the
degeneracy of the matrix field $A_{ij}$ on $S_{\mathbf B}$, some entries
of the inverse matrix field $A^{-1}_{ij}$ have limits $+\infty$ or
$-\infty$ of order at most $k/2$ on this singularity set. Therefore,
the Z-Fourier transform of a term of
the twisted function involving such functions may be not defined, despite
the fact that the Z-Fourier transform of the whole function exist which is
equal to the transform of the straightly represented function. In other
words, the infinities appearing in the separate terms cancel each other
out in the complete function.

This contradictory situation can be resolved as follows. For a given
$\epsilon >0$, let $S_{\mathbf B\epsilon}$ be the $\epsilon$-neighborhood
of the singularity set on the unit Z-sphere and let 
$\mathbb RS_{\mathbf B\epsilon}$ be the conic set covered by the rays 
emanating from the origin which are spanned by unit Z-vectors pointing
to the points of $S_{\mathbf B\epsilon}$. For an $L^2_Z$-function 
$\phi (\mathbf x,K)$ discussed above let 
$\phi_\epsilon(\mathbf x,K)$ be the function
which is the same as $\phi$ on the outside of 
$\mathbb RS_{\mathbf B\epsilon}$ and it is equal to zero in the inside of
this set. Then, regarding the $L^2_Z$-norm, 
$\lim_{\epsilon\to 0}\phi_\epsilon =\phi$ holds. For a function $F(X,K)$,
expressed straightly, define $F_\epsilon (X,K)$ by substituting each 
function $\phi$ by $\phi_\epsilon$. If function $\psi_\epsilon (K)$ is
defined by $1$ outside of $\mathbb RS_{\mathbf B\epsilon}$ and by $0$
inside, then $F_\epsilon (X,K)=\psi_\epsilon (K)F(X,K)$ holds. Convert
$F_\epsilon$ into the twisted form. Then the twisted Z-Fourier transform
is well defined for each 
term of the twisted expression, providing the same
transformed function as what is defined by the straight Z-Fourier 
transform. Thus each straightly represented $L^2_Z$ function is an 
$L^2_Z$-limit of functions which can be converted to twisted functions 
in which each twisted term is an $L^2_Z$ function having well defined 
twisted Z-Fourier transform.
In this process, function $\psi_\epsilon (K)$, which is constant in radial
directions, can be chosen such that it is of
class $C^\infty$ satisfying $0\leq\psi_\epsilon (K)\leq 1$, furthermore,
it vanishes
on $\mathbb RS_{\mathbf B\epsilon /2}$ and is equal to $1$ outside of
$\mathbb RS_{\mathbf B\epsilon}$. Thus we have:

\begin{theorem}
\label{welldef1}
The twisted functions, defined in each of 
their terms by $L^2_Z$ functions, 
form an everywhere dense subspace  ${\mathbf {P\Phi}}_{\mathbf B}^n$
in the complete space $\overline{\mathbf {P\Phi}}_{\mathbf B}^n$ 
of straightly defined 
$L^2_Z$ functions. Although the first space depends on $\mathbf B$, the
second one is a uniquely determined function space which does not 
depend neither on $\mathbf B$ nor on $\mathbf Q$.

A general function from the ambient 
space becomes an appropriate twisted function belonging 
to the dense subspace
after multiplying it with a function $\psi_\epsilon (K)$
which is zero on the above described set
$\mathbb RS_{\mathbf B\epsilon}$ and equal to $1$ on the complement of
this set. This function, $\psi_\epsilon (K)$, can be chosen
to be of class $C^\infty$
such that
it is constant in radial directions 
satisfying $0\leq\psi_\epsilon (K)\leq 1$, furthermore,
it vanishes
on $\mathbb RS_{\mathbf B\epsilon /2}$ and is equal to $1$ outside of
$\mathbb RS_{\mathbf B\epsilon}$. 
Then the $L^2_Z$-approximation is defined
by the limiting $\epsilon\to 0$.

The twisted Z-Fourier transform continuously extends
to the straight Z-Fourier transform defined on the ambient space. 
On the ambient
space this Fourier transform is an authomorphism.  
The Z-Fourier transform is a bijection between the everywhere 
dense subspaces 
${\mathbf {P\Phi}}_{\mathbf B}^n$ 
and ${\mathbf \Phi}_{\mathbf B}^n$. Particularly the relation
$\overline{\mathbf {P\Phi}}_{\mathbf B}^n=
\overline{\mathbf {\Phi}}_{\mathbf B}^n$ holds.

The same statements hold true also for Q-pole function spaces. Total
spaces  
$\mathbf{\Phi}_{\mathbf B}^{(n)}$ resp. $\mathbf{P\Phi}_{\mathbf B}^{(n)}$
are spanned by the corresponding Q-pole functions whose poles, $Q$,  
are in the real span of vectors belonging to $\mathbf B$. 
\end{theorem}
 
In order to show the one to one property, 
suppose that the Z-Fourier transform of a
twisted function vanishes. That is:      
\begin{eqnarray}
\label{fou=0}
\int_{\mathbb R^l} e^{\mathbf i\langle Z,K\rangle}\sum_{\{(p_i,q_i)\}}
\phi_{(p_i,q_i)} (|X|,K)
\prod_{i=1}^{k/2}z^{p_i}_{K_ui}(X)\overline z^{q_i}_{K_ui}(X)dK=
\\
=\int_{\mathbb R^l} e^{\mathbf i\langle Z,K\rangle}\sum_{\{(p_i,q_i)\}}
\Phi_{(p_i,q_i)} (X,K)dK=0.
\nonumber
\end{eqnarray}
Then, for any fixed $X$, the Z-Fourier transform of  
$\sum_{\{(p_i,q_i)\}}\Phi_{(p_i,q_i)} (X,K)$ vanishes. 
Therefore, this function
must be zero for all $X$ and for almost all $K$. For vectors $K_u$
not lying in the singularity set $S_{\mathbf B}$, the $\mathbf B$ is a 
complex basis and complex polynomials 
$\prod_{i=1}^{k/2}z^{p_i}_{K_ui}(X)\overline z^{q_i}_{K_ui}(X)$
of the X-variable are linearly 
independent for distinct set $\{(p_i,q_i)\}$
of exponents. Thus all component functions $\phi_{(p_i,q_i)} (|X|,K)$ 
of the $K$ variable must vanish almost everywhere. 
This proves that the twisted
Z-Fourier transform is a one to one map of the everywhere 
dense twisted subspace onto the everywhere dense twisted range space.
The same proof, applied to straight 
functions, yields the statement on the
straight space where the Z-Fourier transform is obviously an onto map by
the well known theorem of Fourier transforms.   
 
This section is concluded by introducing several invariant functions
by which physical objects such as charge and volume will be precisely
defined. Suppose that system $\mathbf B$ consists of orthonormal
vectors and let $\mathbb B^\perp$ be the orthogonal complement of
the real span, $\mathbb B=Span_{\mathbb R}(\mathbf B)$, of vectors
belonging to $\mathbf B$. Basis $\mathbf B$ defines an orientation
on $\mathbb B$. Let $\mathbb B_+^\perp$ resp. $\mathbb B_-^\perp$
be the two possible orientations which can be chosen on $\mathbb B^\perp$
such that, together with the orientation of  $\mathbb B$, they define
positive resp. negative orientation on the X-space, $\mathbb R^k$.
Charging a particle system represented by $\mathbf B$ means choosing
one of these two orientations. Once the system is charged, choose an
orthonormal basis also in $\mathbb B^\perp$, complying with the chosen
orientation. Together with $\mathbf B$, this system defines an
orthonormal basis, $\mathbf Q$, on the X-space by which the matrix
field $A_{ij}(Z_u)$ defined above can be introduced. The invariants
of these matrix field are independent from the above basis chosen on
$\mathbb B^\perp$. Invariant functions  
$\mathit{ch}(Z_u)=Tr(A_{ij}(Z_u))-(k/2)$ and 
$\mathit{v}(Z_u)=det(A_{ij}(Z_u))$ are called {\it charger} and
{\it volumer} respectively. Their integral regarding a proper probability
density defines the charge resp. mass of the particle being on a given
proper state (these concepts are precisely 
described later in this paper).

\subsection{Hankel transform.}
This twisted Fourier transform is investigated by means of the 
Hankel transform. The statement regarding this transform 
asserts:
\begin{theorem} The Fourier 
transform considered on 
$\mathbb R^l$ 
transforms a product, 
$f(r)F^{(\nu )}(\theta )$,
of radial functions 
and spherical harmonics
to the product,
$H^{(l)}_\nu (f)(r)F^{(\nu )}(\theta )$, of the same form, i. e.,
for any fixed degree $\nu$ of the spherical harmonics, it 
induces maps,
$H^{(l)}_\nu (f)(r)$, 
on the radial functions, which,so called Hankel transform, is uniquely
determined for any fixed indices $l$ and $\nu$. 
\end{theorem} 

This is actually a weak form of the original Hankel theorem
which can be directly settled by the following
mean value theorem of the spherical harmonics, $F^{(\nu )}(\theta )$,
defined on the unit sphere, $S$, about the origin 
of $\mathbb R^l$. These functions are eigenfunctions of the Laplacian
$\Delta_S$ with eigenvalue $\lambda_\nu$, moreover, there exists a 
uniquely determined radial eigenfunction $\varphi_{\lambda_\nu}(\rho )$,
where $0\leq \rho\leq\pi$ and $\varphi_{\lambda_\nu}(0)=1$, on $S$ which
has the same eigenvalue $\lambda_\nu$ such that, 
on a hypersphere $\sigma_\rho(\theta)\subset S$ 
of radius
$\rho$ and center $\theta$ on the ambient sphere $S$, the identity
$\oint_\sigma F^{(\nu )}d\sigma_{no} =F^{(\nu )}(\theta )
\varphi_{\lambda_\nu}(\rho )$ holds, where $d\sigma_{no}$ 
is the normalized  measure measuring $\sigma$ by $1$. 

This mean value theorem can be used
for computing the Fourier transform
$\int_{\mathbb R^l}e^{\mathbf i\langle Z,K\rangle}f(|K|)F^{(\nu )}
(\theta_K)dK$ at a point $Z=(|Z|,\theta_Z)$.
This integral is computed by Fubini's theorem such that one considers
the line $l_Z(t)$, spanned by $\theta_Z$; parameterized with arc-length
$t$; and satisfying $l_Z(Z)>0$, and one computes the integrals first in
hyperplanes intersecting $l_Z$ at $t$ perpendicularly and then on 
$l_Z$ by $dt$.
On the hyperplanes, write up the integral in polar coordinates defined
around the intersection point with $l_Z$, where the radial Euclidean 
distance from this origin is denoted by $\tau$. Consider polar 
coordinates also
on $S$ around $\theta_Z$, where the radial spherical distance from 
this origin is denoted by $\rho$. Then $\tau =|t||\tan \rho |$ holds,
furthermore, a straightforward computation yields:
\begin{eqnarray}
 \int_{\mathbb R^l}e^{\mathbf i\langle Z,K\rangle}f(|K|)F^{(\nu )}
(\theta_K)dK=
H^{(l)}_\nu (f)(r)F^{(\nu )}(\theta_Z),\quad {\rm where}
\\
H^{(l)}_\nu (f)(r)=\Omega_{l-2}\int_{-\infty}^\infty 
e^{\mathbf irt} |t|^{l-1}
\int_0^\pi f(|t\tan\rho |)\varphi_{\lambda_\nu}(\rho )
{\sin^{l-2}\rho\over |\cos\rho |^l}
d\rho dt,
\nonumber
\end{eqnarray}
where $\Omega_{l-2}$ denotes the volume of an $(l-2)$-dimensional
Euclidean unit sphere.
These formulas prove the above statement completely.

\subsection{Projecting to spherical harmonics.}
Among the other mathematical tools by which twisted Z-Fourier
transforms are investigated are the projections 
$
\Pi_{K_u}^{(r,s)} (\varphi (K_u)
\prod_{i=1}^{\kappa}z^{p_i}_{K_ui}(X)\overline z^{q_i}_{K_ui}(X)),
$
corresponding $s^{th}$-order polynomials to  
$r^{th}$-order polynomials of the $K_u$-variable. 
Although the functions they are applied to may
depend also on the X-variable, they refer strictly to the $K_u$-variable, 
meaning, that they are performed, over each X-vector, in the Z-space. 
These characteristics are exhibited also by the fact that these 
projections appear as certain polynomials of the Laplacian
$\Delta_{K_u}$ defined on the unit K-sphere. It is much more convenient 
to describe them in terms of homogeneous functions, which are projected
by them to harmonic homogeneous polynomials. In this version 
these polynomials depend on $K$ and the projections can be described 
in terms of the Laplacian $\Delta_{K}$ defined on the ambient space. 
By
restrictions onto the unit K-spheres, one can easily find then the
desired formulas in terms of $\Pi_{K_u}^{(r,s)}$.

For an $n^{th}$ order homogeneous polynomial, $P_n(K)$, 
projection $\Pi^{(n)}_K:=\Pi^{(n,n)}_K$ onto the space of $n^{th}$ order 
harmonic polynomials can be computed by the formula
\begin{equation}
\label{proj}
\Pi_K^{(n)}(P_n(K)) 
=\sum_s C^{(n)}_s \langle K,K\rangle^s\Delta^s_K(P_n(K)),
\end{equation}
where $C^{(n)}_0=1$ and the other coefficients can be determined
by the recursive formula
$
2s(2(s+n)-1)C^{(n)}_s+C^{(n)}_{s-1}=0.
$
In fact, exactly for these coefficients is the function defined
by an arbitrary $P_n(K)$ on the right side a homogeneous 
harmonic polynomial.
These formulas can be easily established for polynomials  
$P_n(K)=\langle W,K\rangle^n$, defined by a fixed Z-vector $W$. 
Since they span
the space of $n^{th}$ order homogeneous polynomials, the statement
follows also for general complex valued polynomials. 
This projection is a surjective map of the $n^{th}$-order 
homogeneous polynomial space
$\mathcal P^{(n)}(K)$ onto the space, $\mathcal H^{(n)}(K)$, 
of the $n^{th}$-order homogeneous harmonic polynomials   
whose kernel  
is formed by polynomials of the form  $\langle K,K\rangle P_{n-2}(K)$. 
Subspace
$\mathcal H^{(n)}\subset \mathcal P^{(n)}$ is a complement
to this kernel.

The complete decomposition of $P_n(K)$ appears in the
form $P_n(K)=\sum_i\langle K,K\rangle^i HP_{n-2i}(K)$, where $i$ 
starts with $0$ and running through the integer part, $[n/2]$, of
$n/2$, furthermore, $HP_{n-2i}(K)$ are harmonic 
polynomials of order $(n-2i)$. It can
be established by successive application of the above computations. In the
second step in this process, one considers the functions 
$P_n(K)-\Pi_K^{(n)}(P_n(K))$, which appear in the form 
$\langle K,K\rangle P_{n-2}(K)$, and obtains $HP_{n-2}(K)$ by the
projection $\Pi_K^{(n-2)}(P_{n-2}(K))$. Then, also this second term
is removed from $P_n(K)$, in order to get ready for the third step,
where the same computations are repeated. This process can be completed
in at most $[n/2]$ steps. It is clear that
projections $\Pi_K^{(n,n-2i)}$ resulting functions $HP_{n-2i}(K)$
from $P_n(K)$ are of the form  
$\Pi_K^{(n,n-2i)}=D_{(n,n-2i)}\Pi_K^{(n-2i,n-2i)}\Delta_K^i$,
where indices $(n,n-2i)$ indicates that the projection maps 
$n^{th}$-order polynomials
to $(n-2i)^{th}$-order harmonic polynomials (in this respect, projection
$\Pi_K^{(n)}$ is the same as $\Pi_K^{(n,n)}$). The technical calculation
of constants $D_{(n,n-2i)}$ 
is omitted. The corresponding projections $\Pi_{K_u}^{(r,s)}$ defined
on the unit spheres can immediately be established by these projections
defined for homogeneous functions. They appear
as polynomials of the Laplacian $\Delta_{K_u}$ defined on the unit
$K_u$-sphere. 

Since these projections depend just on the degrees $r$ and $s$,
they apply also to twisted functions which
depend on the X-variable as well.
In order to make the Hankel transform applicable, they are used for
decomposing functions in the form
\begin{eqnarray}
\sum_{(r;s)} f_{(r;s)}(\mathbf x,\mathbf k)
\Pi_{K_u}^{(r,s)} (\varphi (K_u)
\prod_{i=1}^{\kappa}z^{p_i}_{K_ui}(X)\overline z^{q_i}_{K_ui}(X))\\
=f_{\alpha}(\mathbf x,\mathbf k)
\Pi_{K_u}^{\alpha} (F^{(p_i,q_i)}(X,K_u)),
\nonumber
\end{eqnarray}
where the right side is just a short way 
to describe the sum appearing on the left side in terms of compound 
indices $\alpha =(r,s)$ and functions $F^{(p_i,q_i)}$.

To be more precise, functions 
$\phi_n(\mathbf x,K)P^{(n)}(X,K_u)$ 
appearing in the pre-spaces  
must be brought to appropriate forms
before these projections can directly be applied to them. First of all, 
function $\phi_n(\mathbf x,K)$ should be considered in the
form $\phi_n(\mathbf x,K)=\sum_v\phi_{n,v}
(\mathbf x,\mathbf k)\varphi^{(v)}(K)$, 
where 
$\varphi^{(v)}(K)$ is an $v^{th}$-order homogeneous harmonic polynomial.
Then, after implementing all term-by-term multiplications in the 
products of
$\Theta_{B_i}=\langle B_i,X\rangle+\langle\mathbf iJ_{K_u}(B_i),X\rangle$
and
$\overline\Theta_{B_i}=\langle B_i,X\rangle-
\langle\mathbf iJ_{K_u}(B_i),X\rangle$, the above polynomials
have to be 
taken to the form  $P^{(n)}(X,K_u)=\sum_{a=0}^nP^{(n,a)}(X,K_u)$,
where polynomial $P^{(n,a)}$ involves exactly $a$ number of 
linear polynomials of the form $\langle J_{K_u}(Q_i),X\rangle$. The 
above projections defined in terms of $r$ directly act on functions 
$\varphi^{(v)}(K)P^{(n,a)}(X,K)$ satisfying $r=v+a$, which are
obtained by term by term multiplications of the sums given above for 
$\phi_n(\mathbf x,K)$ and $P^{(n)}(X,K_u)$. 

This complicated process can be considerably simplified by
considering only one-pole functions defined for single
$Q$'s which are in the real span of the vector system $\mathbf B$. 
In fact, all the 1-pole total spaces
$\mathbf\Phi_{Q}^{(n)}=\sum_{p,q}\mathbf\Phi_{Qpq}$, where $n=p+q$, 
span also the total space 
$\sum_{(p_i,q_i)}\mathbf\Phi^{(n)}_{\mathbf B}$, thus there is 
really enough to
establish the theorems just for these simpler  
one-pole functions. In this case, function $P^{(n,a)}(X,K_u)$ 
is nothing but a constant-times of function 
\begin{equation}
R_Q^{(n,a)}(X,K_u)=\langle Q,X\rangle^{n-a} 
\langle J_{K_u}(Q),X\rangle^{a}=
\langle Q,X\rangle^{n-a} \langle [Q,X],K_u\rangle^{a}.
\end{equation}
Note that, depending on $p$ and $q$, 
the component 
of a particular $P^{(n)}(X,K_u)$ corresponding to a
given $0\leq a\leq n$ may vanish. However, there exist such pairs $(p,q)$
for which this $a$-component is non-zero. Space 
$\mathbf P^{(n,a)}(X,K)$ is defined by the span of functions 
$R_Q^{(n,a)}(X,K)$. Note that $K_u$ has been changed to $K$ in the above 
formula. These functions are $n^{th}$- resp. $a^{th}$-order homogeneous
polynomials regarding the X- resp. K-variables. Keep
in mind that these functions can be derived from functions 
$\Theta_Q^p\overline\Theta_Q^q$,
where $p+q=n$, by linear combinations, therefore, they belong to the
above twisted function spaces. More
precisely, for a given $n$, there exist an invertible matrix 
$M^{(n,a)}_{pq}$
such that
$R_Q^{(n,a)}(X,K_u)=\sum_{p,q}M^{(n,a)}_{pq}\Theta_Q^p\overline\Theta_Q^q$
hold, where $p+q=n$.

Projections
$\Pi_X^{(n)}=\sum_s C^{(n)}_s \langle X,X\rangle^s\Delta^s_X$ 
acting on functions   
$P(X,K_u)=\Theta^p_{Q}(X,K_u)\overline
\Theta^q_{Q}(X,K_u)$ resp.
$P(X,K_u)=\prod_{i=1}^{\kappa}z^{p_i}_{K_ui}(X)
\overline z^{q_i}_{K_ui}(X))$,
where $p+q=n$ resp. $\sum_ip_i+q_i=n$, regarding the X-variable are also 
involved to these investigations. 
Since they are $n^{th}$-order homogeneous functions
regarding $X$, projection $\Pi_X^{(n)}$ applies to them
immediately.
Then, for any fixed $K_u$,
function  $\Pi_X^{(n)}(P(X,K_u))$ is an $n^{th}$-order homogeneous 
harmonic polynomial regarding the X-variable. 
The twisted Z-Fourier transforms,
$\mathcal{HF}_{\mathbf B(p_iq_i)}(\phi )(X,Z)$, 
involving these projections are defined by

\begin{eqnarray}
\label{projio1}
\int_{\mathbb R^l} e^{\mathbf i\langle Z,K\rangle}
\phi (|X|,K)\Pi_X^{(n)}
(\prod_{i=1}^{k/2}z^{p_i}_{K_ui}(X)\overline z^{q_i}_{K_ui}(X))dK.
\end{eqnarray}
 
The corresponding $L^2_Z$ function spaces spanned by the transformed 
functions and the pre-space are denoted by
$\mathbf \Xi_{\mathbf Bp_iq_i}^{(n)}$ and  
$\mathbf{P\Xi}_{\mathbf Bp_iq_i}^{(n)}$ respectively. These function spaces
are well defined also for one-pole functions and also for the third type
of Z-Fourier transforms defined for Dirac type generalized functions
concentrated on Z-sphere bundles.

Although the following results are not used in the rest part of this paper,
because of their importance, we describe some mathematical 
processes by which these
projections can explicitly be computed.
Further on, the formulas concern a fixed $n$ even if 
it is not indicated there. 
Decomposition into 
K-harmonic polynomials will be implemented for the K-homogeneous functions 
$\varphi^{(v)}(K)R_Q^{(a)}(X,K)$ of order $r=v+a$. 
According to two different representations of the first function,
these projections
will be described in two different ways.
The first description is more or less technical, yet, very useful 
in proving the independence theorems stated below. 
In the second 
description, the projected functions are directly constructed.
In both cases we consider Q-pole functions defined by a unit
vector $Q\in Span_{\mathbb R}{\mathbf B}=\mathbb B$. However,
the multipole cases 
referring to vector systems $\mathbf B$ are also discussed
in the theorem established below.
 
According to the formula (\ref{projio1}), a pure harmonic 
one-pole function
with pole $\zeta$ in the K-space is of the form
$\varphi_{\zeta}^{(v)}(K)=
\sum_s D_{2s}\mathbf k^{2s}\langle\zeta ,K\rangle^{v-2s}$,
where $D_0=|\zeta |^{v}$ and the other coefficients 
depending on $|\zeta |^{v-2s}$ and constants $C_s$ 
are uniquely determined by the
harmonicity assumption. It is well known that these functions span the
space of homogeneous harmonic polynomials. Functions 
$\partial^c_{\tilde K}\gamma_{\zeta}^{(v)}$ obtained by 
directional derivatives regarding a fixed vector $\tilde K$ 
are also homogeneous harmonic $\zeta$-pole functions of order $(v-c)$.

The action of operator $\Delta^b_{K}$ on 
$\varphi_\zeta^{(v)}(K)R_Q^{(a)}(X,K)$ results function:

\begin{equation}
\sum_{c=0}^{[{a\over 2}]}D_c\langle Q,X\rangle^{n-a}
(\partial_\zeta^{b-c}\varphi_\zeta^{(v)})(K)
\langle J_\zeta (Q),X\rangle^{b-c}  \langle J_{K} (Q),X\rangle^{a-b-c}
|[Q,X]|^{2c}.
\nonumber
\end{equation}    
The terms of this sum are obtained such that $\Delta_K^c$ acts
on $R_Q^{(n,a)}(X,K)$, resulting the very last term, while the others 
are due to the action of $\Delta^{b-c}$ on the product according to
the formula
$
\partial_\zeta^{b-c}\gamma_\zeta^{(v)}
\partial_\zeta^{b-c}R_Q^{(a)}(X,K)
$. 
Note that, because of the harmonicity, the action of $\Delta_K$ on 
$\varphi_\zeta^{(v)}$ is trivial.

When the complete projection $\Pi^{(s)}_K$ is computed, then
the $b$'s involved to the formula are denoted by $b_j$. 
For given $c$, factor out
$|[Q,X]|^{2c}$ from the corresponding terms. Thus the final projection
formula appears in the form
$\sum_c|K|^{4c}P_c^{(s)}(X,K)|[Q,X]|^{2c}$,
where term $P_c^{(s)}$ is equal to

\begin{equation}
\label{P^s_q}
\sum_{j}D_{cj}\langle Q,X\rangle^{n-a}
(\partial_\zeta^{b_j-c}\varphi_\zeta^{(v)})(K)
\langle J_\zeta (Q),X\rangle^{b_j-c}  \langle J_{K} (Q),X\rangle^{a-b_j-c}.
\end{equation}    

This is a rather formal description of the projected functions.
A more concrete construction is as follows. 
The linear map $\mathcal X\to\mathcal Z$
defined by $X\to [Q,X]$ is surjective 
whose kernel is a $(k-l)$-dimensional
subspace. Next projections $\Pi_K$ will be investigated in 
the Z-space over such an X-vector, $\tilde X$, which is not in this kernel 
and the unit vector 
$\zeta_Q(\tilde X)$ defined by $[Q,\tilde X]=
|[Q,\tilde X]|\zeta_Q(\tilde X)$ is not in the 
singularity set $S_{\mathbf B}$. 
If $X_Q$ denotes the orthogonal projection of
$X$ onto the $l$-dimensional subspace spanned by $Q$ and vectors 
$J_{K_u}(Q)$ considered for all unit vectors $K_u$, then 
$|[Q,X]|^{2}=|X_Q|^2- \langle Q,X\rangle^{2}$.

The direct representation of $\varphi^{(v)}$ is the product
of two harmonic one-pole functions having perpendicular poles. 
One of the poles is $\tilde\zeta=\zeta_Q(\tilde X)$ while the other
is an arbitrary perpendicular unit vector 
$\tilde\zeta^\perp \not\in S_{\mathbf B}$. Then,
for all $0\leq c\leq v$, consider the product 
$\varphi^{(v)}=\varphi_{\tilde\zeta^\perp}^{(c)}
\varphi_{\tilde\zeta}^{(v-c)}$
of harmonic homogeneous one-pole functions. Because of the perpendicular
poles, these products are also harmonic functions of the K-variable, 
furthermore, considering them for all $0\leq c\leq v$ and 
$\tilde\zeta^\perp$, they span the whole space of $v^{th}$-order 
homogeneous harmonic polynomials
of the K-variable, for each point $X$. However, $\tilde\zeta$ is not 
pointing into the direction of $[Q,X]$ in general. 
When such a pure K-function is multiplied
with $R_Q^{(n,a)}(X,K)$, over $\tilde X$, for the required projections, 
only the decomposition
of $\varphi_{\tilde\zeta}^{(v-c)}\langle \tilde\zeta ,K\rangle^a$ 
should be determined. A simple calculation shows:
\begin{equation}
\varphi_{\tilde\zeta}^{(v-c)}\langle \tilde\zeta ,
K\rangle^a=\sum_{i=0}^{a}
D_i\langle K,K\rangle^i \varphi_{\tilde\zeta}^{(v-c+a-2i)}.
\end{equation}
This function multiplied with 
$\langle Q,X\rangle^{n-a}|[Q,\tilde X]|^a 
\varphi_{\tilde\zeta^\perp}^{(c)}$
provides the desired decomposition and projections.
When this projected function is considered over an arbitrary $X$, neither
$\varphi_{\tilde\zeta^\perp}^{(c)}$ nor
$\varphi_{\tilde\zeta}^{(v-c+a-2i)}$  appear as pure K-functions. 
One can state only that, for all $i$, this product is an 
$a^{th}$-order homogeneous
polynomial, also regarding the X-variable. 
This complication is due to that  
$\zeta_Q(\tilde X)$ and $\zeta_Q(X)$ are not parallel, 
thus the projection operator 
involves both functions to the computations.  
According to these arguments we have:

\begin{theorem}
\label{projth}
(A)
For given non-negative integers $r$, $a$, and $(v-a)\leq s\leq (v+a)$;
where $s$ has the same parity as $(v+a)$ or $(v-a)$; functions
$\varphi^{(v)}(K)R_Q^{(n,a)}(X,K)$ project to $(v+a)^{th}$-order 
K-homogeneous functions which, over $\tilde X$, appear in the form

\begin{eqnarray}
\label{proj_s}
D_{(v+a-s)/2}\langle K,K\rangle^{(v+a-s)/2}
\langle Q,X\rangle^{n-a}|[Q,\tilde X]|^a
\varphi_{\tilde\zeta^\perp}^{(c)}
\varphi_{\tilde\zeta}^{(s-c)},
\end{eqnarray}
where $0\leq c\leq v$, which, by omitting term 
$\langle K,K\rangle^{(v+a-s)/2}$, 
are $s^{th}$-order homogeneous K-harmonic
polynomials. For given $v$ and $a$, this projection is trivial
(that is, it maps to zero) for all those values $s$ which do not 
satisfy the above
conditions. Above, exactly the non-trivial terms are determined.

(B) Functions 
$\langle K,K\rangle^{(v+a-s)/2}$ and $\langle Q,X\rangle^{n-a}$
of degrees $(v+a-s)$ resp. $(n-a)$ 
are independent from $c$, furthermore, function
$\varphi_{\tilde\zeta^\perp}^{(c)}\varphi_{\tilde\zeta}^{(s-c)}$
is an $a^{th}$-order homogeneous polynomial of the X-variable.
It follows that, for a given $s$,
the subspaces, $\mathbf {PHo}_Q^{(v,a)}$, spanned by projected functions 
considered for distinct pairs $(v,a)$ are independent, more precisely,
$\sum_{(v,a)}\mathbf {PHo}_Q^{(v,a)}=\mathbf {PHo}_Q^{(s)}$ is a finite
direct sum decomposition of the corresponding space of complex valued 
homogeneous functions which are $n^{th}$-order regarding the X-variable
and $s^{th}$-order harmonic functions regarding the Z-variable. 
This decomposition is further graded by the subspaces, 
$\mathbf {PHo}_Q^{(v,a,c)}$ defined for distinct $c$'s introduced above.

(C) These statements remain true for the total spaces 
$\mathbf {PHo}_{\mathbf B}^{(v,a)}$ and
$\mathbf {PHo}_{\mathbf B}^{(s)}$ obtained by summing up all the 
corresponding previous spaces defined for $Q$'s which are in the
real span of independent vector-system $\mathbf B$. Since the projections 
$\Pi_X$ and $\Pi_K$
commute, they remain true also for spaces  
$\mathbf {PXo}_{\mathbf B}^{(v,a)}$ and
$\mathbf {PXo}_{\mathbf B}^{(s)}$ (as well as for versions defined
for fixed $Q$'s) obtained by applying $\Pi_X$ to the corresponding 
spaces $\mathbf {PHo}$. 
\end{theorem}
\begin{proof} 
The independence-statement in (B) follows also from (\ref{P^s_q}), because
functions $|[Q,X]|^{2c}$ regarding distinct powers $2c$ are independent
and, for a fixed $\tilde K_u$, also functions  $P_c^{(s)}(X,\tilde K)$
of the X-variable expressed by means of complex structure $J_{\tilde K_u}$
as a complex valued function has distinct real and imaginary degrees
with respect to distinct $a$'s. That is, these functions defined for a
fixed $a$ can not be a linear combination of the others defined for other
$a$'s.

Statement (C) can be established by an appropriate generalization of
(\ref{P^s_q}).
Instead of Q-pole functions, now functions belonging to
$\mathbf{PHo}_{\mathbf B}^{(v,a_i)}$ should be projected, where 
$a=\sum a_i$ and degree $a_i$ regards $B_i$. In this
situation, one gets functions 
$\langle [B_i,X],[B_j,X]\rangle^{v_{ij}}$ 
multiplied with the corresponding functions $P^{(s)}_{c_{ij}}(X,V)$,
which, for a fixed $\tilde K_u$ and system $c_{ij}$ of exponents 
have distinct real and imaginary
degrees regarding $J_{\tilde K_u}$ with respect to distinct  
$\sum_i a_i=a\not =a^\prime =\sum_ia^\prime_i$.
\end{proof}

\subsection{Twisted Hankel decomposition.}

In order to construct the complete pre-spaces 
$\mathbf{P\Phi}^{(n)}_Q$, $\mathbf{P\Phi}^{(n)}_{\mathbf B}$,
$\overline{\mathbf{P\Phi}}^{(n)}_Q$, and 
$\overline{\mathbf{P\Phi}}^{(n)}_{\mathbf B}$, 
by means of functions in
$\mathbf{PHo}^{(n,v,a)}_Q$ resp. $\mathbf{PHo}^{(n,v,a)}_{\mathbf B}$,
they must be multiplied with functions of the form $\phi (|X|,|K|)$ which 
multiplied with $|K|^{v+a}$ incorporated to the $\mathbf{PHo}$-functions
provide K-radial $L^2_K$-functions for any fixed $|X|=\mathbf x$. 
These functions
summed up regarding $(v,a)$ span the corresponding $n^{th}$-order twisted
spaces whose  $L^2_K$-closures provide the complete space which can  
also be introduced by straightly defined functions. 
Since they are everywhere dense in the straightly defined function spaces,
the twisted spaces
must be complete regarding the $\mathbf{PHo}$-spaces.
Thus the completion of twisted
spaces is ultimately implemented on the space of K-radial functions.

Projections $\Pi^{(s)}_K$ should be defined for functions belonging to
$\mathbf{P\Phi}^{(n,v,a)}_Q$ resp. $\mathbf{P\Phi}^{(n,v,a)}_{\mathbf B}$,
first. Note that this operation has no effect on radial functions. It's
action is restricted to the spherical harmonics defined by restricting
the above homogeneous K-harmonic polynomials to the unit K-sphere over
each point $X$. By this interpretation, these projection operators
can be expressed as polynomials of the Laplacian $\Delta_{K_u}$ defined
on this unit sphere. The function spaces obtained by projecting the
whole corresponding ambient spaces are denoted by 
$\mathbf{P\Phi}^{(n,s)}_Q$ resp. $\mathbf{P\Phi}^{(n,s)}_{\mathbf B}$.
There is described in the previous theorem that which functions labeled
by $(v,a)$ provide non-trivial $s$-components in these operations. They
provide direct sum decompositions of the ambient spaces which is called
also {\it pre-Hankel decomposition}. When also $\Pi^{(n)}_X$ (which 
operator commutes with  $\Pi^{(s)}_K$, 
for all $s$) is acting, the obtained
pre-Hankel spaces are denoted by 
$\mathbf{P\Xi}^{(n,s)}_Q$ resp. $\mathbf{P\Xi}^{(n,s)}_{\mathbf B}$, where
$n$ still indicates the degree of the involved homogeneous polynomials 
regarding the X-variable.

The Z-Fourier transforms
$\int e^{\mathbf i\langle Z,K\rangle}\mathbf{P\Phi}^{(n,s)}_QdK$ resp. 
$\int e^{\mathbf i\langle Z,K\rangle}\mathbf{P\Phi}^{(n,s)}_{\mathbf B}dK$
define the {\it $s^{th}$-order twisted Hankel spaces} 
$\mathbf{H\Phi}^{(n,s)}_Q$ resp. $\mathbf{H\Phi}^{(n,s)}_{\mathbf B}$, 
which are projected to 
$\mathbf{H\Xi}^{(n,s)}_Q$ resp. $\mathbf{H\Xi}^{(n,s)}_{\mathbf B}$ 
by $\Pi^{(n)}_X$.
The direct sums of these non-complete subspaces define the corresponding
{\it total twisted Hankel spaces}.
They are different from the pre-spaces, 
but also everywhere
dense subspaces of the corresponding straightly defined function spaces.
Thus the closure of these twisted spaces provides again the whole 
straightly defined spaces.

\subsection{Twisted Dirichlet and Z-Neumann functions.}
All above constructions are implemented by using the whole center
$\mathbb R^l$.
In this section twisted functions  
satisfying the Dirichlet or Z-Neumann
condition on the boundary, $\partial M$, of 
ball$\times$ball- type domains, $M$, are explicitly constructed
by the method described in the review of this section. That is, they are
represented by twisted Z-Fourier transforms of Dirac type generalized 
functions concentrated on $\partial M$. Due to this representation,
the eigenfunctions of the exterior operator $\OE$ satisfying given
boundary conditions can explicitly be computed.

In the first step of this process consider a sphere, 
$S_R$, of radius $R$ around the origin of the Euclidean space
$\mathbf R^l$. As it is well known, 
the Dirichlet or Neumann eigenfunctions of the Euclidean 
Laplacian $-\Delta_{\mathbb R^l}$
on the ball $B_R$ bounded by $S_R$ appear as products of  
$s^{th}$-order spherical harmonics $\varphi^{(s)}(K_u)$ with radial 
functions $y_i^{(s)}(\mathbf z)$. For $s=0$, these eigenfunctions
are radial taking $1$ at the origin and having multiplicity $1$. For
$s>0$, the radial functions take $0$ at the origin and the multiplicity,
for fixed $s$ and $i$, is equal to the dimension of the space of
$s^{th}$-order spherical harmonics $\varphi^{(s)}(K_u)$. 
Corresponding to the Dirichlet or Neumann conditions, 
these eigenvalues are denoted by $\lambda_{Di}^{(s)}$
and $\lambda_{Ni}^{(s)}$ respectively. For any fixed $s$ and condition
$D$ or $N$, these infinite sequences satisfy 
$\lambda_i^{(s)}\uparrow\infty$ and, 
except for $0=\lambda_{N1}^{(0)}$, also  
the relation $0<\lambda_{i}^{(s)}$ holds. 

Eigenfunctions corresponding to Dirichlet or Neumann eigenvalues
$\lambda_i^{(s)}=\lambda$ can be represented by the integral formula
\begin{equation}
y_i^{(s)}(\mathbf z)\varphi^{(s)}(Z_u)=\oint_{S_{\sqrt{\lambda} }}
e^{\mathbf i\langle Z,V\rangle}
\varphi^{(s)}(K/\mathbf k)dK_{no},
\end{equation}
where $dK_{no}$ is the normalized 
integral density on the sphere of radius  
$\sqrt{\lambda}$. Apply $-\Delta_{\mathbb R^l}$ on the right
side to see that this function is an eigenfunction of this operator 
with eigenvalue $\lambda_i^{(s)}:=\lambda$, which, because of the 
Hankel transform, must appear as the function being on the left side.
This formula strongly relates to the third version of
the twisted Z-Fourier transforms which is introduced also in 
the review of this section. However, it can be directly applied
there only after decomposing the functions behind the integral sign by
the projections $\Pi^{(r,s)}_{K_u}$.

Before this application, functions
$y_i^{(s)}(\mathbf z):=y(t)$, belonging to an eigenvalue $\lambda$, 
are more explicitly determined as follows. By formulas
\begin{equation}
\Delta_Z=\partial_{t}\partial_{t}+{l-1\over t}\partial_{t}+
{\Delta_S\over t^2},\quad \Delta_S\varphi^{(s)}=-{s(s+l-2)}\varphi^{(s)},
\end{equation}
it satisfies the differential equation
\begin{equation}
y^{\prime\prime}+{l-1\over t}y^\prime +
\{\lambda -{{s(s+l-2)}\over t^2}\}y=0,
\end{equation}
which, after the substitutions $\tau =\sqrt\lambda t$ and $y(t)=z(\tau )$,
becomes
\begin{equation}
z^{\prime\prime}+{l-1\over \tau}z^\prime +
\{1 -{{s(s+l-2)}\over \tau^2}\}z=0.
\end{equation}
That is, function $J(\tau )=\tau^{l/2-1}z(\tau )$ satisfies
the ordinary differential equation
\begin{equation}
J^{\prime\prime}+{1\over \tau}J^\prime +
\{1 -{{(2s+l-2)^2}\over 4\tau^2}\}J=0,
\end{equation}
therefore, it is a bounded Bessel function of order $(s+l/2-1)$. Thus,
except for complex multiplicative constant,
equation $J=J_{s+l/2-1}$ must hold.

In order to find the functions satisfying the boundary conditions,
consider spheres $S_{\lambda (\mathbf x^2)}$ of radius 
$\lambda (\mathbf x^2)$
around the origin of the Z-space. For appropriate functions
$\phi (\mathbf x^2)$ and $\varphi (K_u)$, the twisted Z-Fourier 
transform on the sphere bundle
$S_\lambda$ is defined by: 
\begin{equation}
\label{FourLamb_Z}
\mathcal F_{Qpq\lambda}(\phi\varphi )(X,Z)
=\oint_{S_{\sqrt\lambda}}e^{\mathbf i
\langle Z,K\rangle}\phi (\mathbf x^2)\varphi (K_u)
(\Theta_{Q}^p\overline\Theta^q_{Q})(X,K_u)dK_{no},
\end{equation} 
where $dK_{no}=dK/Vol(\sigma_{\lambda})$ is 
the normalized measure on the sphere
and function $\varphi (K/\mathbf k)\Theta_{Q}^p
\overline\Theta^q_{Q}(X,K/\mathbf k)$
is defined on $S_\lambda$. 
Since no Hankel projections are involved, 
these functions do not satisfy the required
boundary conditions, yet. 
However, by the above arguments we have:

\begin{theorem}
Consider a ball$\times$ball-type domain defined by the Z-balls 
$B_{R(\mathbf x^2)}$ and let $\lambda_i^{(s)}(\mathbf x^2)
:=\lambda(\mathbf x^2)$
be a smooth function defined by the $i^{th}$ eigenvalue in 
the $s^{th}$-order Dirichlet or
Neumann spectra of the Euclidean balls $B_{R(\mathbf x^2)}$. 
Then function (\ref{FourLamb_Z}) defined for
$\lambda =\lambda_i^{(s)}(\mathbf x^2)$, or
\begin{equation}
\label{XFourLamb_Z}
 \oint_{S_{\sqrt{\lambda}}}e^{\mathbf i
 \langle Z,K\rangle}\phi (|X|^2)\pi_K^{(s)}(\beta^{(m)}
 \Pi^{(n)}_X(\Theta_{Q}^p\overline\Theta^q_{Q}))(X,K_u)dK_{no}
 \end{equation}
satisfy the Dirichlet resp. Z-Neumann condition on the domain $M$.  
The restrictions of these functions onto $M$ span 
$\mathbf\Phi^{(n,s)}_Q(M)$
resp. $\mathbf\Xi^{(n,s)}_Q(M)$,  
for any fixed boundary condition.
If functions $\Theta_{Q}^p\overline\Theta^q_{Q}$ are exchanged for the
polynomials
$\prod_{i=1}^{\kappa}z^{p_i}_{K_ui}(X)\overline z^{q_i}_{K_ui}(X)$,
then the above construction provides the Dirichlet-, Z-Neumann-, resp.
mixed-condition-functions spanning
$\mathbf\Phi^{(n,s)}_{\mathbf B}(M)$
resp. $\mathbf\Xi^{(n,s)}_{\mathbf B}(M)$.
\end{theorem}

\subsection{Constructing the orbital and inner force operators.}
The complicated action of the compound angular momentum operator
$\mathbf M_{Z}$ on twisted Hankel functions is due to fact that  
its Fourier transform, $\mathbf iD_K\bullet$, is non-commuting
with the Hankel projections, that is, the commutator on the
right side of 
$D_K\bullet \Pi_{K_u}^{\alpha} =\Pi_{K_u}^{\alpha} D_K\bullet +
[D_K\bullet \Pi_{K_u}^{\alpha}]$ is non-vanishing in general.
Thus there are two non-trivial terms on the right side of equation 
\begin{eqnarray}
\mathbf M_{Z}\int e^{\mathbf i\langle Z,K\rangle}
f_{\alpha}(\mathbf x,\mathbf k)
\Pi_{K_u}^{\alpha} (F^{(p_i,q_i)}(X,K_u))dK=
\\
=\mathbf i\int e^{\mathbf i\langle Z,K\rangle}
f_{\alpha}(\mathbf x,\mathbf k)D_K\bullet
\Pi_{K_u}^{\alpha} (F^{(p_i,q_i)}(X,K_u))dK=
\nonumber
\\
=\mathbf i\int e^{\mathbf i\langle Z,K\rangle}
f_{\alpha}(\mathbf x,\mathbf k) (\Pi_{K_u}^{\alpha} D_K\bullet +
[D_K\bullet ,\Pi_{K_u}^{\alpha}]) (F^{(p_i,q_i)}(X,K_u))dK,
\nonumber
\end{eqnarray}
by which the orbital: 
\begin{eqnarray}
\mathbf{L}_{Z}\int e^{\mathbf i\langle Z,K\rangle}
f_{\alpha}(\mathbf x,\mathbf k)
\Pi_{K_u}^{\alpha} (F^{(p_i,q_i)}(X,K_u))dK=
\\
=\mathbf i\int e^{\mathbf i\langle Z,K\rangle}
f_{\alpha}(\mathbf x,\mathbf k) \Pi_{K_u}^{\alpha} D_K\bullet 
(F^{(p_i,q_i)}(X,K_u))dK
\nonumber
\end{eqnarray}
and the intrinsic spin operators:
\begin{eqnarray}
\mathbf{S}_{Z}\int e^{\mathbf i\langle Z,K\rangle}
f_{\alpha}(\mathbf x,\mathbf k)
\Pi_{K_u}^{\alpha} (F^{(p_i,q_i)}(X,K_u))dK=
\\
=\mathbf i\int e^{\mathbf i\langle Z,K\rangle}
f_{\alpha}(\mathbf x,\mathbf k) [D_K\bullet ,\Pi_{K_u}^{\alpha}] 
(F^{(p_i,q_i)}(X,K_u))dK
\nonumber
\end{eqnarray}
are defined, respectively. 

The commutator appears in the following explicit form
\begin{equation}
\label{[DPI]}
[D_K\bullet ,\Pi_{K_u}^{\alpha}]=
S_{\beta}^\alpha\Pi_{K_u}^{\beta}\mathbf M_{K^\perp_u}=
\mathbf M_{K^\perp_u}S_{\beta}^\alpha\Pi_{K_u}^{\beta},
\end{equation}
where 
$\mathbf M_{K^\perp_u}=\mathbf M_{K}-\partial_{\mathbf k}D_{K_u}\bullet$
is the Z-spherical angular momentum operator defined on Z-spheres. (For
a fixed unit vector $K_u$ and orthonormal basis $\{e_\alpha ,K_u\}$ 
(where $\alpha =1,2,\dots ,l-1$) of K-vectors, this operator is of the 
form $\mathbf M_{K^\perp_u}=\sum_\alpha\partial_{\alpha}D_{\alpha}\bullet$.
)
Formula (\ref{[DPI]}) follows by applying relations 
$D_{K}\bullet\Delta_{K_u}=\Delta_{K_u}D_{K}\bullet -2\mathbf M_{K^\perp_u}$
and the commutativity of $\Delta_{K_u}$ with operators
$\partial_{\alpha}, \partial_{K_u}, D_{\alpha}\bullet$, and 
$D_{K_u}\bullet$ in order to evaluate
\[
D_{K}\bullet\Pi_{K_u}^{(r,r-2i)}=D_{K}\bullet\tilde
D_{(r,r-2i)}\Pi_K^{(r-2i,r-2i)}\Delta_{K_u}^i,
\]
where, according to the denotations introduced above, $r=v+a$ holds.
This computation results the equation 
$[D_{K}\bullet,\Pi_{K_u}^{(r,r-2i)}]=P_{i+1}(\Delta_{K_u})
\mathbf M_{K^\perp_u}$,
where the lowest exponent of $\Delta_{K_u}$ in the polynomial
$P_{i+1}(\Delta_{K_u})$ is $i+1$. 
This term defines a uniquely
determined constant times of projection $\Pi_{K_u}^{(r,r-2(i+1))}$
such that 
$P_{i+1}(\Delta_{K_u})=A_{i+1}\Pi_{K_u}^{(r,r-2(i+1))}+
P_{i+2}(\Delta_{K_u})$ holds,
where the lowest exponent of $\Delta_{K_u}$ in the polynomial
$P_{i+2}(\Delta_{K_u})$ is $i+2$. 
In the next step, the above arguments are repeated regarding  $P_{i+2}$
to obtain $P_{i+3}$. This process can be concluded in finite many steps
which establish the above formula completely.

Formula (\ref{[DPI]}) allows to define an {\it inner algorithm} where these
computations are iterated infinitely many times. 
In the second step it is repeated for
$f^{(2)}_\beta :=S_{\beta}^\alpha f^{}_\alpha$ as follows. 
First note that operator 
$
f_\alpha S_{\beta}^\alpha\Pi_{K_u}^{\beta}\mathbf M_{K^\perp_u}=
\mathbf M_{K^\perp_u} \Pi_{K_u}^{\beta} S_{\beta}^\alpha f_\alpha
$
can be decomposed in the following form:
\begin{eqnarray}
\label{inop}
f_{\beta}^{(2)}\Pi_{K_u}^{\beta}\mathbf M_{K^\perp_u}=
-\partial_{K_u}(f_\beta^{(2)})\Pi_{K_u}^{\beta}
D_{K_u}\bullet +
\mathbf M_{K}\Pi_{K_u}^{\beta}f_{\beta}^{(2)} .
\end{eqnarray}
The action of the first operator on functions 
\begin{eqnarray}
F^{(p,q)}(X,K_u)=\varphi_\zeta^{(r)}(K_u)
\Theta_{Q}^p(X,K_u)\overline\Theta^q_{Q}(X,K_u)
\\
{\rm resp.}\quad
F^{(p,q)}(X,K_u)=\varphi_\zeta^{(r)}(K_u)
\prod_{i=1}^{k/2}z^{p_i}_{K_ui}\overline z^{q_i}_{K_ui},
\end{eqnarray}
where $\varphi_\zeta^{(r)}(K_u)$ denotes an $r^{th}$-order homogeneous
harmonic polynomial introduced at explicit description of Hankel 
projections, results
\[
-(p-q)\mathbf i \partial_{K_u}(f_\alpha S_{\beta}^\alpha )
\Pi_{K_u}^{\beta}(F^{(p,q)}).
\]
This term is already in finalized form which does not alter
during further computations. Together with the orbiting operator, they
define, in terms of the radial operator 
\begin{equation}
\bigcirc^{(1)}_\alpha (f_{\beta_1},\dots ,f_{\beta_d})=
-(p-q)\mathbf i({\mathbf k}f_\alpha + 
\partial_{K_u}(f_\beta S_{\alpha}^\beta )),
\end{equation}
the one-turn operator
\begin{eqnarray}
\mathbf M^{(1)}_{Z}\int e^{\mathbf i\langle Z,K\rangle}
f_{\alpha}(\mathbf x,\mathbf k)
\Pi_{K_u}^{\alpha} (F^{(p_i,q_i)}(X,K_u))dK=
\\
=\int e^{\mathbf i\langle Z,K\rangle}\bigcirc^{(1)}_\alpha
(f_{\beta_1},\dots ,f_{\beta_d})
 \Pi_{K_u}^{\alpha} (F^{(p_i,q_i)}(X,K_u))dK,
\nonumber
\end{eqnarray}
where the name indicates that it is expressed in terms of the first 
power of the inner spin operator $S_{\alpha}^\beta$ permutating the Hankel
radial functions. 

Be aware of the novelty of this spin-concept emerging in
these formulas! It defines just permutation of the radial Hankel
functions to which no actual spinning of the particles can be corresponded.
This concept certainly does not lead to a dead-end-theory like those
pursued in 
classical quantum theory where one tried to explain the inner spin of 
electron by actual spinning.

This abstract merry-go-round does not stop after making one turn.
It is actually the second operator,
$\mathbf M_{K}\Pi_{K_u}^{\beta}f_\beta^{(2)}$,
in (\ref{inop}) which generates the indicated process   
where the above arguments are repeated for functions 
$f^{(2)}_\beta :=S_{\beta}^\alpha f_\alpha$. 
Index $2$ indicates that these 
functions are obtained from the starting functions  
$f^{(1)}_\alpha :=f_\alpha$ by the next step. The details are as follows.

In these computations the second operator is derived by the Hankel
transform turning functions defined on the $\tilde K$-space to functions
defined on the $K$-space. Operator $H^{(-s_\beta)}$ denotes the inverse
of the Hankel transform $H_{s_\beta}^{(l)}$,
where $s_\beta$ denotes the third index in $\beta =(v,a,s)$. Then we have:
\begin{eqnarray}
\mathbf M_{K}f_\beta^{(2)} 
\Pi_{K_u}^{\beta}F^{(p,q)}(X,K_u)=
\\
=\mathbf M_{K}\int e^{\mathbf i\langle K,\tilde K\rangle}
H^{(-s_\beta)}(f^{(2)}_\beta )(\mathbf x,\tilde{\mathbf k})
\Pi_{\tilde K_u}^{\beta}
F^{(p,q)}(X,\tilde K_u)d\tilde K=
\nonumber
\\
=\mathbf i\int e^{\mathbf i\langle K,\tilde K\rangle}
 \tilde f^{(2)}_{\beta} (\mathbf x,
\tilde{\mathbf k})D_{\tilde K}\bullet
\Pi_{\tilde K_u}^{\beta}F^{(p,q)}(X,\tilde K_u)d\tilde K,
\nonumber
\end{eqnarray}
where $\tilde f^{(2)}_{\beta} =H^{(-s_\beta)}(f^{(2)}_\beta ) $.
At this step commutator $[D_{\tilde K}\bullet ,\Pi_{\tilde K_u}^{\beta}]$
can be calculated by (\ref{[DPI]}), resulting functions 
$f^{(3)}_\beta :=S_{\beta}^\alpha\tilde f^{(2)}_{\alpha}$ 
which are subjected to the operations performed in the following step 3. 
These steps must infinitely many times be iterated.

Regarding radial functions $\tilde f^{(2)}_\alpha$, 
which are defined on the $\tilde K$-space,
the orbiting spin and the one-turn operator can be defined in the same way 
as they are defined on the  $K$-space. 
After performing the Z-Fourier and the associated Hankel transforms 
on these finalized functions they
become functions defined on the $K$-space. By adding these
terms to those obtained in the first step, one defines the two-turn
operator $\mathbf M^{(2)}_{Z}$ associated with 
$\bigcirc^{(2)}_\alpha (f_{\beta_1},\dots ,f_{\beta_d})$,
called two-turn merry-go-round and roulette operators respectively.
The sum of the orbiting spin operators defines $\mathbf L^{(2)}$, which
is called one-turn orbiting spin operator involving just $S^\alpha_\beta$
into its definition. 
One should keep in mind that these operators involve also the operators 
defined in the previous step. These computations work out for an arbitrary 
$u^{th}$-step defining operators  
$\mathbf M^{(u)}_{Z}$, 
$\bigcirc^{(u)}_\alpha (f_{\beta_1},\dots ,f_{\beta_d})$, and 
$\mathbf L^{(u)}$ which are called u-turn merry-go-round-, roulette-,
and $(u-1)$-turn orbiting spin operator respectively.
In the end, the action of 
$\mathbf M_{Z}=\mathbf M^{(\infty )}_{Z}$ can be described
in the form:
\begin{eqnarray}
\mathbf M_{Z}\int e^{\mathbf i\langle Z,K\rangle}
f_{\alpha}(\mathbf x,\mathbf k)
\Pi_{K_u}^{\alpha} (F^{(p_i,q_i)}(X,K_u))dK=
\\
=\int e^{\mathbf i\langle Z,K\rangle}\bigcirc_\alpha
(f_{\beta_1},\dots ,f_{\beta_d})
 \Pi_{K_u}^{\alpha} (F^{(p_i,q_i)}(X,K_u))dK,
\nonumber
\end{eqnarray}
where operator 
$\bigcirc_\alpha (f_{\beta_1},\dots ,f_{\beta_d})=
\bigcirc^{(\infty )}_\alpha (f_{\beta_1},\dots ,f_{\beta_d})$,
called high roulette operator, is defined
by the infinite series  $\lim_{u\to\infty}\bigcirc^{(u)}_\alpha$.
In this sense, the complete angular momentum operator $\mathbf M_{Z}$
can be called high merry-go-round operator. 
The point in this formula is that this action can be described
in terms of d-tuples, $(f_{\beta_1},\dots ,f_{\beta_d})$, of radial
functions which can not be reduced to a single function defined by a 
fixed $\alpha$.
  
This statement holds also for the total Laplacian $\Delta$, thus the 
eigenfunctions satisfying the Dirichlet or Z-Neumann  
conditions on the considered manifolds can also be completely described
in terms of radial functions. But the corresponding equations involve
all radial functions $(f_{\beta_1},\dots ,f_{\beta_d})$ defined for
all indices $\beta_i$. By this reason, operator $\bigcirc_\alpha$
plays the role of a confider, giving rise to the effect that the
eigenfunctions satisfying a given boundary condition
are expressed in terms of radial functions which do not 
satisfy these conditions individually but exhibit them
together, confined in a complicated combination. Neither
can these eigenfunctions be observed as single $\OE$-force potential 
functions. On the other hand, for a given boundary condition, 
the $\OE$-eigenfunctions span the whole $L^2$-space, therefore,
$\Delta$-eigenfunctions satisfying the same boundary condition
can be expressed as infinite linear combinations of $\OE$-eigenfunctions.
This observation further supports the idea that the rather
strong $\Delta$-forces are built up by the much weaker $\OE$-forces in
a way how Hawking describes the action of the Weinberg-Salam weak force 
in \cite{h}, pages 79:

``The Weinberg-Salam theory known as spontaneous symmetry breaking.
This means that what appear to be a number of completely different
particles at low energies are in fact found to be all the same type of 
particle, only in different states. At high energies all these
particles behave similarly. The effect is rather like the behavior
of a roulette ball on a roulette wheel. At high energies (when the wheel
is spun quickly) the ball behaves in essentially only one way -- it rolls
round and round. But as the wheel slows, the energy of the ball decreases,
and eventually the ball drops into one of the thirty-seven slots 
in the wheel. In other words, at low energies there are thirty seven 
different states in which the ball can exist. If, for some reason, 
we could only observe the ball at low energies, we would then think
that there were thirty-seven different types of ball!"

This quotation explains the origin of the name given to the
roulette operators. This polarized operators are derived from the
non-polarized merry-go-round operators, which name was suggested to me
by Weinberg's book \cite{w1} where the name ``merry-go-round" appears 
in a different situation not discussed here. 
All these arguments clearly suggest
that the $\Delta$-eigenfunctions must correspond to the strong forces
keeping the quarks together. A formula expressing these eigenfunctions
as linear combinations of weak force potential functions would shed light
to the magnitude of the strong force piled by the roulette operator during 
building up these eigenfunctions.  

The explicit eigenfunction computations give
rise to difficult mathematical problems which are completely open
at this point of the developments. In this paper we explicitly describe
only the eigenfunctions of the exterior orbiting operator $\OE$ defined
by omitting the interior spin operator from the $\Delta$. This operator is 
a scalar operator whose action can be reduced to a single radial function
$f$. By the arguments developed in the introductions and also at several
points in the body of this paper, this operator
is associated with the weak force interaction. 
These computations are established in the next section. 
This section is concluded by saying more about the recovery of several 
concepts of the standard model within this new theory.

Suppose that function
\begin{eqnarray}
\psi (X,Z)=\int e^{\mathbf i\langle Z,K\rangle}
f_{\alpha}(\mathbf x,\mathbf k)
\Pi_{K_u}^{\alpha} (F^{(p_i,q_i)}(X,K_u))dK
\end{eqnarray}
is an eigenfunction of the complete Laplacian $\Delta$. It is called
also probability amplitude of the particle system. 
Also in this formula the Einstein
convention indicates summation. A fixed index $\alpha =(v,a,s)$
is called slot-index and function
\begin{eqnarray}
\mathcal Q_{vas} (X,Z)=\int e^{\mathbf i\langle Z,K\rangle}
f_{vas}(\mathbf x,\mathbf k)
\Pi_{K_u}^{(vas)} (F^{(p_i,q_i)}(X,K_u))dK
\end{eqnarray}
is the so called high slot probability amplitude. These exact
mathematical objects correspond
to quarks whose flavor is associated with index $v$ and its color is 
associated with index $a$. Index $s$ is called azimuthal
index. Note that also these indices have mathematical meanings, namely,
they refer to the degrees of the corresponding polynomials by which
the formula of a slot-amplitude is built up. 
According to these definitions, 
a particle-system-amplitude is the sum of the high slot probability 
amplitudes
which are considered to be the mathematical manifestations of quarks. 
Slot amplitudes defined in the same way by the $\OE$-eigenfunctions are 
called retired or laid-off slot probability amplitudes. 

It is very interesting to see how these 
constituents of a high energy particle are held together by the 
strong force interaction. 
When $D_K\bullet$ is acting on functions $\langle Q,X\rangle$ resp.
$\langle J_K(Q),X\rangle$ of a quark, then the first one becomes of the
second type and the second one becomes of the first type. In either
cases the color index, $a$, changes by $1$ and a slot-amplitude of odd
color index becomes an amplitude of even color index. This process can be
interpreted as gluon exchange in the following way.
Action of $D_K\bullet$ on $\langle Q,X\rangle$ resp.
$\langle J_K(Q),X\rangle$ are interpreted
as gluon absorption resp. emission. More precisely, some of the 
slot-particles (quarks)
of odd color index emit gluon which is absorbed by some of the
slot-particles (quarks)
which also have odd color index. As a result they become quarks
of even color index. The same process yields also for quarks having
even color index, which, after gluon exchange, become quarks of odd
color index. This is a rather clear explanation for the flavor-blindness
and color sensitiveness of gluons. 

It also explains why a high slot-particle (quark) can 
never retire  
to become a laid-off $\OE$-particle. Indeed, in a high 
slot-state defined for fixed slot index $\alpha =(v,a,s)$ the 
corresponding Hankel function does not satisfy the boundary conditions,
however, it can be expanded by the $\OE$-eigenfunctions. 
At this point nothing
is known about the number of $\OE$-eigenfunctions by which a 
high slot-amplitude can be expressed. This number can very well be equal 
to the infinity, but it is always greater then one. 
Then, instead of consisting of a single term, the high slot 
probability density is a multi-term sum of weak 
densities determined by these laid off probability amplitudes.    

Real positive function $\psi\overline\psi$ whose integral on the whole
ball$\times$ball-type domain is $1$ is called probability density. Protons
are defined by functions $\mathit{ch}(Z_u)(\psi_c\overline\psi_c)(Z,X)$,
where $\psi_c$ is a constant times of $\psi$, whose integral on the 
ball$\times$ball-type domain is $1$. If this integral is $0$ or $-1$,
it is called neutron or antiproton respectively. (The mass can be defined
by means of $|\det (A_{ij}(Z_u))|$, but we do not go into these details
here.) This argument shows
that in a complete eigensubspace decomposition of the $L^2$ function
space of a two-step nilpotent Lie group representing a particle system
the eigenfunctions, actually, represent all kind of particles and not just
particular ones. This phenomenon can be considered as a clear 
manifestation of the bootstrap principle, from which super string 
theory grew out, also in the new theory. In super string theory 
the notion was (cf. \cite{g}, pages 128)  
that ``a set of elementary particles could be treated as if
composed in a self-consistent manner of combinations of those same
particles. All the particles would serve as constituents, all the 
particles (even the fermions in a certain sense) would serve as quanta
for force fields binding the constituents together, and all the particles
would appear as bound states of the constituents".

\subsection{$\OE$-forces in the union of 3 fundamental forces.}
 
The eigenfunction computations of $\OE$ on the twisted function space
$\Xi_{.R}^{(n)}$ satisfying
a given boundary condition can be reduced to the same radial 
differential operator what was obtained
for the standard Ginsburg-Landau-Zeeman
operator on Z-crystals. To see this statement, 
let the complete Laplacian (\ref{Delta}) act on (\ref{XFourLamb_Z}).
By the commutativity relation $\mathbf M_Z\oint =\oint \mathbf iD_K\bullet$
and $\Delta_Z\oint =-\oint \mathbf k^2$,
this action is a combination of
X-radial differentiation, $\partial_{\mathbf x}$, 
and multiplications with functions depending 
just on $\mathbf x$, that is, it is completely reduced to X-radial
functions. More precisely,  
\begin{equation}
\Delta (\mathcal {HF}_{QpqR}(\phi\beta ))
=\oint_{S_{R}}e^{\mathbf i
\langle Z,K\rangle}\Diamond_{R,\mathbf x^2}(\phi )\beta
\Pi^{(n)}_X(\Theta_{Q}^p\overline\Theta^q_{Q})dK_{no}
\end{equation}
holds, where
\begin{eqnarray}
 (\Diamond_{R,\mathbf x^2}\phi ) (\mathbf x^2)
=4\mathbf x^2\phi^{\prime\prime}(\mathbf x^2)
+(2k+4(p+q))\phi^{\prime}(\mathbf x^2)+\\
-(R^2(1+{1\over 4}\mathbf x^2) +(p-q)R)
\phi (\mathbf x^2),
\nonumber
\end{eqnarray}
and $\phi^\prime\,\, , \phi^{\prime\prime}$ mean the corresponding 
derivatives of $\phi (\tilde t)$ with respect to the $\tilde t$ variable.
The eigenfunctions of $\Delta$ can be found by seeking
the eigenfunctions of the reduced operator
$\Diamond_{R,\mathbf x^2}$ among the X-radial functions.

Note that no projections $\Pi^{(s)}_K$ are applied in the above
integral formula, thus these eigenfunctions do not satisfy the
boundary conditions in general. These projections, however, do
not commute with $D_K\bullet$, and the eigenfunctions of the
complete operator $\Delta$ can not be expressed in terms of a single 
function $\phi (\mathbf x^2)$. This simple
reduction applies just to the exterior operator, 
$\OE$, defined by neglecting the anomalous intrinsic momentum operator 
$\mathbf{S}_Z$ from $\Delta$ and keeping just the orbital 
(alias, exterior)
spin operator $\mathbf{L}_Z$. Regarding this operator we have:

\begin{theorem}
The exterior operator $\OE$, 
on constant radius Z-ball bundles 
reduces to a radial operator appearing in terms of the 
Dirichlet-, Neumann-, resp. mixed-condition-eigenvalues 
$\lambda_i^{(s)}$ of the 
Z-ball $B_Z(R)$ in the form 
\begin{equation}
\label{BLf_mu}
(\Diamond_{\lambda ,\tilde t}f)(\tilde t)=
4\tilde tf^{\prime\prime}(\tilde t)
+(2k+4n)f^\prime (\tilde t)
-(2m\sqrt{\lambda_i^{(s)}\over 4} +4{\lambda_i^{(s)}\over 4} (1 +
{1\over 4}\tilde t))f(\tilde t).
\end{equation}
By the substitution 
$\lambda=\sqrt{\lambda_i^{(s)}/4}$, this is exactly
the radial Ginsburg-Landau-Zeeman operator 
(\ref{Lf_lambda}) obtained on Z-crystal models. 

Despite this formal identity with electromagnetic forces, the nuclear
forces represented by $\OE$ manifest themself quite differently.
Like for the  electromagnetic forces, one can introduce both charged
and neutral particles also regarding $\OE$. A major 
difference between the two particle-systems is that the particles 
represented by $\OE$ are extended ones. 
This property can be seen, for instance, by the eigenfunctions of $\Delta$ 
which involve also Z-spherical harmonics.
Due to these harmonics, the multiplicities of the eigenvalues regarding
$\OE$ are higher than what is corresponded to the same eigenvalue 
regarding the Ginsburg-Landau-Zeeman operator on Z-crystals.
An other consequence of the extension is that the  
$\OE$-neutrinos always have positive mass, which is zero for neutrinos 
associated with electromagnetic forces.
 
Actually, the nuclear forces represented by $\OE$ are weaker than the
electromagnetic forces. This phenomenon can be explained, by the extension,
by the very same argument of classical electrodynamics 
asserting that the electromagnetic self-mass for a surface 
distribution of
charge with radius $a$ is $e^2/6\pi ac^2$, which, therefore, blows up
for $a\to 0$. Note that, in the history of quantum theory, 
this argument provided the first warning that a point
electron will have infinite electromagnetic self-mass. It is well known
that this problem appeared with even much greater severity in
the problem of infinities invading quantum field theory. The tool
by which these severe problems had been handled is renormalization, which
turned QED into a renormalizable theory. The above argument
clearly suggests that also the $\OE$-theory must be renormalizable.

The major difference between the $\OE$- and the complete $\Delta$-theory
is that the first one is a scalar theory, which can be reduced to a radial
operator acting on a single radial function, while the reduced operator
obtained in the $\Delta$-theory acts on d-tuples of radial functions.
As it is described above, this action defines also a new type of 
nuclear inner spin of the extended particles to which new particles 
such as quarks and gluons can be associated and by which strong nuclear
forces keeping the parts of the nucleus together can be introduced.
Partial operator $\OE$ is the maximal scalar operator in $\Delta$.
Except for $\OE=\OE^{(1)}$, partial operators $\OE^{(u)}$ 
(defined by replacing $\mathbf M_Z$ 
by $\mathbf L_Z^{(u)}$) and $\Delta^{(u)}$ 
(defined by replacing $\mathbf M_Z$ by $\mathbf M_Z^{(u)}$) can 
be reduced just to operators which irreducibly act on d-tuples of radial
functions.

The main unifying principle among the three fundamental forces is
that they are derived from the very same Hamilton operator, $\Delta$.
More precisely, the Hamilton operators of the individual elementary 
particles emerge on corresponding invariant subspaces of $\Delta$ such 
that it is restricted to the subspace corresponded to the given
elementary particles. The corresponding forces are defined by those
acting among these particles. The electromagnetic forces manifest 
themself on function spaces consisting functions which are periodic 
regarding a Z-lattice $\Gamma_Z$. The systems of particles defined on 
these function spaces are without interior. They
consist of particles such as electrons, positrons, and 
electron-positron-neutrinos. 
The various nuclear forces appear on function spaces defined on Z-ball
and Z-sphere bundles by fixed boundary conditions.
The attached particles, which do have interior, and the forces acting
among them are discussed above. 

The function spaces corresponded to 
particular particles are constructed with the corresponding twisted
Z-Fourier transforms. Since the eigenfunctions of the Hamilton operators
can be sought in this form, this transform seems to be the only natural
tool for assigning the invariant function spaces corresponding to 
elementary particles. It is also remarkable that the twisted Z-Fourier 
transforms emerge also in the natural generalization of the de Broglie
waves fitting the new theory.   

A particle-system defined by a fixed
function space consists of all kind of particles which can be defined
by the characteristic property of the given function space. For
instance, all particles without interior appear 
on a $\Gamma_Z$-periodic function spaces. In order to avoid the 
annihilation of particles by antiparticles, the particles belonging
to the same system are considered to be not interacting with each other.
According to this argument, the particles in a system defined by an 
invariant subspace are gregarious which are always accompanied with
other particles. For instance, an electron is always partying with an
electron-neutrino. This complexity of the particle-systems is reflected
by the Laplacian (Hamiltonian) which appears as the sum of Hamiltonians
of particles partying in a system. This phenomenon is a relative of
those described by the bootstrap principle of super string theory.
 
Let it be mentioned yet that the distinct
function spaces belonging to distinct type of forces are not
independent, thus there is the possibility to work out an interaction
theory between particles of distinct types. The existence of such a
viable theory is a completely open question in this field.

\end{theorem}

The $\OE$-forces are very similar to the weak nuclear forces 
described in the Weinberg-Salam theory. Weinberg introduces these
forces on pages 116-120 of his popular book \cite{w1} as follows:

`` The weak nuclear force first turned up in the discovery of 
radioactivity by Henri Becquerel in 1896. In the 1930's it become
understood that in the particular kind of radioactivity that was 
discovered by Becquerel , known as beta decay, the weak nuclear force
causes a neutron inside the nucleus to turn into a proton, at the same
time creating an electron and another particle known today as 
antineutrino, and spitting them out of the nucleus. This is something 
that is not allowed to happen through any other kind of force. The
strong nuclear force that holds the protons and neutrons together
in the nucleus and the electromagnetic force that tries to push the 
protons in the nucleus apart cannot change the identities of these 
particles, and the gravitational force certainly does not do anything
of the sort, so the observation of neutrons  changing into protons
or protons into neutrons provided evidence of a new kind of force
in the nature. As its name implies, the weak nuclear force is weaker
than the electromagnetic or the strong nuclear forces. This is shown
for instance by the fact that nuclear beta decay is so slow; the 
fastest nuclear beta decays take on the average about a hundredth
of a second; languorously slow compared with the typical time scale of
processes caused by the strong nuclear force, which is roughly a 
millionth millionth millionth millionth of a second.

In 1933 Enrico Fermi took the first significant step toward a theory
of this new force. ... There followed a quarter century of experimental
afford aimed at tying up the loose ends of the Fermi theory.
... In 1957 this [problem] was settled and the Fermi theory of the
weak nuclear force was put into its final form. ... Nevertheless, even
though we had a theory that was capable of accounting for everything
that was known experimentally about the weak force, physicists in general 
found the theory highly unsatisfactory.... The things that were wrong
with the Fermi theory were not experimental but theoretical. ...
when the theory was applied to more exotic processes it gave nonsensical 
results....when they did the calculations the answer would turn out
to be infinite.... Infinities like these had been encountered in the 
theory of electromagnetic forces by Oppenheimer and others in the early
1930's, but in the late 1940's theorists had found that all these 
infinities in quantum electrodynamics would cancel when the mass and
electric charge of the electron are properly defined, or ``renormalized".
As more and more became known about the weak forces it became increasingly
clear that the infinities in Fermi's theory of the weak forces 
would not cancel in this way; the theory was not renormalizable. The
other thing that was wrong with the theory of weak forces was that it
has a large number of arbitrary elements....

I had worked on the theory of weak forces off and on since graduate
school, but in 1967 I was working instead on the strong nuclear forces,
the forces that hold neutrons and protons together inside atomic nuclei.
I was trying to develop a theory of the strong nuclear forces based on
an analogy with quantum electrodynamics. I thought that the difference
between the strong nuclear forces and electromagnetism might be explained
by a phenomenon known as broken symmetry, which I explain later. It did
not work. I found myself developing a theory that did not look like at all
the strong forces as they were known to us experimentally. Then it suddenly
occurred to me that these ideas, although they had turned out to be 
completely useless as far as the strong forces were concerned, provided
a mathematical basis for a theory of weak nuclear forces that would do
anything that one might want. I could see the possibility of a theory
of the weak force analogous to quantum electrodynamics. Just as the
electromagnetic force between distant charged particles is caused by
the exchange of photons, a weak force would not act all at once at a 
single point in space (as in the Fermi theory) but it would be caused
by the exchange of photonlike particles between particles at different
positions. These new photonlike particles could not be massless like the
photon (for one thing, if massless they would have been discovered long
before), but they were introduced into the theory in a way that was so
similar to the way that the photon appears in quantum electrodynamics
that I thought that the theory might be renormalizable in the same sense
as quantum electrodynamics--that is, that the infinities in the theory
could be canceled by a redefinition of the masses and other quantities
in the theory. Furthermore, the theory would be highly constrained by
its underlying principles and would thus avoid a large part of 
arbitrariness of previous theories. 

I worked out a particular concrete realization of this theory, that is,
a particular set of equations that govern the way the particles interacted
and that would have the Fermi theory as a low energy approximation.
I found in doing this, although it had not been my idea at all to start
with, that it turned out to be a theory not only of the weak forces,
based on an analogy with electromagnetism; it turned out to be a unified
theory of the weak and electromagnetic forces that showed that they were
both just different aspects of what subsequently became called an 
electroweak force. The photon, the fundamental particle whose emission 
and absorption causes electromagnetic forces, was joint in a tight-knit
family group with the other photonlike particles predicted by the theory:
electrically charged W particles whose exchange produces the weak force
of beta radioactivity, and a neutral particle I called the ``Z",
about which more later. (W particles were an old story in speculations 
about the weak forces; the W stands for ``weak". I picked the letter Z
for their new sibling because the particle has zero electric charge
and also because Z is the last letter of the alphabet, and I hoped that
this would be the last member of the family.) Essentially the same
theory was worked out independently in 1968 by the Pakistani physicist
Abdus Salam, working in Trieste....
Both Salam and I had stated our
opinion that this theory would eliminate the problem of infinities
in the weak forces. But we were not clever enough to prove this.
In 1971 I received a preprint from a young graduate student at the 
University of Utrecht named Gerard 't Hooft, in which he claimed
to show that this theory actually  had solved the problem of the
infinities: the infinities in calculations of observable quantities
would in fact all cancel just just as in quantum electrodynamics...."

\section{Unified wave mechanics.}
There are two natural ways to furnish the time on nilpotent groups.
One of them is defined by solvable extensions 
while in the other case the  
time-axis is introduced by 
Cartesian product with the real line $\mathbb R$. 
These two extensions are called expanding- and static-models
respectively.

The metric on the nilpotent group is positive definite 
where the Laplacian
turns out to be natural physical Hamilton operator corresponding
to elementary particle systems. Concrete systems are represented by the
corresponding invariant subspaces of the Laplacian and one obtains the 
Hamilton operator of a given system by restricting the Laplacian onto these
subspaces. In order to establish the wave equations regarding 
these Hamiltonians, on both extensions indefinite
metrics must be introduced. That is, adequate wave mechanics
associated with these Hamiltonians can just relativistically be
introduced such that one assumes appropriate
Lorenz-indefinite metrics on both extensions.
As it turns out, these metrics really provide the familiar wave operators 
of wave mechanics.

\subsection{Solvable extensions.}
Any 2-step nilpotent Lie group, $N$, extends to a
solvable group, 
$SN$, which is defined on the half-space $\mathcal N\times \mathbb R_+$ by
the group multiplication 
\begin{equation}
(X,Z,t)(X^\prime ,Z^\prime ,t^\prime )
=(X+t^{\frac{1} {2}}X^\prime ,Z+tZ^\prime +\frac {1}{2}t^{\frac{1}{2}}
[X,X^\prime ],tt^\prime ).
\end{equation}
This formula provides the 
multiplication also
on the nilpotent group $N$, which appears as a subgroup 
on the level set $t=1$.
 
The Lie algebra of this
solvable group is $\mathcal S=\mathcal N\oplus \mathcal T$. 
The Lie 
bracket is completely determined by the formulas 
\begin{equation}
[\partial_t,X]=\frac{1}{2}X\quad ;\quad [\partial_t,Z]=Z\quad;\quad 
[\mathcal N,\mathcal N]_{/SN}
=[\mathcal N,\mathcal N]_{/N},
\end{equation}
where $X \in \mathcal X$ and $Z \in \mathcal Z$.

The indefinite metric tensor is defined 
by the left-invariant extension
of the indefinite inner product,
$\langle\, ,\,\rangle$, defined 
on the solvable Lie algebra $\mathcal S$ by  
$\langle\partial_t ,\partial_t\rangle =-1$,  
$\langle\partial_t ,\mathcal N\rangle =0$, and 
$\langle\mathcal N ,\mathcal N\rangle =
\langle\mathcal N ,\mathcal N\rangle_{\mathcal N}$. 
The last formula indicates
that the original innerproduct is kept on the subalgebra $\mathcal N$.
Lie algebra $\mathcal S$ is considered as the tangent space at the origin
$(0,0,1)\in SN$ of the solvable group and 
$\langle\, ,\,\rangle$ is extended
to a left-invariant metric $g$ onto $SN$ by the group multiplication
described above. 
Scaled inner product $\langle\, ,\,\rangle_{q}$ 
with scaling factor $q > 0$ is also defined by  
$\langle \partial_t, \partial_t\rangle_{q}=-1/q^{2}$, but, 
keeping both the inner product 
on $\mathcal N$ 
and the relation $\partial_t\perp \mathcal N$ on $\mathcal N$ intact. 
That is, the scaling regards just the direction regarding $\partial_t$. 
The left invariant
extension of these inner products are denoted by $g_{q}$.

For precise explanations we need some explicit formulas
on these groups as well.
The left-invariant extensions, $\mathbf Y_i, \mathbf V_\alpha ,\mathbf T$, 
of the unit vectors 
\begin{equation}
E_i=\partial_i\quad ,\quad e_\alpha =\partial_\alpha 
\quad ,\quad\epsilon =q\partial_t
\end{equation}
picked up at the
origin $(0,0,1)$ are the vector fields:
\begin{equation}
\mathbf Y_i=t^{\frac 1{2}}\mathbf X_i\quad ;\quad \mathbf V_\alpha 
=t\mathbf Z_\alpha 
\quad ;\quad \mathbf T=qt\partial_t ,
\end{equation}
where $\mathbf X_i$ and $\mathbf Z_\alpha$ are the 
invariant vector fields introduced on $N$ previously.

One can establish these 
formulas by the following standard computations. Consider the 
vectors $\partial_i ,\partial_{\alpha}$ and $\partial_t$ at the origin 
$(0,0,1)$ such that they are the
tangent vectors of the curves $c_A (s)=(0,0,1)+s\partial_A$,
where $A=i,\alpha ,t$. Then
transform these curves to an arbitrary point by the left 
multiplications. 
Then the tangent of the transformed
curve gives the desired left invariant vector at an arbitrary point.

According to these
formulas, not $t$ but $T$ defined by $\partial_T=\mathbf T$ is the correct
physical time parameterization on the $t$- parameter line, which, by
the below arguments, are geodesics on $SN$. The transformation law
$\partial_T=(dt/dT)\partial_t$ yields the relations $(dt/dT)=qt$, 
$\ln t=qT$, and $t=e^{qT}$. Thus a $t$-level set is the same as the 
$T=(\ln t)/q$-level set and subgroup $N$ corresponds both to $t=1$ and
$T=0$. The reversed time $-T$ is denoted by $\tau$.    
 
Let $c_x(s)$ and $c_z(s)$ be integral curves of finite length
$||c_x||$ resp. $||c_z||$ 
of the invariant vector fields $\mathbf X$ and $\mathbf Z$ on the
subgroup $N$. Then the flow generated by $\partial_\tau$ moves these
curves to $c^\tau_x(s)$ resp. $c^\tau_z(s)$ of length 
$||c_x^\tau ||=||c_x||e^{q\tau /2}$ resp. 
$||c_z^\tau ||=||c_z||e^{q\tau}$. 
That is, by considering them as functions 
of the time-variable $\tau$, the length is increasing 
such that the rate of change 
(derivative with respect to $\tau$) is
proportional to the length of the curves. In other words,
this mathematical space-time model represents an expanding micro universe
where the distance between particles is growing exactly in the same way
how this growing distance was measured by Hubble, in 1929, 
between galaxies \cite{h}. 

Edwin Hubble came to this conclusion after 
experimenting red-shift in the spectra of galaxies he observed for 
cataloguing their distances from the earth. It was quite a surprise
that, contrary to the expectation, the 
red- and blue-shifted galaxies occurred with
no equal likelihood but most of them appeared red-shifted. That is,
most of them are moving away from us. Even more surprising was to find
that even the size of a galaxy's red shift is not random, but it is 
directly proportional to the galaxy's distance from us. This means, that
the farther the galaxy is, the faster it is moving away. This is the 
familiar Hubble's law which was actually predicted by the Friedmann
cosmological model, in 1922.  

The tendency to expand must be rooted 
from the very same tendency inbuilt into the microscopic universe.
This argument, however, contradicts the experimental fact 
(explained in the introduction) according to which 
the particles are not expanding. This paradox can be resolved, however,
by recognizing that, due to the expansion, the change of the constant 
magnetic fields, which are present in the particles
also by the new model, induces electromagnetic waves which are
immediately radiated out into the space. This explanation clarifies not
only this paradox but it casts also a new light to the presence of the 
constant radiation experimentally known in the space. These 
arguments greatly enhances the importance of the exact mathematical
models introduced in this paper.

The covariant derivative can be computed by the well known formula
\begin{equation}
\langle\nabla_PQ,R\rangle =
\frac 1{2}\{\langle P,[R,Q]\rangle +\langle Q,[R,P]\rangle +
\langle [P,Q],R\rangle\},
\end{equation}
where $P,Q,R$ are invariant vector fields.
Then we get
\begin{eqnarray}
\label{solvnabla}
\nabla_{X+Z}(X^*+Z^*)=\nabla^N_{X+Z}(X^*+Z^*)-q(\frac 1{2}
\langle X,X^*\rangle+\langle Z,Z^*\rangle )\mathbf T ;\nonumber
\\
\nabla_X\mathbf T=\frac q{2}X\quad ;
\quad\nabla_Z\mathbf T=qZ \quad ;\quad
\nabla_TX=\nabla_TZ=\nabla_TT=0,
\nonumber
\end{eqnarray}
where $\nabla^N$ denotes covariant derivative 
and $X,X^*\in\mathcal X;Z,Z^*\in\mathcal Z;T\in\mathcal T.$

The Laplacian on these solvable groups can be established by the same
computation performed on $N$. Then we get  
\begin{eqnarray}
\Delta=
t^{2}\Delta_Z-q^2(t^2\partial^2_{tt}+t\partial_t)+\\
t\big(\Delta_X+
\frac 1 {4}\sum_{\alpha ;\beta =1}^l
\langle J_\alpha (X),J_\beta (X)\rangle\partial^2_{\alpha \beta} 
+\sum_{\alpha =1}^l\partial_\alpha D_\alpha\bullet \big) 
+q^2(\frac k {2}+l)t\partial_t=
\nonumber\\
=e^{2qT}\Delta_Z-\partial^2_{TT}+\nonumber
\\
e^{qT}\big(\Delta_X+
\frac 1 {4}\sum_{\alpha ;\beta =1}^l
\langle J_\alpha (X),J_\beta (X)\rangle\partial^2_{\alpha \beta} 
+\sum_{\alpha =1}^l\partial_\alpha D_\alpha\bullet \big) 
+q(\frac k {2}+l)\partial_T.
\nonumber
\end{eqnarray}
This is the Laplacian on the solvable extension of a general 2-step
nilpotent Lie group. On the extension of a H-type group it appears in the
following simpler form:
\begin{equation}
\Delta =\{e^{2qT}\Delta_Z-\partial^2_{TT}\}+
\big\{e^{qT}\big(\Delta_X+
\frac 1 {4}
\mathbf x^2\Delta_Z 
+\sum\partial_\alpha D_\alpha\bullet \big) 
+q(\frac k {2}+l)\partial_T\big\}
\end{equation}
This operator is expressed regarding the collapsing time direction.
Substitution $T=-\tau$ provides the operator in terms of the expanding
time direction. Note that the first operator, 
$e^{-2q\tau}\Delta_Z-\partial^2_{\tau\tau}$, looks like "expanding
meson operator", while the second one is similar to 
the Schr\"dinger operator of charged particles. 
This question is further investigated in the next section.

In order to understand the deeper connections to general relativity,
also the Riemannian curvature should explicitly be computed. This 
calculation can straightforwardly be implemented by substituting
formulas (\ref{solvnabla}) into the standard
formula of the Riemannian curvature. Then we get
\begin{eqnarray}
R_q(X^*\wedge X)=R(X^*\wedge X)+ \frac q{2}[X^*,X]\wedge\mathbf T+
\frac {q^2}{4}X^*\wedge X ;
\\
R_q(X\wedge Z)=R(X\wedge Z)+\frac q{4}J_Z(X)\wedge\bold T+\frac{q^2}{2}
X\wedge Z ;
\\ 
R_q(Z^*\wedge Z)=R(Z^*\wedge Z)+q^2Z^*\wedge Z;\\
R_q((X+Z),\mathbf T)(.)=q\nabla_{\frac 1{2}X+Z} (.); \quad
R_q((X+Z)\wedge\mathbf T)=
\\
=\frac{1}{2}q(\sum_\alpha 
J_\alpha(X)\wedge e_\alpha 
-J^*_Z)-q^2(\frac{1}{4}X+Z)\wedge\mathbf T ,
\end{eqnarray}
where $J^*_Z$ is the 2-vector dual to the 
2-form $\langle J_Z(X_1),X_2\rangle$ and
$R$ is the Riemann curvature on 
$N$. The vectors in these formulas are elements of the Lie algebra.

By introducing $H(X,X^*,Z,Z^*):=\langle J_Z(X),J_{Z^*}(X^*)\rangle$,
for the Ricci curvature
we have
\begin{eqnarray}
Ri_q(X)=Ri(X)-q^2(\frac k{4}+\frac l{2})X;\\
Ri_q(Z)=Ri(Z)-q^2(\frac k{2}+l)Z\quad ;\quad
Ri_q(T)=+q^2(\frac k{4}+l)T,
\end{eqnarray}
where the Ricci tensor $Ri$ on $N$ is described
by formulas

\begin{eqnarray}
Ri(X,X^*)=-\frac 1 {2} \sum_{\alpha =1}^lH(X,X^*,e_\alpha ,e_\alpha )=
-\frac 1{2}H_{\mathcal X}(X,X^*)=-\frac l{2}\langle X,X^*\rangle ;
\nonumber
\\
Ri(Z,Z^*)=\frac 1 {4} \sum_{i=1}^k
H(E_i,E_i,Z,Z^*)=\frac 1{4}H_{\mathcal Z}(Z,Z^*)=\frac k {4}
\langle Z,Z^*\rangle ,
\nonumber
\end{eqnarray}
and by
$Ri(X,Z)=0$. By assuming $q=1$, we have
\begin{eqnarray}
Ri_1(X)=-(\frac k{4}+l)X\, ;\,
Ri_1(Z)=-(\frac k{4}+l)Z;\\
Ri_1(T)=(\frac k{4}+l)T, \quad
\mathcal R=-(\frac k{4}+l)(k+l+1),
\\
Ri_1(X+Z,X^*+Z^*)-\frac 1{2}\mathcal R\langle X+Z,X^*+Z^*\rangle =\\ 
=(\frac k{4}+l)(k+l-\frac 1{2})\langle X+Z,X^*+Z^*\rangle ,
\\ 
Ri_1(X+Z,T)-\frac 1{2}\mathcal R\langle X+Z,T\rangle =0,
\\
Ri_1(T,T)-\frac 1{2}\mathcal R\langle T,T\rangle =(\frac k{4}+l)
(k+l+\frac 3{2})\langle T,T\rangle .
\end{eqnarray}

These tensors are defined in terms of the elements of the Lie algebra.
In order to compute them on local coordinate systems, first 
the metric tensor
$
g_{ij}=g\big(\partial_i,
\partial_j)\, ,\, g_{i\alpha}=g ( \partial_i,
\partial_\alpha )\, ,\, g_{\alpha\beta}=g(\partial_\alpha
,\partial_\beta )
$
and its inverse, $g^{ij},g^{i\alpha},g^{\alpha\beta}$, on $N$,
need to be calculated. By the explicit 
form of the invariant vector fields we have:
\begin{eqnarray}
g_{ij}=\delta_{ij} + \frac 1 {4}\sum_{\alpha =1}^l
\langle J_\alpha (X),\partial_i\rangle\langle J_\alpha
(X),\partial_j\rangle ,\,\,  g_{\alpha\beta}=\delta_{\alpha \beta} ,
\\
g_{i\alpha}=-\frac 1 {2} \langle J_\alpha (X),\partial_i\rangle ,
\,\,
g^{ij}=\delta_{ij} , \,\, g^{i\alpha}=\frac 1 {2} \langle\partial_i,
J_\alpha (X)\rangle ,
\\
g^{\alpha \beta}=
\delta_{\alpha \beta} + \frac 1 {4} \langle J_\alpha (X), 
J_\beta (X)\rangle =(1+\frac 1 {4}\mathbf x^2) \delta_{\alpha \beta}.
\end{eqnarray}
These components determine the metric tensor on $SN$ 
by the formulas:
$
tg_{ij}=\, ,\,t^{3/2}g_{i\alpha}\, , \, t^2g_{\alpha\beta}
$.

\subsection{Static Schr\"odinger and neutrino equations.}

The static model is defined by the Cartesian product, $N\times\mathbb R$,
of metrics, where $\mathbb R$,
parameterized by $t$, is endowed by the indefinite inner
product $\langle \partial_t, \partial_t\rangle_{q}=-1/q^{2}$. 
The several objects
such as Riemann curvature can be computed by laws corresponding to the
Cartesian products, thus they are non-trivial only regarding the
nilpotent direction. These explicit formulas can easily be
recover from the previous ones. 

In what follows
we utilize Pauli's computation (\ref{pnonrel_1})-(\ref{nonrel_waveeq})
regarding  the non-relativistic approximation of the relativistic
wave equation. By choosing $q=1/c$,
the Laplacian appears in the following form:
\begin{eqnarray}
\label{lapl}
\Delta =
(\Delta_Z-\frac 1{c^2}\partial^2_{tt})+
\big(\Delta_X+\frac 1 {4}\mathbf x^2\Delta_Z 
+\sum\partial_\alpha D_\alpha\bullet \big)
=\\ 
=
(\Delta_Z+\frac{2m\mathbf i}{\hbar}\partial_t-\frac 1{c^2}\partial^2_{tt})
+\big(\Delta_X+\frac 1 {4}\mathbf x^2\Delta_Z 
+\sum\partial_\alpha D_\alpha\bullet -\frac{2m\mathbf i}{\hbar}
\partial_t \big).
\nonumber
\end{eqnarray}

On the Z-space, operator $\Delta_Z-\frac 1{c^2}\partial^2_{tt}$ is nothing
but the wave operator (\ref{y1}). According to Yukawa's exposition, 
the eigenfunctions, $U$, of this operator describe 
the eigenstates of nuclear forces. 
Due to (\ref{waveeq}), the general solutions of this wave equation 
are de Broglie's wave packets (\ref{wavepack}). On the mathematical model, 
however, these wave packets are 
represented by twisted functions of the form  
\begin{eqnarray}
\psi_{\mathbf Bp_iq_i} (X,Z,t)=
\\
=\int_{\mathbb R^l} e^{\mathbf i(\langle Z,K\rangle -\omega t)}
\phi (\mathbf x,\mathbf k)\varphi (K_u)
\prod_{i=1}^{k/2}z^{p_i}_{K_ui}(X)\overline z^{q_i}_{K_ui}(X)dK=
\\
=\int_{\mathbb R^l} e^{\mathbf i(\langle Z,K\rangle -\omega t)}
\phi (\mathbf x,\mathbf k)F_{\mathbf Bp_iq_i} (X,K_u)dK,
\\
{\rm where}\quad\quad
\sqrt{\mathbf k^2+\frac{m^2c^2}{\hbar^2}}=\frac{\omega}{c}.
\end{eqnarray}

Wave packets $\psi_{Qpq} (X,Z,t)$ are defined by means of the functions
\begin{equation}
F_{Qpq} (X,K_u)=\varphi (K_u) (\Theta_{Q}^p\overline\Theta^q_{Q})(X,K_u).
\end{equation}
They are defined also for Z-sphere bundles 
$S_R(\mathbf x)$, which versions are
indicated by denotations 
$\psi_{\mathbf Bp_iq_iS_R}$ and $\psi_{QpqS_R}$. 
From this respect, $\psi_{\mathbf Bp_iq_i\mathbb R^l}$ and 
$\psi_{Qpq\mathbb R^l}$ correspond to the above introduced
wave packets. 

For a fixed Z-vector, $Z_\gamma$, the regarding denotations
are 
$\psi_{\mathbf Bp_iq_iZ_\gamma}$ and $\psi_{QpqZ_\gamma}$, 
where $\mathbf B$
is an orthonormal basis regarding $J_{Z{\gamma u}}$. In this case the 
integral is taken with respect to the Dirac delta measure concentrated
at $Z_\gamma$, thus these formulas can be written up without indicating
this integral or constant $\varphi (K_u)$. By projections $\Pi_X^{(n)}$,
one defines
\begin{equation}
\Psi_{....}(X,Z,t)=   
\int_{\circ} e^{\mathbf i(\langle Z,K\rangle -\omega t)}
\phi (\mathbf x,\mathbf k)\Pi_X^{(n)}F_{...} (X,K_u)dK,
\end{equation}
where dots $....$ can be substituted by any of the symbols 
$\mathbf Bp_iq_i\mathbb R^l,\,Qpq\mathbb R^l$, .. e. t. c., and circle,
$\circ$, could symbolize any of the integral domains 
$\mathbb R^l,\, S_R,\, Z_\gamma$.

If also projections $\Pi_K^{(r,s)}$ are applied to $F_{...}$, 
the corresponding functions are $\psi^{(r,s)}_{....}$ resp. 
$\Psi^{(r,s)}_{....}$. This operation makes sense only for
integral domains $\mathbb R^l$ or $S_R(\mathbf x)$ 
but it is not defined for the
singular Dirac delta domain $Z_\gamma$. 
Anti de Broglie wave packets are defined by replacing $-\omega$ with 
$+\omega$ in the above formulas. The corresponding functions are denoted
by $\psi^{anti}_{....}$ and $\Psi^{anti}_{....}$. The associated particles
are called antiparticles. These objects can be introduced also by keeping
$-\omega$ and replacing $\mathbf i$ by $-\mathbf i$.

As it is indicated, the right side of (\ref{lapl}) is computed by adding 
$\frac{2m\mathbf i}{\hbar}\partial_t-
\frac{2m\mathbf i}{\hbar}\partial_t=0$ 
to the left side. 
Then the wave operator
associated with nuclear forces becomes 
\begin{equation}
\label{N}
\mathit N=\Delta_Z+\frac{2m\mathbf i}{\hbar}
\partial_t-\frac 1{c^2}\partial^2_{tt},
\end{equation}
which is the non-relativistic wave operator established in 
(\ref{nonrel_waveeq}).
As it is explained in (\ref{nonrel_wavepack})-(\ref{nonrel_waveeq}),
the $\mathit N$-harmonic waves, defined by 
$\mathit N (\tilde \psi )(Z,t)=0$,
relate to the relativistic wave by the formula
\begin{equation}
\psi (Z,t)=e^{-\frac{\mathbf imc^2}{\hbar}t} \tilde \psi (Z,t).
\end{equation}

Also remember that frequency $\tilde\omega$ is derived from 
 \begin{equation}
\omega =\frac{E}{\hbar}
=\frac{mc^2}{\hbar}\sqrt{1+ \frac{\hbar^2\mathbf k^2}{m^2c^2}}=
\frac{mc^2}{\hbar}+\tilde\omega
=\frac{mc^2}{\hbar}+\frac{\hbar}{2m}\mathbf k^2+\dots ,
\end{equation}
by the Taylor expansion of function $\sqrt{1+x}$. Thus the third term
depends on $\hbar^3$ and, by stepping further, 
this exponent is increased by $2$, by each step. For low speed
particles, value $\tilde\omega =\frac{\hbar}{2m}\mathbf k^2$ is a good 
approximation of the frequency, thus also $E=\hbar\tilde\omega$ is a good
approximation for the
energy of the particle associated with this non-relativistic wave.
Note that $E=E_{kin}$ is nothing but the kinetic energy owned by the
material particle. By this reason, the particle associated with the wave
operator $\mathit N$ 
can be consider as one of the residues of a decaying material particle
which has neither mass nor charge and the only source of its energy is
the kinetic energy of the decaying material particle. Such particles are
the neutrinos, thus $\mathit N$ is called neutrino operator accompanying
the electron-positron system.

The energy $mc^2$ of the material particle is completely attributed to the
other particle associated with the second operator
\begin{equation}
\label{S}
\mathit S=\Delta_X+\frac 1 {4}\mathbf x^2\Delta_Z 
+\sum\partial_\alpha D_\alpha\bullet -\frac{2m\mathbf i}{\hbar}
\partial_t .
\end{equation}
incorporated into the Laplacian (\ref{lapl}). In order to understand
this particle represented by this operator, we introduce first the 
de Broglie wave packets $\tilde \Psi_{....}(X,Z,t)$ and 
$\tilde \Psi^{anti}_{....}(X,Z,t)$ in the same way as before, but now,
the $\omega$ is replaced by $\tilde\omega$ which can take values such as
$\frac{\hbar}{2m}\mathbf k^2,\frac{\hbar}{2m}(4r+4p+k)\mu , 
\frac{\hbar}{2m}((4r+4p+k)\mu +4\mu^2)$.
By (\ref{land}), the Schr\"odinger equation for an electron is: 

\begin{eqnarray}
\label{land_schr}
-\big(\Delta_{(x,y)} -
{ eB\over \hbar c\mathbf i}
D_z\bullet
+{e^2B^2\over 4\hbar^2 c^2}(x^2+y^2) \big)\psi =
\frac{2m\mathbf i}{\hbar}\frac{\partial \psi}{\partial t}.
\end{eqnarray}

On the Z-crystal model, operator $\lhd_\mu$ is defined by the
action of the Laplacian $\Delta$ of the nilpotent group on functions
of the form
$\psi (X)e^{2\pi\mathbf i\langle Z_\gamma ,Z\rangle}$.
In terms of  $\lambda =\pi \mathbf z_\gamma $, this action can
be described as acting only on $\psi$ by the operator
$
\lhd_{\mu}=
\Delta_X +2 \mathbf i D_{\mu }\bullet -\mu^2\mathbf x^2-4\mu^2.
$ 
If the last constant term is omitted, the operator left is denoted by 
$\sqcup_\mu$.
Then, in the 2D case after substitution $\mu ={eB/2\hbar c}$, 
the negative of this reduced operator becomes nothing but the 
Hamilton operator standing on the left side of the
above Schr\"odinger equation.
If $\tilde K=2\pi Z_\gamma\, ,\, \mu =\tilde{\mathbf k}/2\, ,\,
m=\kappa m_e$, and
$f_\mu (\mathbf x^2)$ is a function such that 
$f_\mu (\tilde t)$ is an eigenfunction
of the radial Landau-Zeeman operator $\Diamond_{\tilde t}+4\mu^2$, 
defined in (\ref{Lf_lambda}), with eigenvalue 
$-\tilde\omega =-(4r+4p+k)\mu$, then for
$ \mathit S\big(\tilde\Psi^{anti}_{...\tilde K}(X,Z,t)\big)$ we have:
\begin{eqnarray}
\mathit S\big(   
e^{\mathbf i(\langle Z,\tilde K\rangle +\frac{\hbar}{2m}\tilde\omega t)}
f_\mu (\mathbf x^2)\Pi_X^{(n)}F_{...} (X,\tilde K_u)\big)=
\\
=e^{\mathbf i\langle Z,\tilde K\rangle}(\sqcup_\mu -
\frac{2m\mathbf i}{\hbar}\frac{\partial}{\partial t})\big(
e^{\frac{\hbar\mathbf i}{2m}\tilde\omega t}
f_\mu (\mathbf x^2)\Pi_X^{(n)}F_{...} (X,\tilde K_u)\big)=0.
\nonumber
\end{eqnarray}
Thus on Z-crystals, operator $\mathit S$ is nothing but  
Schr\"odinger's classical wave operator of an electron positron system. 

Note that no 
non-relativistic limiting was used to obtain this operator. It is 
naturally incorporated into the complete 
Laplacian $\Delta$. Although it is the same as the non-relativistic wave
operator obtained earlier by non-relativistic limiting, even the neutrino
operator, $\mathit N$, is not the result of a non-relativistic limiting.
The Laplacian $\Delta$ is the sum of these two natural operators, meaning
that it actually represents a system consisting of electrons positrons
and electron-positron-neutrinos. 
The above arguments also suggest that this system can be regarded as the
result of a sort of nucleus-decay. A rigorous theory describing this 
process is yet to be established. It is clear, however, that
the basic mathematical tool underlying this physical theory must be
the decomposition of the Laplacian into operators corresponding to the
constituents of a given particle system. Dealing with Laplacian means
that one does not violates the principle of energy conservation. Moreover,
this tool provides also the exact operators associated with the particles,
which is the most attractive new feature of these exact mathematical
models. 

Actually, the elementary particles discovered in classical quantum 
theory were introduced by the very same idea.
For instance, the neutrino was first postulated in 1930 by 
Wolfgang Pauli to preserve conservation of energy, conservation of 
momentum, and conservation of angular momentum in beta decay – 
the decay of a neutron into a proton, an electron and an antineutrino. 
Pauli theorized that an undetected particle was carrying away 
the observed difference between the energy, momentum, 
and angular momentum of the initial and final particles. The only 
difference between the two ways introducing the neutrinos is that
Pauli did not have a Riemann manifold in hand in the Laplacian of which
he would have been able to separate the neutrino from the other particles
resulted by the decay.

The only term in the Laplacian containing 
second order derivatives regarding
the time variable $t$ is the neutrino operator. 
This term is of first order in the Schr\"odinger operator. Because 
of this, waves $\tilde\Psi$ are not solutions of the neutrino
operator and waves $\tilde\psi$ obtained above by the Taylor expansion are
not solutions of the Schr\"odinger equation. In order to cope with this
difficulty, non-relativistic approximation
can be implemented such that one attributes the
kinetic energy represented by $\Delta_Z$ in the neutrino operator 
to the Hamilton
operator associated with $\mathit S$ by considering the total 
Schr\"odinger operator
\begin{equation}
\label{TS}
{\mathbb S}=\Delta_X+(1+\frac 1 {4}\mathbf x^2)\Delta_Z 
+\sum\partial_\alpha D_\alpha\bullet -\frac{2m\mathbf i}{\hbar}
\partial_t ,
\end{equation}
which is the sum of the Schr\"odinger operator and $\Delta_Z$.

In this step, the two operator is pulled together to form an operator
which is of first order regarding the time variable. 
This scheme is completely analogous to those applied by Schr\"odinger 
when, instead
of the Klein-Gordon equation, he introduced his equation. A major
difference is, however, that the above operator accounts also with the
energy of neutrinos accompanying the electron-positron system, moreover,
the non-relativistic approximation is applied to the neutrino operator
and not to the electron-positron operator. The neglected Taylor-terms
in this approximation depend on $\hbar^s$, where $s\geq 3$.
The wave functions regarding this pulled-together operator 
are defined by the eigenvalues 
$-\tilde\omega =-((4r+4p+k)\mu +4\mu^2$. Then,
in terms of 
$\mathbf F^{(n)}_{...} (X,\tilde K_u)=\Pi_X^{(n)}F_{...} (X,\tilde K_u)$, 
we have:
\begin{eqnarray}
\mathbb S\big(\tilde\Psi^{anti}_{...\tilde K}(X,Z,t)\big)=\mathbb S\big(   
e^{\mathbf i(\langle Z,\tilde K\rangle +\frac{\hbar}{2m}\tilde\omega t)}
f_\mu (\mathbf x^2)\mathbf F^{(n)}_{...} (X,\tilde K_u)\big)=
\\
=e^{\mathbf i\langle Z,\tilde K\rangle}(\lhd_\lambda -
\frac{2m\mathbf i}{\hbar}\frac{\partial}{\partial t})\big(
e^{\frac{\hbar\mathbf i}{2m}\tilde\omega t}
f_\mu (\mathbf x^2)\mathbf F^{(n)}_{...} (X,\tilde K_u)\big)=0.
\nonumber
\end{eqnarray}

\begin{eqnarray}
\mathbb S\big(\tilde\Psi^{anti}_{...S_R}\big)=\mathbb S\big(\oint_{S_R}   
e^{\mathbf i(\langle Z,K\rangle +\frac{\hbar}{2m}\tilde\omega t)}
f_{\frac 1{2}\mathbf k} (\mathbf x^2)\mathbf F^{(n)}_{...} (X,K)dK_n\big)=
\\
=\oint_{S_R}e^{\mathbf i\langle Z,K\rangle}(\lhd_{\frac 1{2}\mathbf k} -
\frac{2m\mathbf i}{\hbar}\frac{\partial}{\partial t})\big(
e^{\frac{\hbar\mathbf i}{2m}\tilde\omega t}
f_{\frac 1{2}\mathbf k} 
(\mathbf x^2)\mathbf F^{(n)}_{...} (X,K)dK_n\big)=0.
\nonumber
\end{eqnarray}

\begin{eqnarray}
\mathbb S\big(\tilde\Psi^{anti}_{...\mathbb R^l}\big)=
\mathbb S\big(\int_{\mathbb R^l}   
e^{\mathbf i(\langle Z,K\rangle +\frac{\hbar}{2m}\tilde\omega t)}
f_{\frac 1{2}\mathbf k} (\mathbf x^2)\mathbf F^{(n)}_{...} (X,K_u)dK\big)=
\\
\int e^{\mathbf i\langle Z,K\rangle}
(\lhd_{\frac 1{2}\mathbf k} -
\frac{2m\mathbf i}{\hbar}\frac{\partial}{\partial t})\big(
e^{\frac{\hbar\mathbf i}{2m}\tilde\omega t}
f_{\frac 1{2}\mathbf k} 
(\mathbf x^2)\mathbf F^{(n)}_{...}
(X,K_u)\big)\mathbf k^{l-1}dK_ud\mathbf k=0.
\nonumber
\end{eqnarray}

The same formulas hold for operator $\mathit S$ which appears
as the classical Schr\"odinger operator 
$\sqcup_{\frac 1{2}\mathbf k} -\frac{2m\mathbf i}{\hbar}
\frac{\partial}{\partial t}$ behind the integral sign.
Similar arguments work out also for operator $\OE$, in which case
the Schr\"odinger operators act on wave functions 
$\tilde\Psi^{anti(r,s)}_{...S_R}(X,Z,t)$. This is still a scalar operator
which can be reduced to a radial operator acting on a single radial 
function.  
The radial operator to which the complete operator $\Delta$ can be reduced
act on d-tuples of radial function, therefore integral formulas regarding
these cases must be built up in terms of function 
$f_\beta\Pi^{(\beta )}_{K_u}\mathbf F^{(n)}$ where d-tuple
$(f_1,\dots ,f_d)$ is an eigen d-tuple of the reduced radial operator.  
The particles defined by these operators are denoted by 
$W_{\OE} =W_{\OE^{(1)}}$ resp. 
$W_\Delta =W_\Delta^{(\infty)}$. They are called
clean-weak and clean-high W-particles respectively, while the other 
particles $W^{(u)}_{\OE} =W_{\OE^{(u)}}$ resp. 
$W^{(u)}_\Delta =W_{\Delta^{(u)}}$ are the so called dirty W-particles. 
The neutrino operator is the same in all of these cases, thus
the associated particles are denoted  $Z_{\OE}$. These denotations are
are suggested by the theory of weak nuclear forces. They indicate that
W-type particles can analogously be defined also regarding strong forces.
However, the beta decay can be explained just by the clean weak nuclear
forces.

\subsection{Expanding Schr\"odinger and neutrino equations.}

For the sake of simplicity the following formulas are established regarding
the collapsing (shrinking) time-direction $T$ under the condition
$q=1$. Formulas regarding the expanding 
time-direction $\tau$ can be obtained 
by the substitution $T=-\tau$. Instead of $t$, the expanding wave functions
are introduced in terms of $e^T$.
That is, the shrinking twisted wave packets are of the form
\begin{eqnarray}
\Psi_{....}(X,Z,T)=   
\int_{\circ} e^{\mathbf i(\langle Z,K\rangle -\omega e^T)}
\phi (\mathbf x,\mathbf k)\Pi_X^{(n)}F_{...} (X,K_u)dK
\\
=\int_{\circ} e^{\mathbf i(\langle Z,K\rangle -\omega e^T)}
\phi (\mathbf x,\mathbf k)\mathbf F^{(n)}_{...} (X,K_u)dK,
\nonumber
\end{eqnarray}
where
$
\sqrt{\mathbf k^2+\frac{m^2c^2}{\hbar^2}}=\frac{\omega}{c},
$
and, as above,  dots $....$ can be substituted by any of the symbols 
$\mathbf Bp_iq_i\mathbb R^l,\,Qpq\mathbb R^l$, .. e. t. c., and circle,
$\circ$, could symbolize any of the integral domains 
$\mathbb R^l,\, S_R,\, Z_\gamma$. 
De Broglie's wave packets $\tilde \Psi_{....}(X,Z,T)$ and 
$\tilde \Psi^{anti}_{....}(X,Z,T)$ are introduced by the same modification,
that is, the $\omega$ is replaced by $\tilde\omega$ in the latter formula, 
which can take values such as
$\frac{\hbar}{2m}\mathbf k^2,\frac{\hbar}{2m}(4r+4p+k)\mu, 
\frac{\hbar}{2m}((4r+4p+k)\mu +4\mu^2)$.

The meson operator appears now in the form:
\begin{equation}
\label{M}
\mathit M=e^{2T}\Delta_Z+\partial_T-\partial^2_{TT},
\end{equation}
The same computation implemented on the static model shows that the
shrinking matter waves $\Psi_{....}(X,Z,T)$, defined in terms of $\omega$,
are really harmonic, meaning $\mathit M \Psi_{....}=0$, 
regarding this operator. Moreover, wave packet 
$\hat \Psi (X,Z,T)$ defined by
\begin{equation}
\Psi (X,Z,T)=e^{-\frac{\mathbf imc^2}{\hbar}e^T} \hat \Psi (X,Z,T)
\end{equation}
is harmonic regarding the shrinking neutrino operator
\begin{equation}
\label{NtrnoS}
\mathit N=
e^{2T}\Delta_Z+(1+\frac{2m\mathbf i}{\hbar}e^T)\partial_T-\partial^2_{TT}.
\end{equation}

According to this computation, the corresponding decomposition
of the Laplacian into non-polarized neutrino and Schr\"odinger operator
of a particle system is as follows
\begin{eqnarray}
\Delta =\{e^{2T}\Delta_Z-\partial^2_{TT}\}+
\\
+\big\{e^{T}\big(\Delta_X+\frac 1 {4}\mathbf x^2\Delta_Z 
+\sum\partial_\alpha D_\alpha\bullet \big) 
+(\frac k {2}+l)\partial_T\big\}=
\nonumber
\\
=
\big\{ e^{2T}\Delta_Z+(1+\frac{2m\mathbf i}{\hbar}e^T)\partial_T-
\partial^2_{TT}\big\} +
\\
+\big\{e^{T}\big(\Delta_X+\frac 1 {4}\mathbf x^2\Delta_Z 
+\sum\partial_\alpha D_\alpha\bullet \big) 
+(\frac k {2}+l-1-\frac{2m\mathbf i}{\hbar}e^T)\partial_T\big\}
\nonumber
\\
=
\big( e^{2T}\Delta_Z+(1+\frac{2m\mathbf i}{\hbar}e^T)\partial_T-
\partial^2_{TT}\big) +
\\
+e^{T}\big(\Delta_X+\frac 1 {4}\mathbf x^2\Delta_Z 
+\sum\partial_\alpha D_\alpha\bullet 
-\frac{2m\mathbf i}{\hbar}\partial_T \big) 
+\big(\frac k {2}+l-1\big)\partial_T.
\nonumber
\end{eqnarray}

In terms of $\tau =-T$, these operators define the expanding
non-polarized neutrino, Schr\"odinger, and tractor operators respectively.
The force associated with the third operator supplies the energy what is
needed to maintain the expansion. Let it also be pointed out that
according to these models the particles are not just moving away from
each other but this movement is also accelerating. This acceleration 
can be computed by taking the second derivatives of the distance
function introduced at explaining the expansion. 
This acceleration can be explained just by this new force represented 
by the third operator.   

It is important to keep in mind that these operators are non-polarized 
The polarized operators appear behind the integral sign after these 
non-polarized operators are acting on the de Broglie waves expressed 
by means of twisted Z-Fourier transforms.

\section{Spectral isotropy.} 

Spectral isotropy means that, on an arbitrary ball$\times$ball-
or sphere$\times$ball-type manifold with a fixed boundary condition, 
for any two unit X-vectors 
$Q$ and $\tilde Q$, the Laplacian is isospectral on the invariant 
function spaces
$
\mathbf \Xi_{Q}=\sum_n\mathbf \Xi_{Q}^{(n)}
$
and
$
\mathbf \Xi_{\tilde Q}=\sum_n\mathbf \Xi_{\tilde Q}^{(n)}
$
satisfying the given boundary condition. Recall that total space
$\mathbf \Xi_{Q}$ is everywhere dense in the straight space spanned
by functions of the form $f(|X|,Z)\langle Q,X\rangle$, furthermore,
the boundary conditions can totally be controlled by $(X,Z)$-radial
functions, therefore, this total function space is the same than what is
defined in terms of the straight functions.

Next we prove that any of the Heisenberg type groups is spectrally
isotropic.  On general 2-step nilpotent Lie
groups, where the endomorphisms $J_Z$ can have distinct eigenvalues,
this statement can be established just in a much weaker form not discussed
in this paper. Contrary to these general cases, the H-type groups
have the distinguishing characteristics that they represent systems 
consisting identical
particles and their anti-particles. Also note that on the expanding model 
this spectral
isotropy explains why the radiation induced by the change of
the constant magnetic field attached to the the spin operator is the
same whichever direction it is measured from. This radiation isotropy,
which has been measured with great accuracy, actually indicates that
the Heisenberg type groups are enough to describe the elementary particles 
and there is no need to involve more general 2-step nilpotent Lie groups 
to this new theory. 

This spectral isotropy is established by the intertwining operator
$
\omega_{Q\tilde Qpq\bullet}:
\mathbf \Xi_{Qpq.}\to
\mathbf \Xi_{\tilde Qpq\bullet}
$, 
defined by
\begin{eqnarray}
\label{intertw1}
\mathcal{HF}_{Qpq\bullet}(\phi )=
\int_{\bullet} e^{\mathbf i\langle Z,K\rangle}
\phi (\mathbf x,K)\Pi_X^{(n)}(\Theta_{Q}^p\overline\Theta^q_{Q})(X,K_u)
dK_\bullet\to 
\\
 \to \mathcal{HF}_{\tilde Qpq}(\phi )=
\int_{\bullet} e^{\mathbf i\langle Z,K\rangle}
\phi (\mathbf x,K)\Pi_X^{(n)}(\Theta_{\tilde Q}^p
\overline\Theta_{\tilde Q}^q)(X,K_u)dK_{\bullet} ,
\nonumber 
\end{eqnarray}
where heavy dot $\bullet$ may represent $R_Z(\mathbf x)$ or $\mathbb R^l$.
They indicate the function spaces this operator is defined for.
The very same operators are defined by corresponding Hankel
functions obtained by the Hankel decomposition of the above functions
to each other. This statement immediately follows from the fact that
the sums of the Hankel components restore the original functions in
the above formulas. Also note that the $(X,Z)$-radial Hankel functions
are the same regarding the two corresponded functions and they differ
from each other just by the Hankel polynomials obtained by the 
projections. Thus the operator defined by Hankel decomposition
must really be the same as the above operator. It follows that operator
(\ref{intertw1}) preserves the Hankel decompositions. 

An other remarkable property of this transform is
that it can be induced by
point transformations
of the form $(O_{\tilde QQ},id_Z)$,
where $id_Z$ is the identity map on the Z-space and $O_{\tilde QQ}$
is orthogonal transformation on the X-space, transforming
subspace $S_{\tilde Q}$, spanned by $\tilde Q$ and all  
$J_{Z_u}(\tilde Q)$, onto the similarly defined 
$S_{Q}$. This part of the 
map is uniquely determined, whereas, between the complement
X-spaces it can be arbitrary orthogonal transformation. One should keep in
mind that such a point transformation pulls back a function just from 
$\Xi_{Q\bullet}$ into the function space $\Xi_{\tilde Q\bullet}$ and it is
not defined for the whole $L^2$ function spaces. 

Next we prove that this operator intertwines the restrictions of
the Laplacian to these invariant subspaces, term by term. Since X-radial
functions are mapped to the same X-radial functions, furthermore, 
also X-spherical harmonics of the same degree are intertwined with
each other, the statement holds for $\Delta_X$. This part of the statement
can be settled also by the formula
$
\Delta_X=\sum\partial_{z_{K_ui}}\partial_{\overline z_{K_ui}}
$
written up in a coordinate system established by an orthonormal 
complex basis where 
$Q$ is the first element of this basis
$\{Q_{K_ui}\}$, for all $K_u$. That is, 
$\Theta_Q=z_1$ holds and the statement really follows by the above formula.
A third proof can be derived from the fact that this operator is induced
by the above described point transformation.

Due to the relations
\begin{eqnarray}
\label{MF}
\mathbf M\mathcal F_{Qpq}(\phi )=
\mathcal F_{Qpq}((q-p)\mathbf k\phi )\, ,\,
\Delta_Z\mathcal F_{Qpq}(\phi )=
\mathcal F_{Qpq}(-\mathbf k^2\phi ),\\
\label{HMF}
\mathbf M\mathcal{HF}_{Qpq}(\phi )
=\mathcal{HF}_{Qpq}((q-p)\mathbf k\phi )\, ,\,
\Delta_Z\mathcal {HF}_{Qpq}(\phi )=
\mathcal {HF}_{Qpq}(-\mathbf k^2\phi ),
\nonumber
\end{eqnarray}
the other parts of the Laplacians are also
obviously intertwined. In these formulas, the second line follows 
from the first
one by the commutativity of operator $D_K\bullet$ with the
projection $\Pi_X$. 

The intertwining property regarding the Dirichlet or Z-Neumann conditions 
on ball$\times$ball- resp. sphere$\times$ball-type domains can also
be easily established either by the above point transformations,
or with the help of the Hankel transform implying that functions of 
the form $f(|X|,|Z|)$ appearing in the transform are intertwined 
with themself. Since the
boundary conditions are expressed in terms of these double radial 
functions, this argument provides a second proof for the statement. 
A third proof can be obtained by the explicit formulas established for 
twisted functions satisfying these boundary conditions.

The most interesting new feature of this spectral isotropy is that
it holds even in cases when the space is not spatially isotropic. 
Let it be recalled that spatial-isotropy is the first assumption on 
the Friedmann model and the overwhelming evidence supporting this 
assumption was exactly the isotropic radiation measured by Penzias 
and Wilson, in 1965. 
The mathematical models demonstrate, however, that the spectral 
isotropy manifests itself even in much more general situations when the 
space is rather not spatially isotropic. In order to explain
this situation more clearly, we describe the isometries of H-type groups
in more details.

Generically speaking, these groups are not isotropic regarding the 
X-space. They satisfy this property just in very rare occasions. 
Starting with Heisenberg-type groups $H^{(a,b)}_3$, there is a subgroup,   
$\mathbf {Sp}(a)\times \mathbf{Sp}(b)$, 
of isometries which acts as the identity
on the Z-space and which acts transitively just on the X-sphere 
of $H_3^{(a+b,0)}$. In this case the intertwining property for operators 
$\omega_{Q\tilde Q,\bullet}$ also follows from the existence of 
isometries transforming $Q$ to $\tilde Q$. But the isometries
are not transitive on the X-spheres of the other spaces satisfying 
$ab\not =0$. 
This statement follows from the fact that the complete group 
of isometries is
$(\mathbf {Sp}(a)\times \mathbf{Sp}(b))\cdot SO(3)$,
where the action of $SO(3)$, described in terms of unit quaternions
$q$ by 
\[
\alpha_q(X_1,\dots ,X_{a+b},Z)=
(qX_1\overline q,\dots ,qX_{a+b}\overline q,qZ\overline q),
\] 
is transitive on the Z-sphere. 
Thus the above tool is not available to prove the spectral
isotropy in these cases. Yet, by the above arguments, the
$\omega_{Q\tilde Q\bullet}$ is an intertwining operator on its own right,
without the help of the isometries.

Note that the members of a family defined
by a fixed number $a+b$ have the same X-space but non-isomorphic 
isometry groups having
different dimensions in general. More precisely, two members,
$H^{(a,b)}_3$ and $H^{(a^\prime ,b^\prime)}_3$,
are isometrically isomorphic if and only if $(a,b)=(a^\prime ,b^\prime)$ 
holds up to an order. Furthermore,
the sphere$\times$sphere-type manifolds are homogeneous just on 
$H_3^{(a+b,0)}\simeq H_3^{(0,a+b)}$, 
while they are locally inhomogeneous, even on the X-spheres, 
on the other members of the family.
Let it be emphasized again that this homogeneity concerns not just 
the homogeneity of the 
X-spheres but the whole sphere$\times$sphere-type manifold.

The isometries are well known also for all H-type groups $H^{(a,b)}_l$.
The X-space-isotropy is obviously true also on the Heisenberg groups 
$H^{(a,b)}_1$ which can be defined as H-type groups satisfying $l=1$.
Besides this and the above quaternionic examples, it yet holds just on
$H_7^{(1,0)}\simeq H_7^{(0,1)}$. Thus the X-isotropy regarding isometries
is a rare property, indeed. This is why the spectral isotropy, 
yielded by any of the H-type groups, 
is a very surprising phenomenon indeed. It puts a completely
new light to the radiation isotropy evidencing the spatial-isotropy
assumed in Friedmann's model. According to the above theorem,
the radiation isotropy seems to
be evidencing all the new relativistic models of elementary
particle systems which are built up in this paper by nilpotent 
Lie groups and their solvable extensions. These models are far beyond those
satisfying the spatial-isotropy assumption.

By these arguments, all the isospectrality examples established in
\cite{sz1}-\cite{sz4} for a family $H^{(a,b)}_l$ defined by the same 
$a+b$ and $l$ can be reestablished almost in the same way. 
Note that such a family is defined
on the same $(X,Z)$-space and two members defined for  
$(a,b)$ resp. $(a^\prime ,b^\prime)$ are not locally isometric, unless 
$(a,b)=(a^\prime ,b^\prime)$ upto an order. 
The above intertwining operator
proving the spectral isotropy appears now in the following modified form
$
\Omega_{Qpq\bullet}:
\mathbf \Xi_{Qpq\bullet}\to
\mathbf \Xi^\prime_{Qpq\bullet}
$, 
that is, it corresponds one-pole functions having the same pole but which
are defined by the distinct complex structures $J_{K_u}$ resp. 
$J^\prime_{K_u}$ to each other. The precise correspondence is then 
\begin{eqnarray}
\mathcal{HF}_{Qpq\bullet}(\phi )=
\int_{\bullet} e^{\mathbf i\langle Z,K\rangle}
\phi (\mathbf x,K)\Pi_X^{(n)}(\Theta_{Q}^p\overline\Theta^q_{Q})(X,K_u)
dK_\bullet\to 
\\
\to \mathcal{HF}^\prime_{Qpq}(\phi )=
\int_{\bullet} e^{\mathbf i\langle Z,K\rangle}
\phi (\mathbf x,K)\Pi_X^{(n)}(\Theta^{\prime p}_{Q}
\overline\Theta^{\prime q}_{Q})(X,K_u)dK_{\bullet} ,
\nonumber 
\end{eqnarray}
which, by the same argument used for proving spectral isotropy,
intertwines both the Laplacian and the boundary conditions on 
any of the ball$\times$ball- resp. 
sphere$\times$ball-type domains.  

In order to establish the complete isospectrality,
pick the same system $\mathbf B$ of independent vectors 
for both of these manifolds and,
by implementing the obvious alterations in the previous formula, define
$
\Omega_{\mathbf Bp_iq_i\bullet}:
\mathbf \Xi_{\mathbf Bp_iq_i\bullet}\to
\mathbf \Xi^\prime_{\mathbf Bp_iq_i\bullet}.
$
This is a well defined operator between the complete $L^2$ function spaces
which follows from the theorem asserting that the
twisted Z-Fourier transforms are $L^2$ bijections mapping 
$\mathbf{P\Phi}_{\mathbf B}$ onto an everywhere dense subspace of
the complete straightly defined space 
$\overline{\mathbf{\Phi}}_{\mathbf B}$. It     
can be defined also by all those maps $\Omega_{Qpq\bullet}$ where 
pole $Q$ is in the real span of the vector system $\mathbf B$. Thus, by 
the above argument, operators $\Omega_{\mathbf Bp_iq_i\bullet}$ intertwine
the complete $L^2$ function spaces along with the boundary conditions.

Interestingly enough, 
the isospectrality can be establish also in a new way,
by using only the intertwining operators $\omega_{Q\tilde Qpq\bullet}$. 
Indeed, suppose that the elements, 
$\{\nu_{p,q,i}\}$, 
of the spectrum appear on a total one-pole space  
$
\mathbf \Xi_{Qpq}
$ with multiplicity, say $m_{pq,i}$. By the one-pole intertwining
operators this spectrum is uniquely determined and each eigenvalue
regarding the whole $L^2$-space must be listed on this list, furthermore, 
the multiplicity regarding the whole $L^2$-space is the multiple of these
one-pole multiplicities by 
the dimension of $\Xi_{\mathbf B pq}$. 
On the other hand, for  
$Q\in\mathbb R^{r(l)a}$, the isospectrality 
obviously follows from $\mathbf\Xi_{Qpq}=\mathbf\Xi^{\prime}_{Qpq}$, 
thus both the elements of the spectra and the regarding multiplicities
must be the same on these two manifolds.

This proof clearly demonstrates that how can the spectrum 
``ignore'' the isometries. It explains also the striking examples
where one of the members of the isospectrality family is a homogeneous
space while the others are locally inhomogeneous. It also 
demonstrates that spectral isotropy implies the isospectrality. Recall
that this isospectrality is established above
by an intertwining operator which most conspicuously 
exhibits the following so called C-symmetry principle of physics: 
``The laws remain the
same if, in a system of particles, some of them are exchanged for their
anti-particles." This intertwining operator really operates such that
some of the particles are exchanged for their anti-particles. 
Thus this proof demonstrates that spectral isotropy
implies the C-symmetry. Moreover, the isospectrality is the manifestation 
of the physical C-symmetry. This is a physical verification of the
isospectrality of the examples established 
in \cite{sz1}-\cite{sz4}. Actually
these examples provide a rigorous mathematical proof for the C-symmetry
which is not a theorem but one of the principles in physics. 
Let it be mentioned yet that the
isospectrality proof provided in this paper is completely new, where the 
Hankel transform appears in the very first time in these investigations.
(All the other proofs are established by 
different integral transformations.)
 
The isospectrality theorem naturally extends to 
the solvable extensions endowed with positive definite invariant 
metrics. 
Just functions $\phi(\mathbf x,K)$ should be exchanged for  
$\phi(\mathbf x,t,K)$ in the above formulas, that is, the 
intertwining is led back to the nilpotent group.
It is important to keep in mind that the metrics are positive definite
in the spectral investigations of the 
solvable extensions. The group
of isometries acting on the sphere$\times$sphere-type manifolds of 
$SH^{(a,b)}_3$, where $ab\not =0$, is 
$(\mathbf {Sp}(a)\times \mathbf{Sp}(b))\cdot SO(3)$,
while 
it is $\mathbf{Sp}(a+b)\cdot\mathbf{Sp}(1)$ 
on $SH^{a+b,0)}_3$. 
By these formulas, the
same statement proved for the nilpotent groups can be generalized 
to the solvable isospectrality family. 
These solvable examples provide also
new striking examples of isospectral 
metrics where one of them is homogeneous
while the other is locally inhomogeneous. 
By summing up, we have

\begin{theorem} 
Operators $\Omega_{Qpq}$ and  
$\omega_{Q\tilde Qpq}$ defined for combined spaces
intertwine the Laplacians, moreover,
they can be induced by
point transformations
of the form 
$(K_Q,id_Z)$ resp. $(O_{Q\tilde Q},id_Z)$,
where $id_Z$ is the identity map on the Z-space and the first ones
are orthogonal transformations on the X-space, transforming
subspace $S_Q$, spanned by $Q$ and all  $J_{Z_u}(Q)$, 
to $S^\prime_Q$ resp. $S_{\tilde Q}$. This part of the 
map is uniquely determined, whereas, between the complement
spaces it can be arbitrary orthogonal transformation. 

By this induced map interpretation, both the Dirichlet and Z-Neumann 
conditions are also intertwined by these operators. This statement
follows also from the fact that
the very same operators are defined by corresponding the Hankel
functions obtained by the Hankel decomposition of the above functions
to each other. That is, these maps intertwine also the corresponding Hankel
subspaces along with the exterior operator $\OE$ and the interior
strong force operator $\mathbf S$.    

So far the isospectrality is established for one-pole functions. For a 
global establishment consider a system $\mathbf B$ of $k/2$ independent
vectors described earlier on the X-space. Then the global operator
$\Omega_{\mathbf Bp_iq_i}: 
\mathcal{HF}_{\mathbf Bp_iq_i}(\phi )
\to \mathcal{HF}^\prime_{\mathbf Bp_iq_i}(\phi )$ can be defined also
by $Q$-pole functions satisfying $Q\in Span_{\mathbb R}(\mathbf B)$.
This proves that the  $\kappa_{\mathbf Bp_iq_i}$ defines a
global intertwining operator. 
\end{theorem}
\noindent{\bf Acknowledgement} I am indebted to the Max Planck Institute
for Mathematics in the Sciences, Leipzig, for the excellent working 
conditions provided for me during my visit in the 07-08 academic year.
My particular gratitude yields to Prof. E. Zeidler for the interesting
conversations.

\bibliographystyle{my-h-elsevier}

\end{document}